\documentclass[a4paper,11pt]{amsart}

\usepackage[dvipsnames]{xcolor} 
\usepackage{amsmath,amsthm,amssymb}
\usepackage{enumitem}
\usepackage{tikz}
\usepackage{hyperref}
\usepackage{amsfonts,marvosym,amscd}
\usepackage{mathtools}
\usepackage[margin=38mm]{geometry}
\usepackage{mathscinet}
\usepackage{wasysym}
\usepackage{stmaryrd}
\usepackage{graphicx}
\usepackage{psfrag}
\usepackage{datetime}
\usepackage{tikz-cd} 
\usepackage{adjustbox}
\usepackage{subcaption}
\usepackage{tabstackengine}
\setstackgap{L}{.75\baselineskip}

\usetikzlibrary{decorations.pathmorphing}
\usetikzlibrary{decorations.pathreplacing}
\usetikzlibrary{positioning}
\usetikzlibrary{calc}

\makeatletter
\tikzset{
    dot diameter/.store in=\dot@diameter,
    dot diameter=3pt,
    dot spacing/.store in=\dot@spacing,
    dot spacing=10pt,
    dots/.style={
        line width=\dot@diameter,
        line cap=round,
        dash pattern=on 0pt off \dot@spacing
    }
}
\makeatother

\newtheorem{theorem}{Theorem}[section]
\newtheorem{lemma}[theorem]{Lemma}
\newtheorem{proposition}[theorem]{Proposition}
\newtheorem{corollary}[theorem]{Corollary}

\newtheorem{conjecture}[theorem]{Conjecture}

\newtheorem{innercustomthm}{Theorem}
\newenvironment{customthm}[1]
  {\renewcommand\theinnercustomthm{#1}\innercustomthm}
  {\endinnercustomthm}

\theoremstyle{definition}
\newtheorem{definition}[theorem]{Definition}

\newtheorem{example}[theorem]{Example}

\theoremstyle{remark}
\newtheorem{remark}[theorem]{Remark}

\newcommand{\tcs}[1]{{\tabbedCenterstack{#1}}}
\newcommand{\tcss}[1]{{\scriptsize\tabbedCenterstack{#1}}}

\DeclareMathOperator{\Hom}{Hom}
\DeclareMathOperator{\im}{im}
\DeclareMathOperator{\Ext}{Ext}
\DeclareMathOperator{\End}{End}
\DeclareMathOperator{\Fac}{Fac}
\DeclareMathOperator{\add}{add}

\DeclareMathOperator{\modules}{mod}
\renewcommand{\mod}{\modules}
\DeclareMathOperator{\additive}{add}

\DeclareMathOperator{\ind}{ind}
\DeclareMathOperator{\Filt}{Filt}
\DeclareMathOperator{\proj}{proj}
\newcommand{\relproj}[1]{\mathcal{P}(#1)}
\DeclareMathOperator{\simp}{simp}
\newcommand{\relsimp}[1]{\simp(#1)}
\DeclareMathOperator{\brick}{brick}
\DeclareMathOperator{\thick}{thick}
\DeclareMathOperator{\soc}{soc}
\DeclareMathOperator{\rad}{rad}
\DeclareMathOperator{\coim}{coim}

\newcommand{\tors}{\mathsf{tors}\,}
\newcommand{\fftors}{\mathsf{ff}\text{-}\mathsf{tors}\,}
\newcommand{\twosilt}{2\text{-}\mathsf{silt}\,}
\newcommand{\sttilt}{\mathsf{s}\tau\text{-}\mathsf{tilt}\,}
\newcommand{\Tors}[1]{\mathrm{Tors}(#1)}
\newcommand{\Torf}[1]{\mathrm{Torf}(#1)}

\renewcommand{\emptyset}{\varnothing}

\newcommand{\iysh}{\langle 1 \rangle}

\newcommand{\kbproj}{K^{b}(\proj \Lambda)}
\newcommand{\twoterm}{K^{[-1,0]}(\proj \Lambda)}

\newcommand{\eor}[1]{{#1}^{\bot_{1}}}
\newcommand{\eol}[1]{\prescript{\bot_{1}}{}{#1}}
\newcommand{\hor}[1]{#1^{\bot_{0}}}
\newcommand{\hol}[1]{\prescript{\bot_{0}}{}{#1}}

\newcommand{\st}{\mid}

\newcommand{\summ}[1]{\mathsf{S}_{\mathrm{s}}(#1)}
\newcommand{\sumt}[1]{\mathsf{S}_{\tau}(#1)}
\newcommand{\exch}[1]{\mathsf{E}(#1)}
\newcommand{\bricks}[1]{\mathsf{B}(#1)}

\newcommand{\bleq}{\leqslant_{\mathsf{B}}}
\newcommand{\bl}{<_{\mathsf{B}}}

\newcommand{\sleq}{\leqslant_{\mathsf{S}}}

\newcommand{\dleq}{\leqslant_{\textup{\pentagon}}}
\newcommand{\dl}{<_{\textup{\pentagon}}}
\newcommand{\dlessdot}{\lessdot_{\textup{\pentagon}}}

\newcommand{\hleq}{\leqslant_{\mathsf{HN}}}
\newcommand{\hl}{<_{\mathsf{HN}}}
\newcommand{\hnf}[2]{\mathsf{SSF}_{#1}(#2)}
\newcommand{\hnfbig}[2]{\mathsf{SSF}_{#1}\Bigl(#2\Bigr)}
\newcommand{\bhnf}[2]{\mathsf{SF}_{#1}(#2)}
\newcommand{\bhnfbig}[2]{\mathsf{SF}_{#1}\Bigl(#2\Bigr)}

\makeatletter
\newcommand\@dotsep{4.5}
\def\@tocline#1#2#3#4#5#6#7{\relax
  \ifnum #1>\c@tocdepth 
  \else
    \par \addpenalty\@secpenalty\addvspace{#2}%
    \begingroup \hyphenpenalty\@M
    \@ifempty{#4}{%
      \@tempdima\csname r@tocindent\number#1\endcsname\relax
    }{%
      \@tempdima#4\relax
    }%
    \parindent\z@ \leftskip#3\relax \advance\leftskip\@tempdima\relax
    \rightskip\@pnumwidth plus1em \parfillskip-\@pnumwidth
    #5\leavevmode\hskip-\@tempdima{#6}\nobreak
    \leaders\hbox{$\m@th\mkern \@dotsep mu\hbox{.}\mkern \@dotsep mu$}\hfill
    \nobreak
    \hbox to\@pnumwidth{\@tocpagenum{\ifnum#1=1\fi#7}}\par
    \nobreak
    \endgroup
  \fi}
\AtBeginDocument{%
\expandafter\renewcommand\csname r@tocindent0\endcsname{0pt}
}
\def\l@subsection{\@tocline{2}{0pt}{2.5pc}{5pc}{}}
\makeatother

\usepackage[backend=biber,
	style=alphabetic,
	url=false,
	doi=false,
	isbn=false,
	maxbibnames=99]{biblatex}
\DeclareFieldFormat{title}{#1}
\renewbibmacro{in:}{}
\title[A structural view of maximal green sequences]{A structural view of\\maximal green sequences}
\author{Mikhail Gorsky}
\urladdr{https://sites.google.com/site/homepageofmikhailgorsky/}
\email{mikhail.gorskii@univie.ac.at}
\address{Faculty of Mathematics, Oskar-Morgenstern-Platz 1, 1090 Vienna, Austria}
\author{Nicholas J. Williams}
\urladdr{https://nchlswllms.github.io/}
\email{nicholas.williams@lancaster.ac.uk}
\address{Department of Mathematics and Statistics, Fylde College, Lancaster University, Lancaster, LA1 4YF, United Kingdom}


\begin{document}

\begin{abstract}
    We study the structure of the set of all maximal green sequences of a finite-dimensional algebra. There is a natural equivalence relation on this set, which we show can be interpreted in several different ways, underscoring its significance. There are three partial orders on the equivalence classes, analogous to the partial orders on silting complexes and generalising the higher Stasheff--Tamari orders on triangulations of three-dimensional cyclic polytopes. We conjecture that these partial orders are in fact equal, just as the orders in the silting case have the same Hasse diagram. This can be seen as a refined and more widely applicable version of the No-Gap Conjecture of Br\"ustle, Dupont, and Perotin. We prove our conjecture in the case of Nakayama algebras.
\end{abstract}

\maketitle

\tableofcontents

 \section{Introduction}

Maximal green sequences were introduced by Keller in \cite{kel-green}, but were already implicit in the physics literature in the context of BPS spectra of particles in string theory \cite{ccv}; see also \cite{gmn,accrv,xie}. The origin of maximal green sequences lies in the theory of cluster algebras, which were introduced by Fomin and Zelevinsky \cite{fz1}. Cluster algebras are commutative rings with distinguished generating sets known as `clusters' which are related by a process called `mutation'. Every cluster has a skew-symmetrisable matrix associated to it, with mutation both transforming the matrix and the variables of the cluster. If the matrix is in fact skew-symmetric, it gives a quiver. The clusters of a cluster algebra form the vertices of a graph whose edges connect clusters related by mutation; this graph is known as the `exchange graph'.

Subsequent to their introduction, deep connections were found between cluster algebras and the representation theory of finite-dimensional algebras \cite{mrz,ccs,bmrrt}. Here, the clusters of a cluster algebra correspond to objects in a certain category, thereby ``categorifying'' the cluster algebra. Some categorifications of cluster algebras --- in particular, those related to $\tau$-tilting theory \cite{air} --- give an orientation to the edges of the exchange graph, thereby producing a partial order. For a helpful survey of this phenomenon, see \cite{by}. A maximal green sequence is then simply a maximal chain of finite length in the resulting poset. The name comes from a combinatorial construction due to Keller \cite{kel-green}, in which some vertices of the quiver of a cluster are green and some are red, with a maximal green sequence given by a sequence of mutations at green vertices which turns the quiver from all green to all red.

The original motivation for introducing maximal green sequence comes from the theory of Donaldson--Thomas invariants, which, roughly speaking, count the aforementioned BPS states. A maximal green sequence for a quiver gives an explicit formula for the refined Donaldson--Thomas invariant associated to the quiver by Kontsevich and Soibelman \cite{ks_stability}. Consequently, maximal green sequences give quantum dilogarithm identities \cite{fv_aca,fk_qd,reineke_poisson}. The existence of a maximal green sequence also gives a formula for the twist automorphism of a cluster algebra \cite{gls_generic}, as well as guaranteeing the existence of a theta basis \cite{ghkk} and a generic basis \cite{qin_bases} in the upper cluster algebra. Maximal green sequences are intimately related to Bridgeland stability conditions on module categories \cite{bridgeland}, where a chamber of type I with finitely many isomorphism classes of stable objects gives a maximal green sequence \cite{bridgeland}. A useful survey of maximal green sequences is given in \cite{dem-kel}.

If one is interested in a mathematical object, then it is natural to study the structure of the set of all of them. An equivalence relation on the set of maximal green sequences of a finite-dimensional algebra was considered in \cite{njw-hst,njw-phd}, where it was shown that equivalence classes of maximal green sequences of the linearly oriented $A_n$ quiver were in bijection with triangulations of the three-dimensional cyclic polytope with $n + 3$ vertices. In the first main result of this paper, we show that the equivalence relation on maximal green sequences admits the six following interpretations. The fact that the equivalence relation detects several different aspects of maximal green sequences underscores its significance. For the detail on the terminology of this theorem, see Section~\ref{sect:back}.

\begin{customthm}{A}[Theorem~\ref{thm:mg_equiv}, Lemma~\ref{lem:same_bricks}, Example~\ref{ex:brick_counter}]\label{thm:intro:mg_equiv}
Let $\Lambda$ be a finite-dimensional algebra over a field $K$ with $\mathcal{G}$ and $\mathcal{G}'$ maximal green sequences of~$\Lambda$. Then the following are equivalent.
\begin{enumerate}[label=\textup{(}\arabic*\textup{)}]
    \item $\mathcal{G}$ and $\mathcal{G}'$ can be deformed into each other across squares.\label{op:intro_eq_thm:squares}
    \item $\mathcal{G}$ and $\mathcal{G}'$ have the same set of exchange pairs.
    \item $\mathcal{G}$ and $\mathcal{G}'$ have the same set of indecomposable direct summands of two-term silting complexes.
    \item $\mathcal{G}$ and $\mathcal{G}'$ have the same set of indecomposable direct summands of support $\tau$-tilting modules.
    \item For any $\Lambda$-module $M$, the set of semistable factors of $M$ is the same for the respective Harder--Narasimhan filtrations given by $\mathcal{G}$ and $\mathcal{G}'$.
    \item For any $\Lambda$-module $M$, the multiset of stable factors of $M$ is the same for the respective Harder--Narasimhan filtrations given by $\mathcal{G}$ and $\mathcal{G}'$.\label{op:intro_eq_thm:stable}
\end{enumerate}
Moreover, if $\mathcal{G}$ and $\mathcal{G}'$ satisfy any of the equivalent statements above, then $\mathcal{G}$ and $\mathcal{G}'$ have the same set of bricks. However, $\mathcal{G}$ and $\mathcal{G}'$ may have the same set of bricks without any of the above holding.
\end{customthm}

When $\Lambda$ is a Jacobian algebra and so has a corresponding cluster algebra, conditions \ref{op:intro_eq_thm:squares} to \ref{op:intro_eq_thm:stable} are satisfied if and only if the sets of cluster variables appearing in the clusters along the two maximal green sequences coincide.
In terms of quantum dilogarithm factorisations \cite{kel-green,reineke_poisson}, two maximal green sequences are equivalent if the factorisations are related by swapping commuting terms.

In some ways it is rather surprising that two maximal green sequences may have the same set of bricks while having different sets of $\tau$-rigid modules, since bricks are dual to indecomposable $\tau$-rigid modules \cite{dij}. This is related to the ``tropical duality'' that exists between $c$-vectors and $g$-vectors for cluster algebras \cite{nz_trop}. However, we show that two maximal green sequences with the same bricks are equivalent in the case of Nakayama algebras (Theorem~\ref{thm:nak_brick_equiv}). More generally, as we explain, having the same factors of Harder--Narasimhan filtrations can be considered an augmentation of the condition of having the same bricks. An intriguing implication of Theorem~\ref{thm:intro:mg_equiv} is that the set of indecomposable summands of support $\tau$-tilting modules of a maximal green sequence determines the semistable factors of every module; it would be interesting if there were a direct construction relating the two.

The equivalence relation on maximal green sequences then reveals even more structure, since there exist natural partial orders on the equivalence classes. In the linearly oriented type~$A$ case, these partial orders correspond to the higher Stasheff--Tamari orders on the triangulations of a three-dimensional cyclic polytope \cite{njw-hst,njw-phd}, a higher-dimensional version of the Tamari lattice \cite{kv-poly,er}. For the preprojective algebra of type~$A$, equivalence classes of maximal green sequences correspond to elements of the two-dimensional higher Bruhat order \cite{ms}, as we explore in a sequel paper \cite{mgsii}. Posets of maximal green sequences were also studied in \cite{gorsky_phd,gorsky_note} using the language of \cite{bkt}. For path algebras of~$ADE$ quivers, the covering relations in the partial order somewhat implicitly discussed there correspond to edge subdivisions of certain subword complexes. For certain Delzant polytopal realizations of those, the covering relations thus correspond to blow-ups of corresponding toric varieties. The order agrees with the first order defined below.

The work \cite{njw-hst,njw-phd} also implies that the partial orders on equivalence classes of maximal green sequences are analogous to the well-known partial orders on silting complexes \cite{ai,rs-simp}. In the silting case, there are two different partial orders, one defined locally via mutation and the other defined globally via inclusion of the associated aisles; these orders were shown to have the same Hasse diagram \cite{ai}. These correspond to the two different versions of the two-dimensional higher Stasheff--Tamari orders, which were shown to be equal in \cite{njw-equal}. Via the bijection between two-term silting complexes and functorially finite torsion classes, the second order corresponds to the inclusion order on torsion classes.  Following these cases, we conjecture that the three partial orders on equivalence classes of maximal green sequences should also coincide. As we explain (Remark~\ref{rmk:no_gap}), this conjecture can be seen as a refined and more widely applicable version of the No-Gap Conjecture. The three partial orders are as follows: one defined locally in terms of certain deformations, analogous to the irreducible mutation order on silting complexes; one globally defined in terms of inclusion of indecomposable summands of two-term silting complexes, analogous to the order on silting complexes via inclusion of aisles; and one extra global order in terms of Harder--Narasimhan filtrations. The third order can be seen as a refinement of a natural analogue of the inclusion order on torsion classes; the analogue itself can be formulated via inclusion of sets of bricks and is not compatible with the equivalence relation by the last paragraph of Theorem~\ref{thm:intro:mg_equiv}. We show the following result towards the conjectured equality of the partial orders. See Definition~\ref{def:iepd} for the definition of an increasing elementary polygonal deformation and Definition~\ref{def:partial_orders}\eqref{op:partial_orders:hn} for the definition of refinement of Harder--Narasimhan filtrations.

\begin{customthm}{B}[{Theorem~\ref{thm:first->second} and Theorem~\ref{thm:def->hn}}]\label{thm:intro:poset}
Let $\Lambda$ be a finite-dimen\-sional algebra over a field $K$ with $\mathcal{G}$ and $\mathcal{G}'$ maximal green sequences of~$\Lambda$. Furthermore, suppose that $[\mathcal{G}']$ is the result of a series of increasing elementary polygonal deformations of $[\mathcal{G}]$. Then we have the following.
\begin{enumerate}[label=\textup{(}\arabic*\textup{)}]
    \item Every indecomposable direct summand of a two-term silting complex of $\mathcal{G}'$ is a direct summand of a two-term silting complex of~$\mathcal{G}$.
    \item The Harder--Narasimhan filtrations induced by $\mathcal{G}'$ refine those induced by~$\mathcal{G}$.
\end{enumerate}
\end{customthm}

The partial orders can also be seen in the context of partial orders on chambers of stability conditions, although this is not the language we choose to use. There is a partial order on the chambers of King stability conditions \cite{king} given by those from $\tau$-tilting theory \cite{bst,asai_wall}, and a partial order on chambers of type II for Bridgeland stability conditions given by inclusion of aisles of $t$-structures \cite{bridgeland}. These partial orders have proven useful in studying the stability manifold \cite{psz,amy}. Roughly speaking, we study partial orders on equivalence classes of certain type I chambers. As an aside, it is curious to note that while all two-term silting complexes give chambers of King stability conditions \cite{bridge_scat,bst,asai_wall}, not all maximal green sequences give chambers of Bridgeland stability conditions, see \cite[Counterexample~7.16]{Qiu15} and
\cite[Theorem~L3]{ai2020}. The partial order given by deformations is very natural in terms of stability conditions, since increasing elementary polygonal deformations correspond to crossing walls of type I which decrease the number of stables.

For Nakayama algebras, or algebras with two simple modules up to isomorphism, we prove the converse implications from Theorem~\ref{thm:intro:poset}. Namely, we show that the three partial orders in fact coincide, along with an additional order given by inclusion of bricks. The latter is not well-defined for an arbitrary finite-dimensional algebra~$\Lambda$, since in general non-equivalent maximal green sequences may have the same bricks. 

\begin{customthm}{C}[{Theorem~\ref{thm:two_simples} and Corollary~\ref{cor:nak_orders}}]\label{thm:intro:nakayama}
Let $\Lambda$ be a finite-dimensional algebra over a field $K$ which is either a Nakayama algebra or an algebra with two isomorphism classes of simple modules. Further, let $\mathcal{G}$ and $\mathcal{G}'$ maximal green sequences of~$\Lambda$. Then the following are equivalent. 
\begin{enumerate}[label=\textup{(}\arabic*\textup{)}]
    \item $[\mathcal{G}']$ is the result of a series of increasing elementary polygonal deformations of~$[\mathcal{G}]$. 
    \item Every indecomposable direct summand of a support $\tau$-tilting module of $\mathcal{G}'$ is a direct summand of a support $\tau$-tilting module of~$\mathcal{G}$.\label{op:thm_nak_intro:tau}
    \item The Harder--Narasimhan filtrations induced by $\mathcal{G}'$ refine those induced by~$\mathcal{G}$.
    \item Every brick of $\mathcal{G}'$ is a brick of~$\mathcal{G}$.\label{op:thm_nak_intro:brick}
\end{enumerate}
\end{customthm}

\vspace{0.3cm}

\noindent {\bf Organisation of the paper.} The structure of this paper is as follows. We begin in Section~\ref{sect:back} by giving background to the paper. We introduce maximal green sequences in Section~\ref{sect:back:mgs}, and then give their description in terms of $\tau$-tilting theory in Section~\ref{sect:back:ttt}, followed by their description in terms of sequences of bricks in Section~\ref{sect:back:bricks}. We introduce equivalence relations on maximal green sequences in Section~\ref{sect:equiv}, and prove our first main result Theorem~\ref{thm:intro:mg_equiv} showing the different ways of characterising the equivalence relation. In Section~\ref{sect:partial_order}, we introduce three partial orders on equivalence classes of maximal green sequences and prove our second main result Theorem~\ref{thm:intro:poset}, showing how these are related. We then go on to consider the interaction between one of the partial orders and the exchange pairs of a maximal green sequence. We then discuss an interesting example of a poset of equivalence classes of maximal green sequences of an algebra related to a triangulation of the twice-punctured torus. Finally, in Section~\ref{sect:nak} we study the posets for Nakayama algebras in detail and prove Theorem~\ref{thm:intro:nakayama}.

\vspace{0.3cm}

\noindent {\bf Acknowledgements.} We would like to thank Aran Tattar for help with the proof of Lemma~\ref{lem:extension_of_bricks}, and Haruhisa Enomoto, Bernhard Keller, Hipolito Treffinger, Osamu Iyama, Aaron Chan, Daniel Labardini-Fragoso, and H\r{a}vard Terland for useful discussions. This work is part of a project that has received funding from the European Research Council (ERC) under the European Union’s Horizon 2020 research and innovation programme (grant agreement No.\ 101001159). Parts of this work were done during stays of MG at the University of Stuttgart, and he is very grateful to Steffen Koenig for the hospitality. NJW is currently supported by EPSRC grant EP/V050524/1, and part of the work on this paper was done while a JSPS short-term postdoctoral research fellow at the University of Tokyo.

\vspace{0.3cm}

\noindent {\bf Notation.} 
Throughout this paper, we let $\Lambda$ be a finite-dimensional algebra over a field $K$, with $\modules \Lambda$ the category of finitely generated right $\Lambda$-modules. In this paper we use the symbols `$\subset$' and `$\supset$' to denote strict inclusion of sets, that is, inclusion but not equality. This is more commonly denoted with the symbols `$\subsetneq$' and `$\supsetneq$' respectively.

\section{Background}\label{sect:back}

\subsection{Partially ordered sets}

Given a set $\mathsf{P}$, a \emph{partial order} on $\mathsf{P}$ is a relation $\mathsf{R} \subseteq \mathsf{P} \times \mathsf{P}$ which is reflexive, symmetric, and transitive. Partially ordered sets are referred to as \emph{posets}. We usually write partial orders with the symbol $\leqslant$, so that if $(x, y) \in \mathsf{R}$, where $\mathsf{R}$ is a partial order on a set $\mathsf{P}$, we write $x \leqslant y$. A~\emph{covering relation} in a partial order $\leqslant$ is a relation $x < z$ such that if $x \leqslant y \leqslant z$, then $y = x$ or $y = z$. One can also say that $z$ \emph{covers}~$x$. It is usual to write $x \lessdot z$ when $x < z$ is a covering relation. An \emph{interval} of a poset $\mathsf{P}$ is a subset of the form $\{y \in \mathsf{P} \st x \leqslant y \leqslant z\}$ for some $x, z \in \mathsf{P}$.

The \emph{Hasse diagram} of a partial order $\leqslant$ on a set $\mathsf{P}$ is the quiver with the elements of $\mathsf{P}$ as vertices, with arrows $z \to x$ whenever $x \lessdot z$ is a covering relation. In this paper, we illustrate posets using their Hasse diagrams. Recall that a Hasse diagram is \emph{$n$-regular} if every vertex is incident to precisely $n$ arrows.

\subsection{Maximal green sequences}\label{sect:back:mgs}

Maximal green sequences were introduced by Keller in the context of Donaldson--Thomas theory using a combinatorial definition in terms of quivers \cite{kel-green}. It follows from work of Nagao that this is equivalent to having a maximal chain of torsion classes \cite{nagao}. This is the first notion of a maximal green sequence that we will cover.

\subsubsection{Torsion classes}

Torsion pairs were introduced by Dickson to generalise the structure given by torsion and torsion-free abelian groups to arbitrary abelian categories \cite{dickson}. A \emph{torsion pair} is a pair of full subcategories $(\mathcal{T}, \mathcal{F})$ of $\modules \Lambda$ such that
\begin{enumerate}
    \item $\Hom_{\Lambda}(\mathcal{T}, \mathcal{F}) = 0$;
    \item if $\Hom_{\Lambda}(T, \mathcal{F}) = 0$, then $T \in \mathcal{T}$;
    \item if $\Hom_{\Lambda}(\mathcal{T}, F) = 0$, then $F \in \mathcal{F}$.
\end{enumerate}
Here $\mathcal{T}$ is called the \emph{torsion class} and $\mathcal{F}$ is called the \emph{torsion-free class}. More generally, a full subcategory $\mathcal{T}$ is called a torsion class if it is a torsion class in some torsion pair, and likewise for torsion-free classes. It is well-known that a full subcategory $\mathcal{T}$ of $\modules \Lambda$ is a torsion class if and only if it is closed under factor modules and extensions \cite[Theorem~2.3]{dickson}. Given a set $\mathcal{X}$ of $\Lambda$-modules, we write $\Tors{\mathcal{X}}$ for the smallest torsion class containing~$\mathcal{X}$ and $\Torf{\mathcal{X}}$ for the smallest torsion-free class containing~$\mathcal{X}$.

Given a torsion pair $(\mathcal{T}, \mathcal{F})$ in $\modules \Lambda$ and a $\Lambda$-module $M$, there is an exact sequence \[0 \to L \to M \to N \to 0\] such that $L \in \mathcal{T}$ and $N \in \mathcal{F}$, which is unique up to isomorphism. Here $L$ is called the \emph{torsion submodule} of $M$ and $N$ is called the \emph{torsion-free factor module}.

The torsion classes of $\modules \Lambda$ form a complete lattice under inclusion, denoted $\tors \Lambda$. We call the covering relations of this lattice \emph{minimal inclusions}. Hence, $\mathcal{T} \subset \mathcal{T}'$ is a minimal inclusion if and only if whenever we have $\mathcal{T} \subseteq \mathcal{T}'' \subseteq \mathcal{T}'$, we must either have $\mathcal{T}'' = \mathcal{T}$ or $\mathcal{T}'' = \mathcal{T}'$.

We will be particularly interested in the subposet $\fftors \Lambda$ of functorially finite torsion classes of $\Lambda$, where `functorially finite' is defined as follows. Given a subcategory $\mathcal{X} \subseteq \modules \Lambda$ and a map $f\colon X \rightarrow M$, where $X \in \mathcal{X}$ and $M \in \modules \Lambda$, we say that $f$ is a \emph{right $\mathcal{X}$-approximation} if for any $X' \in \mathcal{X}$, the sequence \[\mathrm{Hom}_{\Lambda}(X',X) \rightarrow \mathrm{Hom}_{\Lambda}(X',M) \rightarrow 0\] is exact, following \cite{as-preproj}. \emph{Left $\mathcal{X}$-approximations} are defined dually. The subcategory $\mathcal{X}$ is said to be \emph{contravariantly finite} if every $M \in \modules\Lambda$ admits a right $\mathcal{X}$-approximation, and \emph{covariantly finite} if every $M \in \modules \Lambda$ admits a left $\mathcal{X}$-approximation. If $\mathcal{X}$ is both contravariantly finite and covariantly finite, then $\mathcal{X}$ is \emph{functorially finite}.

Certain sorts of approximations are of particular note. A morphism $f \colon X \to Y$ is \emph{right minimal} if any morphism $g \colon X \to X$ such that $fg = f$ is an isomorphism. \emph{Left minimal} morphisms are defined dually. A right $\mathcal{X}$-approximation is a \emph{minimal right $\mathcal{X}$-approximation} if it is also right minimal, and \emph{minimal left $\mathcal{X}$-approximations} are defined analogously.

\subsubsection{First notion of maximal green sequence}

A \emph{maximal green sequence} is a maximal chain in $\tors \Lambda$ of finite length. More explicitly, a maximal green sequence is a chain of minimal inclusions of torsion classes \[\modules \Lambda = \mathcal{T}_0 \supset \mathcal{T}_1 \supset \dots \supset \mathcal{T}_{r-1} \supset \mathcal{T}_r = \{0\}.\] We shall see two further definitions of maxmial green sequence. We shall regard the three notions as being cryptomorphic to each other.

We note at this point that a result of Demonet, Iyama, and Jasso \cite[Theorem~3.1]{dij} implies that every torsion class in a finite maximal chain is functorially finite, so that maximal green sequences are in fact finite maximal chains in $\fftors \Lambda$.

\subsection{Relative projectives in torsion classes: \texorpdfstring{$\tau$}{tau}-tilting}\label{sect:back:ttt}

Our second notion of maximal green sequences will operate in terms of the relative projectives of the torsion classes in the maximal chain. Relative projectives in torsion classes were studied by Adachi, Iyama, and Reiten in \cite{air} in terms of what is called `$\tau$-tilting theory'.

\subsubsection{Relative projectives in torsion classes}

Given a torsion class $\mathcal{T} \in \tors \Lambda$, a module $X \in \mathcal{T}$ is a \emph{relative projective} if $\Ext_\Lambda^1(X, M) = 0$ for all $M \in \mathcal{T}$. We write $\relproj{\mathcal{T}}$ for the direct sum of one copy of each indecomposable relative projective in $\mathcal{T}$, up to isomorphism.

\subsubsection{Support $\tau$-tilting pairs}

A $\Lambda$-module $M$ is called \emph{$\tau$-rigid} if $\Hom_{\Lambda}(M,$ $\tau M) = 0$, where $\tau$ is the Auslander--Reiten translate. A pair $(M, P)$ of $\Lambda$-modules where $P$ is projective is called \emph{$\tau$-rigid} if $M$ is $\tau$-rigid and $\Hom_{\Lambda}(P, M) = 0$. A $\tau$-rigid pair $(M, P)$ is called \emph{support $\tau$-tilting} if $|M| + |P| = |\Lambda|$, where $|X|$ denotes the number of non-isomorphic indecomposable direct summands of $X$. In this case $M$ is called a \emph{support $\tau$-tilting module}. We write $\sttilt \Lambda$ for the set of isomorphism-class representatives of  basic support $\tau$-tilting modules over $\Lambda$.

\begin{theorem}[{\cite[Theorem~2.7]{air}}]\label{thm:air_relproj}
There is a bijection
\begin{align*}
    \fftors \Lambda &\longleftrightarrow \sttilt \Lambda, \\
    \mathcal{T} &\longmapsto \relproj{\mathcal{T}}, \\
    \Fac M &\longmapsfrom M.
\end{align*}
\end{theorem}

Here \[\Fac M := \{X \in \modules \Lambda \st \text{ there is an epimorphism } M^{\oplus m} \twoheadrightarrow X \text{ for some } m\}.\]

\subsubsection{Two-term silting}

An equivalent framework to support $\tau$-tilting modules is given by two-term silting complexes. We will often prefer to work with these objects instead, for reasons that we will explain.

We denote by $\kbproj$ the homotopy category of bounded complexes of projective right $\Lambda$-modules. We will usually consider $\twoterm$, the subcategory of $\kbproj$ consisting of \emph{two-term complexes}, that is, complexes concentrated in degrees $-1$ and $0$: \[P^{-1} \to P^{0}.\]

An object $T$ of $\kbproj$ is called \emph{pre-silting} if $\Hom_{\kbproj}(T, T[i]) = 0$ for all $i > 0$. A pre-silting complex $T$ is \emph{silting} if, additionally, $\thick T = \kbproj$. Here $\thick T$ denotes the smallest full subcategory of $\kbproj$ which contains $T$ and is closed under cones, $[\pm 1]$, direct summands, and isomorphisms. For a two-term complex $T$ to be pre-silting, it suffices that $\Hom_{\kbproj}(T, T[1]) = 0$. Moreover, for a pre-silting two-term complex $T$ to be silting, it suffices that $|T| = |\Lambda|$ by \cite[Proposition~3.3(b)]{air}. We write $\twosilt \Lambda$ for the set of isomorphism-class representatives of basic two-term silting complexes of $\Lambda$.

\begin{theorem}[{\cite[Theorem~3.2]{air}}]
There is a bijection
\begin{align*}
    \twosilt \Lambda &\longleftrightarrow \sttilt \Lambda, \\
    T &\longmapsto H^{0}(T).
\end{align*}
\end{theorem}

Hence, support $\tau$-tilting and two-term silting are essentially equivalent. Functorially finite torsion classes of $\Lambda$ are therefore also in bijection with two-term silting complexes over $\Lambda$. The advantage of support $\tau$-tilting is that it is easier to describe the relation with torsion classes. The advantage of two-term silting is that it is easier to talk about mutation, which is why we usually work in this framework.

\subsubsection{Mutation}

Given a silting complex $T = E \oplus X$ in $\kbproj$ where $X$ is indecomposable, let \[X \xrightarrow{f} E' \xrightarrow{g} Y \to X[1]\] be a triangle in $\kbproj$ such that $f$ is a minimal left $\additive E$-approximation of $X$. This triangle is known as the \emph{exchange triangle}. Then, by \cite[Theorem~2.35]{ai}, $Y$ is indecomposable with $g$ a minimal right $\additive E$-approximation of $Y$ and, by \cite[Theorem~2.31]{ai}, $T' = E \oplus Y$ is a silting complex. In this situation, we say that $T'$ is a \emph{green mutation} (or \emph{left mutation}) of $T$ and $T$ is a \emph{red mutation} of $T'$ (or \emph{right mutation}).
The opposite convention for green and red mutations is used by some authors. Such choice would not make any difference to our considerations of maximal chains in the lattice of two-term silting complexes, since the chain remains the same whichever direction one traverses it in.
We call $(X, Y)$ the \emph{exchange pair} of the mutation.

We say that a pair of two-term silting complexes $T, T' \in \twoterm$ are \emph{mutations} of each other if and only if $T = E \oplus X$ and $T' = E \oplus Y$ where $X$ and $Y$ are indecomposable. By \cite[Corollary~3.8(b)]{air}, we have that $T$ and $T'$ are mutations of each other if and only if either $T'$ is a green mutation of~$T$, or $T$ is a green mutation of~$T'$. The set of basic two-term silting complexes of $\Lambda$ forms a poset denoted $\twosilt \Lambda$ where the covering relations are that $T' \lessdot T$ if and only if $T'$ is a green mutation of $T$. The partial order itself is then the transitive-reflexive closure of these covering relations.

\begin{theorem}[{\cite[Corollary~2.34, Corollary~3.9]{air}}]\label{thm:air:silt_tors}
The bijection between $\twosilt \Lambda$ and $\fftors \Lambda$ induces an isomorphism between the Hasse diagrams of these posets.

In particular, if functorially finite torsion classes $\mathcal{T}$ and $\mathcal{T}'$ correspond to two-term silting complexes $T$ and $T'$ respectively, then there is a minimal inclusion $\mathcal{T} \supset \mathcal{T}'$ if and only if $T'$ is a green mutation of $T$.
\end{theorem}

\subsubsection{Second notion of maximal green sequence}\label{sect:back:silt_notion}

We can use Theorem~\ref{thm:air:silt_tors} to obtain the second notion of maximal green sequence. Since a maximal green sequence is a maximal chain of minimal inclusions in $\fftors \Lambda$, by Theorem~\ref{thm:air:silt_tors}, we have that a maximal green sequence is a maximal sequence of green mutations of two-term silting complexes. More explicitly, a maximal green sequence is a sequence of two-term silting complexes $\Lambda = T_0, T_1, \dots, T_r = \Lambda[1]$ such that for each $i \in \{1, \dots, r\}$, we have that $T_i$ is a green mutation of $T_{i - 1}$. This was first observed by Br\"ustle, Smith, and Treffinger in terms of support $\tau$-tilting pairs \cite[Proposition~4.9]{bst}.

Such a maximal green sequence can be specified by labelling each of the summands of $\Lambda$ from $1$ to $n$ where $n = |\Lambda|$ and then giving a list of numbers specifying the sequence of summands to be mutated. (When summand $i$ is mutated, the summand that replaces it is then also labelled $i$.) In terms of the original notion of maximal green sequence from \cite{kel-green} using quiver mutation, this can be seen as a sequence of vertices of the quiver to mutate. We will use this perspective in a couple of instances.

\subsection{Relative simples in torsion classes: bricks}\label{sect:back:bricks}

We now give background leading up to the third notion of maximal green sequences, which comes from looking at the relative simples in torsion classes.

\subsubsection{Relative simples in torsion classes}

Relatively simple objects in exact categories were first studied in \cite{bg_hall} and were considered in the specific case of torsion-free classes in \cite{enomoto_bruhat}. A $\Lambda$-module $B$ in a torsion class $\mathcal{T}$ is a \emph{relative simple} if there is no short exact sequence \[0 \to A \to B \to C \to 0\] such that $A$ and $C$ are both non-zero modules in $\mathcal{T}$. Since torsion classes are closed under factor modules, it is in fact necessary and sufficient for $B$ to have no proper submodules which are in $\mathcal{T}$. We write $\relsimp{\mathcal{T}}$ for the set of isomorphism-class representatives of relative simples of $\mathcal{T}$.

\subsubsection{Brick labelling}\label{sect:back:brick_label}

In a somewhat, but not entirely, analogous way to how relative projectives of torsion classes correspond to support $\tau$-tilting modules, relative simples of torsion classes correspond to certain modules known as `bricks'.

An object $B$ of $\modules \Lambda$ is a \emph{brick} if $\End_{\Lambda}B$ is a division ring. Equivalently, $B$ is a brick if every non-zero endomorphism of $B$ is an isomorphism. This is the case if and only if $B$ has no proper factor module which is isomorphic to a proper submodule.

\begin{theorem}[{\cite{bcz,dirrt}}]
We have that $\mathcal{T} \supseteq \mathcal{U}$ is a minimal inclusion if and only if $\mathcal{T} \cap \hor{\mathcal{U}} = \Filt (B)$ for a brick $B$. Moreover, this brick $B$ is unique up to isomorphism.
\end{theorem}

Here $\Filt$ and $\hor{\mathcal{U}}$ are defined as follows. We have
\[\hor{\mathcal{U}} := \{M \in \modules \Lambda \st \Hom_{\Lambda}(U, M) = 0 \text{ for all } U \in \mathcal{U}\}.\] Given a full subcategory $\mathcal{C}$ of $\modules \Lambda$, $\Filt(\mathcal{C})$ is the full subcategory of $\modules \Lambda$ consisting of modules $M$ with a finite filtration \[M = M_0 \supset M_1 \supset \dots \supset M_{l - 1} \supset M_l = 0\] such that $M_{i - 1}/M_{i} \in \add \mathcal{C}$ for all $1 \leqslant i \leqslant l$. Here $\add \mathcal{C}$ is the full subcategory of $\modules \Lambda$ consisting of direct summands of finite direct sums of objects in $\mathcal{C}$.
It is known that $\Filt (B)$ for a brick $B$ is a \emph{wide} subcategory of $\mod \Lambda$ \cite[1.2]{ringel_rksb}, meaning that it is closed under extensions, kernels and cokernels. It is therefore also closed under images.

In this way, the covering relations of $\tors \Lambda$ can be labelled by bricks. We often write the brick labels of the inclusions by \[\mathcal{T} \overset{B}{\supset} \mathcal{U}.\] For two torsion classes $\mathcal{T} \supseteq \mathcal{U}$, we denote $[\mathcal{T}, \mathcal{U}] := \mathcal{T} \cap \hor{\mathcal{U}}$. The relation between brick labels and intervals in the lattice of torsion classes extends beyond intervals given by covering relations.

\begin{theorem}[{\cite[Theorem~6.8]{tattar_qa}, \cite[Theorem~3.5]{enomoto_bruhat}}]\label{thm:filt_intervals}
Let \[\mathcal{T}_i \overset{B_{i + 1}}{\supset} \dots \overset{B_j}{\supset} \mathcal{T}_j\] be a chain of minimal inclusions in $\tors \Lambda$ with brick labels. Then \[[\mathcal{T}_i, \mathcal{T}_{j}] = \Filt(B_{i + 1}, B_{i + 2}, \dots, B_j).\]
\end{theorem}

The following proposition is a slight generalisation of \cite[Proposition~3.8]{enomoto_bruhat}.

\begin{proposition}\label{prop:simples_in_intervals}
Given torsion classes $\mathcal{T}_i \supseteq \mathcal{T}_j$ in $\modules \Lambda$, every relatively simple object in $[\mathcal{T}_i, \mathcal{T}_j]$ must occur as a brick label in every finite maximal chain in the interval of $\tors \Lambda$ between $\mathcal{T}_i$ and $\mathcal{T}_j$.
\end{proposition}
\begin{proof}
To see this, take a relatively simple object $S$ in $[\mathcal{T}_i, \mathcal{T}_j]$ and suppose that there is a maximal chain connecting $\mathcal{T}_i$ and $\mathcal{T}_j$ with brick labels $B_{i + 1}, B_{i + 2}, \dots, B_j$. Then, since $[\mathcal{T}_i, \mathcal{T}_j] = \Filt(B_{i + 1}, B_{i + 2}, \dots, B_j)$, we must have that $S$ has a filtration with factors in $\{B_{i + 1}, B_{i + 2}, \dots, B_j\}$. But, since $S$ is relatively simple in $[\mathcal{T}_{i}, \mathcal{T}_{j}]$, this filtration can only have one factor, and so we must have that $S = B_k$ for some $k$, as desired.
\end{proof}

Relatively simple objects in $[\mathcal{T}_i, \mathcal{T}_j]$ were studied in \cite{asai_pfeifer} in the case where this category is abelian. Finally, the relationship between the relative simples of torsion classes in a maximal green sequence and the brick labels is as follows.

\begin{theorem}[{\cite[Proposition~3.8]{enomoto_bruhat}}]\label{thm:enomoto_relsimp}
Let $\mathcal{G}$ be a maximal green sequence \[\modules \Lambda = \mathcal{T}_0 \overset{B_1}{\supset} \mathcal{T}_1 \overset{B_2}{\supset} \dots \overset{B_{r - 1}}{\supset} \mathcal{T}_{r-1} \overset{B_{r}}{\supset} \mathcal{T}_r = \{0\}.\] Then \[\bigcup_{i = 0}^{r}\relsimp{\mathcal{T}_i} = \{B_1, B_2, \dots, B_r\}.\]
\end{theorem}

More precisely, \cite[Proposition~3.8]{enomoto_bruhat} proves the inclusion 
\[\bigcup_{i = 0}^{r}\relsimp{\mathcal{T}_i} \subseteq \{B_1, B_2, \dots, B_r\},\]
and the converse inclusion follows from the fact that $B_i$ is a relative simple in $\mathcal{T}_{i-1}$, for $1 \leq i \leq r$.

\subsubsection{Third notion of maximal green sequence}

We can now give the third notion of maximal green sequences. In the appendix to \cite{dem-kel}, Demonet shows that (not necessarily finite) maximal chains of torsion classes may be characterised in terms of bricks, generalising \cite[Theorem~1.1]{igusa_mgs}. Another related result to this is \cite[Theorem~5.3]{treff_hn}.

A \emph{backwards $\Hom$-orthogonal sequence of bricks} is a sequence of bricks \[B_1, B_2, \dots, B_r\] such that if $i < j$ then $\Hom_{\Lambda}(B_j, B_i) = 0$. A backwards $\Hom$-orthogonal sequence of bricks is \emph{maximal} if one cannot insert a brick at any point in the sequence without losing the backwards $\Hom$-orthogonality property. By \cite[Theorem~A.3]{dem-kel}, there is a bijection between maximal green sequences and maximal backwards Hom-orthogonal sequences of bricks. Given a maximal green sequence \[\modules \Lambda = \mathcal{T}_0 \supset \mathcal{T}_1 \supset \dots \supset \mathcal{T}_r = \{0\},\] one obtains the maximal backwards $\Hom$-orthogonal sequence of bricks by taking $B_i$ as the brick label of the minimal inclusion $\mathcal{T}_{i - 1} \supset \mathcal{T}_{i}$. Conversely, given a maximal backwards $\Hom$-orthogonal sequence of bricks \[B_1, B_2, \dots, B_r,\] one obtains the corresponding maximal green sequence by taking $\mathcal{T}_{i}$ to be the smallest torsion class $\Tors{B_{i + 1}, \dots, B_r}$ containing $B_{i + 1}, \dots, B_r$.

\subsubsection{Harder--Narasimhan filtrations}

It was shown in \cite[Theorem~3.8]{igusa_lsc} that maximal green sequences induce so-called ``Harder--Narasimhan filtrations''. This was then shown for arbitrary chains of torsion classes in \cite[Theorem~2.10]{treff_hn}, see also  \cite[Section 3.3]{bkt}. Such filtrations were originally studied in \cite{hn} and are well known from their role in the theory of stability conditions, for example \cite{rudakov,bridgeland}. Let \[B_1, B_2, \dots, B_r\] be a maximal green sequence $\mathcal{G}$ of a finite-dimensional algebra $\Lambda$ given as a maximal backwards $\Hom$-orthogonal sequence of bricks. Then, every non-zero $\Lambda$-module $M$ has a unique filtration \[M = M_0 \supset M_1 \supset \dots \supset M_{l - 1} \supset M_l = 0\] such that $F_{j} := M_{j - 1 }/M_{j} \in  \Filt(B_{i_{j}})$ for some $i_{j}$, with \[i_{1} < i_{2} < \dots < i_{l}.\] This filtration is called the \emph{Harder--Narasimhan filtration} (HN filtration) or the \emph{$\mathcal{G}$-Harder--Narasimhan filtration}. This is precisely the filtration of $M$ by the torsion submodules associated to it by the torsion classes in $\mathcal{G}$, with duplicated torsion submodules removed. We call the $F_{j}$ the \emph{semistable factors} here, or the \emph{$\mathcal{G}$-semistable factors}. We write \[\hnf{\mathcal{G}}{M} = \{F_j\}_{1 \leqslant j \leqslant l}\] for the set of $\mathcal{G}$-semistable factors of $M$. The existence of Harder--Narasimhan filtrations can be conceived of as a factorisation, in some sense, of the category $\modules \Lambda$ into the smaller abelian categories $\Filt(B_i)$, which only have one simple object each.

Since the module $F_j$ lies in $\Filt(B_{i_j})$, it has a filtration where all of the factors are isomorphic to $B_{i_j}$. (In terms of stability conditions, this is the Jordan--H\"older filtration of a semistable module in terms of stable modules \cite[Theorem~3]{rudakov}.) We write $\bhnf{\mathcal{G}}{F_j}$ for the multiset of $B_{i_j}$ factors in this filtration of $F_j$. That is, $\bhnf{\mathcal{G}}{F_j}$ is $s$ copies of $B_{i_j}$, where $s$ is the number of factors in this filtration of~$F_j$. We furthermore write \[\bhnf{\mathcal{G}}{M} = \bigsqcup_{i = 1}^{l} \bhnf{\mathcal{G}}{F_j}.\] We refer to $\bhnf{\mathcal{G}}{M}$ as the multiset of \emph{stable factors} or \emph{$\mathcal{G}$-stable factors} of~$M$.

\section{Equivalence relations on maximal green sequences}\label{sect:equiv}

We now begin the study of the structure of the set of all maximal green sequences of $\Lambda$. In particular, there is a natural equivalence relation on this set, which essentially holds when two maximal green sequences differ by swapping adjacent commuting mutations. 
We give six different criteria for a pair of maximal green sequences to be equivalent. We also show that two non-equivalent maximal green sequences can have the same sets of bricks. Hence, having the same set of bricks is not a sufficient criterion for two maximal green sequences to be equivalent.

\subsection{Preliminary lemmas}

Before we start characterising equivalent maximal green sequences, we must prove some preliminary lemmas. In the following lemma, we collect some useful facts. Here, for $T \in \twoterm$, we denote
\begin{align*}
\eor{T} &= \{X \in \twoterm \mid \Hom_{\kbproj}(T, X[1]) = 0\}, \\
\eol{T} &= \{X \in \twoterm \mid \Hom_{\kbproj}(X, T[1]) = 0\}.
\end{align*}

\begin{lemma}\label{lem:mut_dir}
Let $T = U \oplus X$ and $T' = U \oplus Y$ be two two-term silting complexes in $\twoterm$ such that $T'$ is a green mutation of $T$. Then:
\begin{enumerate}[label=\textup{(}\arabic*\textup{)}]
\item $\eol{T} \subset \eol{T'}$; $\eor{T} \supset \eor{T'}$;
\label{op:silt1}
\item $X \in \eol{T'}$, $Y \in \eor{T}$;\label{op:silt2}
\item $X \notin \eor{T'}$, $Y \notin \eol{T}$.\label{op:silt3}
\end{enumerate} 
\end{lemma}
\begin{proof}
\ref{op:silt1} follows from \cite[Theorem~2.35]{ai}. \ref{op:silt2} then follows immediately from this, since clearly $X \in \eol{T}$ and $Y \in \eor{T'}$. \ref{op:silt3} then follows, since, by assumption, $T' \oplus X$ and $T \oplus Y$ cannot be silting complexes.
\end{proof}

By iterating Lemma~\ref{lem:mut_dir}, we obtain the following.

\begin{corollary}\label{cor:mut_dir}
Let $\mathcal{G}$ be a maximal green sequence of $\Lambda$ containing two-term silting complexes $T$ and $T'$ where $T'$ occurs after $T$. Furthermore, let $X$ be an indecomposable summand of $T$ which is not a summand of $T'$ and let $Y$ be an indecomposable summand of $T'$ which is not a summand of $T$. Then:
\begin{enumerate}[label=\textup{(}\arabic*\textup{)}]
\item $\eol{T} \subset \eol{T'}$; $\eor{T} \supset \eor{T'}$;\label{op:orthogonals}
\item $X \in \eol{T'}$, $Y \in \eor{T}$;\label{op:in}
\item $X \notin \eor{T'}$, $Y \notin \eol{T}$.\label{op:notin}
\end{enumerate}
\end{corollary}

We now introduce the following notation.

\begin{definition}
Let $\mathcal{G}$ be a maximal green sequence of $\Lambda$. We consider $\mathcal{G}$ to be a sequence $T_0, T_1, \dots, T_r$ of green mutations of two-term silting complexes with $T_0  = \Lambda$ and $T_r = \Lambda[1]$ in $\twoterm$.
\begin{enumerate}
    \item We denote by \[\summ{\mathcal{G}} := \bigcup_{i = 0}^{r}\{\text{Indecomposable summands of } T_i\}/\cong\] the set of isomorphism classes of indecomposable complexes which occur as direct summands of two-term silting complexes in $\mathcal{G}$.
    \item We denote by \[\exch{\mathcal{G}} := \bigcup_{i = 1}^{r} \left\{(X, Y)\st T_{i - 1} = E \oplus X,\, T_i = E \oplus Y\right\}\] the set of exchange pairs of $\mathcal{G}$.
\end{enumerate}
Furthermore, now consider $\mathcal{G}$ to be a sequence of support $\tau$-tilting modules $M_0, M_1, \dots$, $M_r$ corresponding to the two-term silting complexes $T_i$.
\begin{enumerate}[resume]
    \item We denote by \[\sumt{\mathcal{G}} := \bigcup_{i = 0}^{r}\{\text{Indecomposable summands of } M_i\}/\cong\] the set of isomorphism classes of indecomposable modules which occur as direct summands of support $\tau$-tilting modules in $\mathcal{G}$.
\end{enumerate}
Finally, consider $\mathcal{G}$ to be a maximal backwards $\Hom$-ortho\-gonal sequence of bricks $B_1, B_2, \dots, B_r$.
\begin{enumerate}[resume]
    \item We denote the set of bricks in the sequence by \[\bricks{\mathcal{G}} := \{B_1, B_2, \dots, B_r\}.\]
\end{enumerate}
\end{definition}

The following fact is well-known, but we do not know whether a proof has appeared in the literature. We use $\ind \twoterm$ to denote a set of representatives of the isomorphism classes of indecomposable complexes in $\twoterm$.

\begin{lemma}\label{lem:at_most_once}
Let $A \in \ind \twoterm$ and let $\mathcal{G}$ be a maximal green sequence of $\Lambda$. Then there is at most one exchange pair $(X, Y) \in \exch{\mathcal{G}}$ such that $X \cong A$ and at most one exchange pair $(X, Y) \in \exch{\mathcal{G}}$ such that $Y \cong A$. Moreover, the exchange pair $(X, A)$ must occur before the exchange pair $(A, Y)$ in $\mathcal{G}$.
\end{lemma}
\begin{proof}
Suppose for contradiction that $\mathcal{G}$ possesses two exchange pairs $(A, X_{1})$ and $(A, X_{2})$ which respectively correspond to green mutations from $T_{1}$ to $T'_{1}$ and from $T_{2}$ to $T'_{2}$. Suppose without loss of generality that $T_{1}$ occurs before $T_{2}$ in~$\mathcal{G}$. Then, by Corollary~\ref{cor:mut_dir}, we obtain that $A \notin \eor{T'_{1}} \supseteq \eor{T_{2}} \ni A$, which is a contradiction. The case where $A$ is the second half of the exchange pair is similar.

For the final statement it suffices to observe that if $(A, Y)$ occurs without $(X, A)$ before it, then $A$ must be a projective, and if $(X, A)$ occurs without $(A, Y)$ after it, then $A$ must be a shifted projective. But both cannot simultaneously be true.
\end{proof}

We finish this subsection by proving the following useful result on exchange pairs.

\begin{lemma}\label{lem:exch_crit}
Exchange pairs have the following properties.
\begin{enumerate}[label=\textup{(}\arabic*\textup{)}]
    \item Suppose that $(X, Y)$ is an exchange pair for a green mutation from $T$ to $T'$ in a maximal green sequence $\mathcal{G}$. Then $X$ is the unique indecomposable in $\summ{\mathcal{G}} \cap \eor{T}$ such that \[\Hom_{\kbproj}(Y, X[1]) \neq 0.\]\label{op:exch_crit:use_exch}
    \item Conversely, let $T$ be a two-term silting complex in a maximal green sequence $\mathcal{G}$. Suppose that, for  $Y \in \summ{\mathcal{G}} \cap \eor{T}$, there exists an indecomposable $X \in \summ{\mathcal{G}} \cap \eor{T}$, unique up to isomorphism, such that \[\Hom_{\kbproj}(Y, X[1]) \neq 0.\] Then $(X, Y) \in \exch{\mathcal{G}}$.\label{op:exch_crit:get_exch}
\end{enumerate}
\end{lemma}
\begin{proof}
\begin{enumerate}[wide]
    \item Suppose that there exists an indecomposable two-term complex $Z \in \summ{\mathcal{G}} \cap \eor{T}$ such that $\Hom_{\kbproj}(Y, Z[1]) \neq 0$. Then $Z \notin \eor{T'}$, since $Y$ is a summand of~$T'$. Hence, by Corollary~\ref{cor:mut_dir}, we cannot have that $Z$ is a summand of $T'$ or any silting complex which occurs later in $\mathcal{G}$ than $T'$. If $Z$ leaves $\mathcal{G}$ before $T$, then $Z \notin \eor{T}$, which is contrary to our assumption. We conclude that $Z$ must leave $\mathcal{G}$ between $T$ and $T'$, and so $Z \cong X$.
    
    \item Now suppose that $(X, Y)$ is such that $Y \in \summ{\mathcal{G}} \cap \eor{T}$ and $X$ is the unique indecomposable in $\summ{\mathcal{G}} \cap \eor{T}$ such that $\Hom_{\kbproj}(Y, X[1]) \neq 0$. The fact that $X \in \eor{T}$ means that we cannot have $Y$ as a summand of $T$. Therefore, $Y$ cannot occur in $\mathcal{G}$ before $T$ by Corollary~\ref{cor:mut_dir}\ref{op:notin}, otherwise $Y \notin \eor{T}$. Thus, $Y$ occurs in $\mathcal{G}$ after $T$ as the second half of an exchange pair $(Z, Y)$. But then we must have that $Z \in \summ{\mathcal{G}} \cap \eor{T}$ by Corollary~\ref{cor:mut_dir}\ref{op:in} and that $\Hom_{\kbproj}(Y, Z[1]) \neq 0$. We obtain that $Z \cong X$, since $X$ is the unique such indecomposable up to isomorphism.
\end{enumerate}
\end{proof}

\subsection{Polygons} \label{ssec:polygons}

One of the equivalence relations we introduce on maximal green sequences corresponds to deformations across squares in $\twosilt \Lambda$. Squares are a special case of the larger class of ``polygons'' in $\twosilt \Lambda$, so we introduce all polygons at this juncture. Several authors have studied the notion of a polygon in the poset of torsion classes or two-term silting objects. A lattice-theoretic notion is used by Reading \cite{reading_regions}, and Garver and McConville \cite{g-mc}. Our notion is instead based on that of Hermes and Igusa \cite{hi-no-gap}.

\begin{definition} \label{def:polygons}
A \emph{polygon} in the poset $\twosilt \Lambda$ is a subposet consisting of all two-term silting complexes possessing some presilting complex $E$ as a direct summand, where $|E| = |\Lambda| - 2$. A polygon in $\fftors \Lambda$ is the image of a polygon in $\twosilt \Lambda$ under the bijection between the two posets.
\end{definition}

\begin{proposition}
A polygon in the poset $\twosilt \Lambda$ falls under one of the four different types shown in Figure~\ref{fig:poly}.
\end{proposition}
\begin{proof}
The two-regularity of the Hasse diagram of the polygon follows from the fact that all but two summands are fixed. The existence of a unique maximum and minimum follows from the existence of the Bongartz and co-Bongartz completions \cite[Theorem~2.10]{air}. The cases displayed are then clearly exhaustive, which are as follows. 
\begin{enumerate}[label=(\alph*)]
    \item Square: the two paths from the maximum to the minimum are both of length two.
    \item Oriented polygon: one path from the maximum to the minimum is of length two and the other is finite of length greater than two.
    \item Unoriented polygon: both paths from the maximum to the minimum are finite of length greater than two.
    \item Infinite polygon: the polygon contains infinitely many elements. Equivalently, there is at most one path of finite length from the maximum to the minimum.
\end{enumerate}
\end{proof}

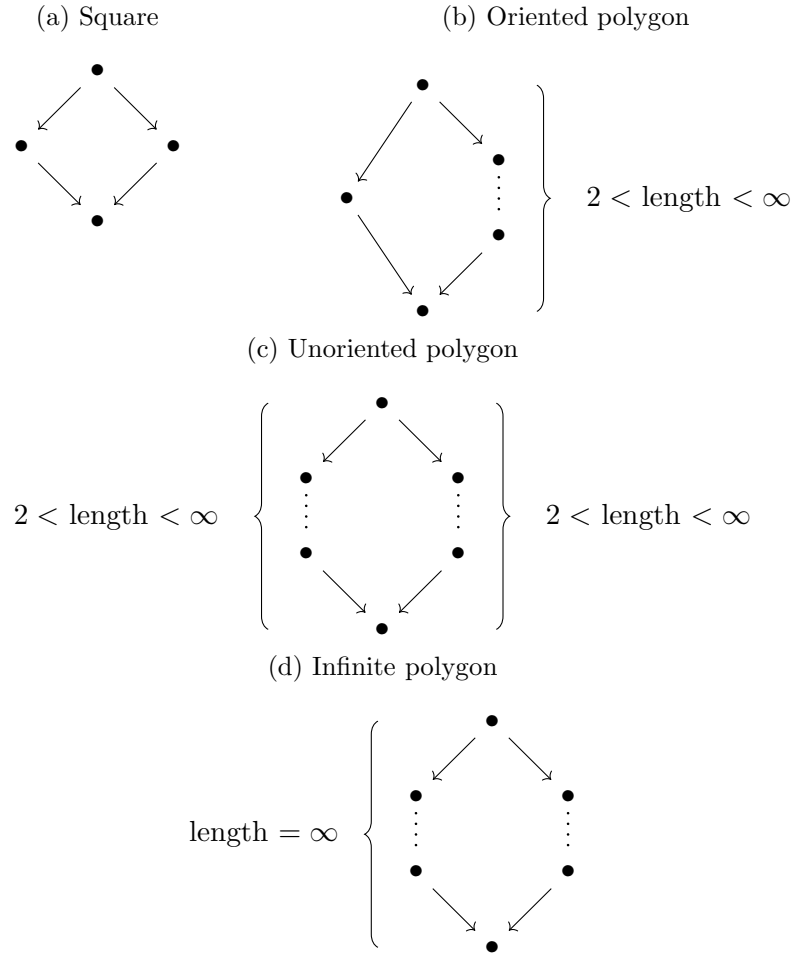
\begin{figure}
    \caption{Different types of polygons}\label{fig:poly}
    \begin{subfigure}[t]{0.35\textwidth}
    \caption{Square}\label{fig:poly:square}
    \[
    \begin{tikzpicture}
        
        \node (t) at (0,0) {$\bullet$};
        \node (l) at (-1,-1) {$\bullet$};
        \node (r) at (1,-1) {$\bullet$};
        \node (b) at (0,-2) {$\bullet$};
        
        \draw[->] (t) -- (l);
        \draw[->] (t) -- (r);
        \draw[->] (l) -- (b);
        \draw[->] (r) -- (b);
        
    \end{tikzpicture}
    \]
    \end{subfigure}
    \begin{subfigure}[t]{0.55\textwidth}
    \caption{Oriented polygon}\label{fig:poly:orient}
    \[
    \begin{tikzpicture}
        
        \node (t) at (0,0) {$\bullet$};
        \node (l) at (-1,-1.5) {$\bullet$};
        \node (r1) at (1,-1) {$\bullet$};
        \node (r2) at (1,-2) {$\bullet$};
        \node (b) at (0,-3) {$\bullet$};
        
        \draw[->] (t) -- (l);
        \draw[->] (t) -- (r1);
        \draw[dot diameter=1pt, dot spacing=4pt, dots] (r1) -- (r2);
        \draw[->] (l) -- (b);
        \draw[->] (r2) -- (b);
        
        \draw[decorate,decoration={brace,amplitude=5pt}] (1.5,0) -- (1.5,-3);
        
        \node at (3.5,-1.5) {$2 <$ length $< \infty$};
        
    \end{tikzpicture}
    \]
    \end{subfigure}
    \begin{subfigure}[t]{0.9\textwidth}
    \caption{Unoriented polygon}\label{fig:poly:unorient}
    \[
    \begin{tikzpicture}
        
        \node (t) at (0,0) {$\bullet$};
        \node (l1) at (-1,-1) {$\bullet$};
        \node (l2) at (-1,-2) {$\bullet$};
        \node (r1) at (1,-1) {$\bullet$};
        \node (r2) at (1,-2) {$\bullet$};
        \node (b) at (0,-3) {$\bullet$};
        
        \draw[->] (t) -- (l1);
        \draw[->] (t) -- (r1);
        \draw[dot diameter=1pt, dot spacing=4pt, dots] (r1) -- (r2);
        \draw[dot diameter=1pt, dot spacing=4pt, dots] (l1) -- (l2);
        \draw[->] (l2) -- (b);
        \draw[->] (r2) -- (b);
                
        \draw[decorate,decoration={brace,amplitude=5pt}] (1.5,0) -- (1.5,-3);
        \draw[decorate,decoration={brace,amplitude=5pt,mirror}] (-1.5,0) -- (-1.5,-3);
        
        \node at (3.5,-1.5) {$2 <$ length $< \infty$};
        \node at (-3.5,-1.5) {$2 <$ length $< \infty$};
        
    \end{tikzpicture}
    \]
    \end{subfigure}
    \begin{subfigure}[t]{0.9\textwidth}
    \caption{Infinite polygon}\label{fig:poly:infinite}
    \[
    \begin{tikzpicture}
        
        \node (t) at (0,0) {$\bullet$};
        \node (l1) at (-1,-1) {$\bullet$};
        \node (l2) at (-1,-2) {$\bullet$};
        \node (r1) at (1,-1) {$\bullet$};
        \node (r2) at (1,-2) {$\bullet$};
        \node (b) at (0,-3) {$\bullet$};
        
        \draw[->] (t) -- (l1);
        \draw[->] (t) -- (r1);
        \draw[dot diameter=1pt, dot spacing=4pt, dots] (r1) -- (r2);
        \draw[dot diameter=1pt, dot spacing=4pt, dots] (l1) -- (l2);
        \draw[->] (l2) -- (b);
        \draw[->] (r2) -- (b);
        
        \draw[decorate,decoration={brace,amplitude=5pt,mirror}] (-1.5,0) -- (-1.5,-3);

        \node at (-3,-1.5) {length $= \infty$};
        
    \end{tikzpicture}
    \]
    \end{subfigure}
\end{figure}

\begin{definition}
A \emph{finite polygon} in $\twosilt \Lambda$ is a polygon which is not an infinite polygon. A finite polygon therefore has two finite paths from its maximum to its minimum.

Two maximal green sequences $\mathcal{G}$ and $\mathcal{G}'$ are related by \emph{deformation across a \textup{(}finite\textup{)} polygon} if we have
\begin{align*}
    \mathcal{G} &= (T_0, \dots, T_i, U_1, \dots, U_k, T_{i + 1}, \dots, T_r), \\
    \mathcal{G}' &= (T_0, \dots, T_i, V_1, \dots, V_l, T_{i + 1}, \dots, T_r),
\end{align*}
as sequences of green mutations of silting complexes, with 
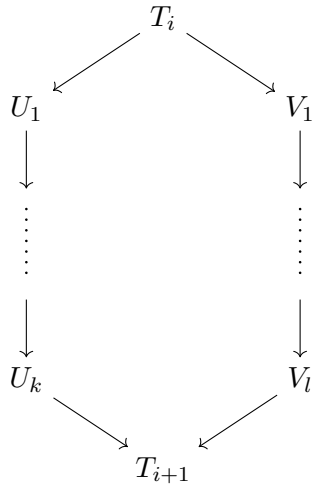
\[
\begin{tikzpicture}[scale=1.2,xscale=1.5]

\node (TT) at (0,0) {$T_{i}$};
\node (U1) at (-1,-1) {$U_{1}$};
\node (U2) at (-1,-2) {};
\node (Uk-1) at (-1,-3) {};
\node (Uk) at (-1,-4) {$U_{k}$};
\node (TB) at (0,-5) {$T_{i + 1}$}; 
\node (V1) at (1,-1) {$V_{1}$};
\node (V2) at (1,-2) {};
\node (Vl-1) at (1,-3) {};
\node (Vl) at (1,-4) {$V_{l}$};

\draw[->] (TT) -- (U1);
\draw[->] (U1) -- (U2);
\draw[->] (Uk-1) -- (Uk);
\draw[dot diameter=1pt, dot spacing=4pt, dots] (U2) -- (Uk-1);
\draw[->] (Uk) -- (TB);
\draw[->] (TT) -- (V1);
\draw[->] (V1) -- (V2);
\draw[->] (Vl-1) -- (Vl);
\draw[dot diameter=1pt, dot spacing=4pt, dots] (V2) -- (Vl-1);
\draw[->] (Vl) -- (TB);

\end{tikzpicture}
\] a finite polygon.
\end{definition}

In the case of a finite polygon, the two paths around the polygon from the maximum to the minimum are the only paths between these silting complexes in the poset $\twosilt \Lambda$.

\begin{lemma}\label{lem:iepd_unique}
Let $E \oplus X \oplus X'$ and $E \oplus Y \oplus Y'$ be the respective maximum and minimum of a polygon in $\twosilt \Lambda$. Then the only paths from $E \oplus X \oplus X'$ to $E \oplus Y \oplus Y'$ in $\twosilt \Lambda$ are the two paths around the polygon.
\end{lemma}
\begin{proof}
Each vertex in the Hasse diagram of the polygon has degree two, corresponding to the two indecomposable summands of the two-term silting complex at the vertex which are not summands of $E$. Having another path from $E \oplus X \oplus X'$ to $E \oplus Y \oplus Y'$ would require mutating an indecomposable direct summand of $E$. However, Lemma~\ref{lem:at_most_once} then precludes the path reaching $E \oplus Y \oplus Y'$.
\end{proof}

Any convex subposet which looks like a square or an oriented polygon must in fact be a square or an oriented polygon. Recall that a subposet $\mathsf{P}'$ of a poset $\mathsf{P}$ is convex if whenever $p \leqslant q \leqslant r$ with $p, r \in \mathsf{P}'$, we have that $q \in \mathsf{P}'$ too.

\begin{lemma}
Any convex subposet of $\twosilt \Lambda$ isomorphic to Figure~\ref{fig:poly}\textup{(}\subref{fig:poly:square}\textup{)} or \ref{fig:poly}\textup{(}\subref{fig:poly:orient}\textup{)} is a polygon.
\end{lemma}
\begin{proof}
This follows from the fact that in each of these cases there is a path from the maximum of the poset to the minimum of the poset of length two. Since the covering relations from Figure~\ref{fig:poly} correspond to covering relations in $\twosilt \Lambda$, we must have that the maximum of the poset is two mutations away from the bottom of the poset, and so shares all but two summands with it. Hence, there exists $E$ such that $|E| = |\Lambda| - 2$ such that $E$ is a summand of all two-term silting complexes along the length-two path from maximum to minimum. It then follows from Lemma~\ref{lem:at_most_once} that $E$ must be a summand of all the two-term silting complexes along the other path too. We conclude that the subposet is indeed a polygon.
\end{proof}

\begin{remark}\label{rmk:reduction}
The different types of polygons correspond to the two-term silting theory of different algebras $\Gamma$ with $|\Gamma| = 2$. Indeed, let $E$ be a two-term presilting complex with $|E| = |\Lambda| - 2$. Then the complexes which complete $E$ to a silting complex must lie in \[\mathcal{Z} = (\hol{E[>0]}) \cap (\hor{E[<0]}).\] 
Here
\begin{align*}
    \hol{E[>0]} &= \{X \in \kbproj \st \Hom_{\kbproj}(X, E[i]) = 0 \: \forall i > 0\},\\
    \hor{E[<0]} &= \{X \in \kbproj \st \Hom_{\kbproj}(E[i], X) = 0 \: \forall i < 0\}.
\end{align*}
Let $\widetilde{(-)} \colon \mathcal{Z} \to \mathcal{Z}/[E]$ be the ideal quotient of $\mathcal{Z}$ by the ideal of morphisms factoring through $\additive E$. We then have that $\mathcal{Z}/[E]$ is a triangulated category by \cite[Theorem~4.2]{iy-red}. Further, let $E \oplus X \oplus Y$ be the Bongartz completion of~$E$. Then, it follows from \cite{jasso_red,iyama_yang} that the polygon determined by $E$ is isomorphic as a poset to $\twosilt \Gamma$ where $\Gamma = \End_{\mathcal{Z}/[E]} (X \oplus Y)$ via sending a two-term silting complex $E \oplus X' \oplus Y'$ in $\twoterm$ to the two-term silting complex $\widetilde{X'} \oplus \widetilde{Y'}$ of $\Gamma$. This process is known as \emph{silting reduction} and was introduced in \cite{ai} and given this description in \cite{iyama_yang}.
\end{remark}

\subsection{Deformations across squares}

A particular instance of deformation across a polygon is deformation across a square. This will be used to define one notion of equivalence for maximal green sequences. In this section, we study deformations across squares in terms of silting complexes and bricks.

\subsubsection{Deformations across squares in terms of silting}

We prove some results on deformations across squares from the perspective of silting.

\begin{lemma}\label{lem:sq_exch_pairs}
If the exchange pairs in one path around the square are $(X, Y)$ and then $(X', Y')$, then the exchange pairs in the path around the other side of the square are $(X', Y')$ and then $(X, Y)$.
\end{lemma}
\begin{proof}
Let $T$ be the two-term silting complex at the top of the square. Then $X$ is certainly an indecomposable summand of $T$. We must also have that $X'$ is an indecomposable summand of $T$, since we cannot have $X' \cong Y$. Hence, let $T = E \oplus X \oplus X'$. This gives the path around the square we know as \[E \oplus X \oplus X' \to E \oplus Y \oplus X' \to E \oplus Y \oplus Y'.\] Let $T'$ be the two-term silting complex in the middle of the path around the other side of the square. We must have that $T'$ has all but one indecomposable summand in common with $E \oplus X \oplus X'$ and $E \oplus Y \oplus Y'$, since it is related to each of these silting complexes by mutation. It is then immediate to see that $T '\cong E \oplus X \oplus Y'$, which gives us the result.
\end{proof}

The following criterion for when one can swap the order of two exchange pairs will be useful later.

\begin{lemma}\label{lem:sq_crit}
Let $\mathcal{G}$ be a maximal green sequence of $\Lambda$ with an exchange pair $(X, Y)$ immediately succeeded by an exchange pair $(X', Y')$. Then one can deform $\mathcal{G}$ across a square with sides $(X,Y)$ and $(X', Y')$ to obtain another maximal green sequence $\mathcal{G}'$ if and only if $\Hom_{\kbproj}(Y',X[1]) = 0$.
\end{lemma}
\begin{proof}
Let \[
\begin{tikzcd}
T_{i-1} \ar[r,"{(X,Y)}"] & T_{i} \ar[r,"{(X',Y')}"] & T_{i + 1} 
\end{tikzcd}
\] be the relevant part of $\mathcal{G}$. Suppose that we can deform $\mathcal{G}$ across a square with these sides to give a maximal green sequence $\mathcal{G}'$, which, by Lemma~\ref{lem:sq_exch_pairs}, instead has the sequence \[
\begin{tikzcd}
T_{i - 1} \ar[r, "{(X',Y')}"] & T'_{i} \ar[r,"{(X,Y)}"] & T_{i + 1}. 
\end{tikzcd}
\] Then $T'_{i}$ contains both $Y'$ and $X$ as summands, and so we must have that $\Hom_{\kbproj}(Y',X[1]) = 0$, since $T'_{i}$ is silting.

Now we show the reverse direction, maintaining our labelling of $\mathcal{G}$ as above, and supposing that $\Hom_{\kbproj}(Y', X[1]) = 0$. We then have that $T_{i - 1} = E \oplus X \oplus X'$ for some $E \in \twoterm$. Moreover, if we let $T'_{i} = E \oplus X \oplus Y'$, then $\Hom_{\kbproj}(T'_{i}, T'_{i}[1]) = 0$. Indeed, we have the following:
\begin{itemize}
\item $\Hom_{\kbproj}(E \oplus Y', (E \oplus Y')[1]) = 0$, as $E \oplus Y'$ is a summand of $T_{i + 1}$; 
\item $\Hom_{\kbproj}(E \oplus X, (E \oplus X) [1]) = 0$, as $E \oplus X$ is a summand of $T_{i - 1}$; 
\item $\Hom_{\kbproj}(Y', X[1]) = 0$ by assumption;
\item $\Hom_{\kbproj}(X, Y'[1]) = 0$, since $Y' \in \eor{T_i} \subset \eor{T_{i - 1}}$ by Lemma~\ref{lem:mut_dir}.
\end{itemize} 
By Lemma~\ref{lem:at_most_once}, we have $Y' \not\cong X$. Therefore, $T'_{i}$ has the maximal number of isomorphism classes of indecomposable summands. Hence $T'_{i}$ is a silting complex and we obtain a maximal green sequence $\mathcal{G}$ by replacing the relevant portion of $\mathcal{G}$ with \[
\begin{tikzcd}
T_{i - 1} \ar[r,"{(X',Y')}"] & T'_{i} \ar[r,"{(X,Y)}"] & T_{i + 1}. 
\end{tikzcd}
\]
\end{proof}

\subsubsection{Deformations across squares in terms of bricks}

We now consider deformations across squares from the point of view of bricks.

\begin{lemma} \label{lem:square-bricks}
Given a maximal green sequence $\mathcal{G}$ of $\Lambda$ as the maximal backwards $\Hom$-orthogonal sequence of bricks \[B_{1}, B_{2}, \dots, B_{r},\] a maximal green sequence $\mathcal{G}'$ is related to $\mathcal{G}$ by deformation across a square if and only if $\mathcal{G}'$ is given by \[B_{0}, \dots, B_{i - 1}, B_{i + 1}, B_{i}, B_{i + 2}, \dots, B_{r}\] as a maximal backwards $\Hom$-orthogonal sequence of bricks for some~$i$.
\end{lemma}
\begin{proof}
The maximal green sequences $\mathcal{G}$ and $\mathcal{G}'$ are related by deformation across a square if and only if there is a square \[
\begin{tikzpicture}

\node (T) at (0,0) {$\mathcal{T}_{i - 1}$};
\node (L) at (-1,-1) {$\mathcal{T}_{i}$};
\node (R) at (1,-1) {$\mathcal{T}'_{i}$};
\node (B) at (0,-2) {$\mathcal{T}_{i + 1}$};

\draw[->] (T) -- (L);
\draw[->] (L) -- (B);
\draw[->] (T) -- (R);
\draw[->] (R) -- (B);

\end{tikzpicture}
\] such that $\mathcal{G}$ and $\mathcal{G}'$ differ only in that $\mathcal{G}$ contains the left-hand path around the square, whilst $\mathcal{G}'$ contains the right-hand path around the square. Let $B_{i}$ and $B_{i + 1}$ be the respective brick labels of the minimal inclusions $\mathcal{T}_{i - 1} \supset \mathcal{T}_{i}$ and $\mathcal{T}_{i} \supset \mathcal{T}_{i + 1}$. Then we have that $[\mathcal{T}_{i - 1}, \mathcal{T}_{i + 1}] = \Filt(B_{i}, B_{i + 1})$. Moreover, $B_{i}$ and $B_{i + 1}$ must be precisely the relatively simple objects in $\Filt(B_{i}, B_{i + 1})$ since if either $B_i \in \Filt(B_{i + 1})$ or $B_{i + 1} \in \Filt(B_i)$, then backwards $\Hom$-orthogonality is violated, and there cannot be more relatively simple objects in  $\Filt(B_{i}, B_{i + 1})$ by Proposition~\ref{prop:simples_in_intervals}. By applying Proposition~\ref{prop:simples_in_intervals} to the other path around the square, we see that the brick labels of the other two minimal inclusions must also be $B_{i}$ and $B_{i + 1}$. We then must have that $B_{i + 1}$ labels $\mathcal{T}_{i - 1} \supset \mathcal{T}'_{i}$ and $B_{i}$ labels $\mathcal{T}'_{i} \supset \mathcal{T}_{i + 1}$, since $\mathcal{T}_{i} \neq \mathcal{T}'_{i}$.
\end{proof}

In order to prove the analogue of Lemma~\ref{lem:sq_crit} for bricks, we need the following lemma. 

\begin{lemma}\label{lem:extension_of_bricks}
Suppose that $L$ and $N$ are bricks over $\Lambda$ such that \[\Hom_{\Lambda}(L, N) = \Hom_{\Lambda}(N, L) = 0.\] Then every non-split extension of $L$ and $N$ is a brick.
\end{lemma}
\begin{proof}
Suppose that \[0 \to L \xrightarrow{f} M \xrightarrow{g} N \to 0\] is a non-split extension of $L$ and $N$. We want to show that $M$ is a brick, that is, if we let $h\colon M \to M$ be an endomorphism, then $h$ is either zero or an isomorphism. Since $\Hom_{\Lambda}(L, N) = 0$, we have that $ghf = 0$. Hence, by the universal properties of kernels and cokernels, we have a commutative diagram \[
\begin{tikzcd}
0 \ar[r] & L \ar[r,"f"] \ar[d,"a"] & M \ar[r,"g"] \ar[d,"h"] & N \ar[r] \ar[d,"b"] & 0 \\
0 \ar[r] & L \ar[r,"f"] & M \ar[r,"g"] & N \ar[r] & 0.
\end{tikzcd}
\]
Since $L$ and $N$ are both bricks, we have that $a$ and $b$ are both either isomorphisms or zero. If they are both isomorphisms, then $h$ is also an isomorphism by the Five Lemma.

Hence, we suppose that at least one of $a$ and $b$ is not an isomorphism. Suppose first that $a$ is zero and $b$ is an isomorphism. Thus, $hf = fa = 0$. By the universal property of cokernels, we have a map $s\colon N \to M$ such that $sg = h$. \[
\begin{tikzcd}
0 \ar[r] & L \ar[r,"f"] \ar[d,"0"] & M \ar[r,"g"] \ar[d,"h"] & N \ar[r] \ar[d,"b"] \ar[dl,"s",dashed] & 0 \\
0 \ar[r] & L \ar[r,"f"] & M \ar[r,"g"] & N \ar[r] & 0
\end{tikzcd}
\] We then have that $gsg = gh = bg$, which implies that $gs = b$, since $g$ is epic. Since $b$ is an isomorphism, it has an inverse $b^{-1}$. We then have that $gsb^{-1} = \mathrm{id}_N$, so that $sb^{-1}$ is a section of $g$. But this means that the extension $0 \to L \to M \to N \to 0$ is split, which is a contradiction. The case where $a$ is an isomorphism and $b$ is zero is similar to this.

The final case to consider is where $a$ and $b$ are both zero. This gives that $hf = fa = 0$, and so we have a map $s\colon N \to M$ such that $h = sg$. Then $gsg = gh = bg = 0$, which implies that $gs = 0$, since $g$ is epic. By the universal property of kernels, we have that there is a map $t\colon N \to L$ such that $s = ft$. \[
\begin{tikzcd}
0 \ar[r] & L \ar[r,"f"] \ar[d,"0"] & M \ar[r,"g"] \ar[d,"h"] & N \ar[r] \ar[d,"0"] \ar[dl,"s",dashed] \ar[dll,"t",dashed,swap] & 0 \\
0 \ar[r] & L \ar[r,"f",swap] & M \ar[r,"g"] & N \ar[r] & 0
\end{tikzcd}
\] However, $\Hom_{\Lambda}(N, L) = 0$, so $t = 0$. Consequently, $s = ft = 0$ and, in turn, $h = sg = 0$. We conclude that every endomorphism of $M$ is either an isomorphism or zero, as desired.
\end{proof}

We apply this to show the following.

\begin{lemma}\label{lem:brick_no_ext}
Let $\mathcal{G}$ be a maximal green sequence of $\Lambda$ given by a maximal backwards $\Hom$-orthogonal sequence of bricks \[B_1, B_2, \dots, B_i, B_{i + 1}, \dots, B_r.\] If $\Hom_{\Lambda}(B_i, B_{i+1}) = 0$, then $\Ext_{\Lambda}^{1}(B_{i + 1}, B_{i}) = 0$.
\end{lemma}
\begin{proof}
If $\Ext_{\Lambda}^{1}(B_{i + 1}, B_{i}) \neq 0$, then there is a non-split short exact sequence \[0 \to B_{i} \to B \to B_{i + 1} \to 0.\] The module $B$ here must then be a brick by Lemma~\ref{lem:extension_of_bricks}.

We claim that the sequence of bricks given by \[B_1, B_2, \dots, B_{i - 1}, B_{i}, B, B_{i + 1}, B_{i + 2}, \dots, B_r\] is also backwards $\Hom$-orthogonal. Suppose there exists $B_j$ with $j > i + 1$ such that there is a non-zero homomorphism $B_j \to B$. The composition $B_j \to B \to B_{i + 1}$ must then be zero, by backwards $\Hom$-orthogonality of the original sequence. But this gives a non-zero map $B_j \to B_{i}$ by the universal property of the kernel, which is a contradiction. One can similarly argue that there is no non-zero map $B \to B_j$ for $j < i$.

We must finally show that there cannot be any non-zero maps $B_{i + 1} \to B$ or $B \to B_{i}$. In the first case, if the composition $B_{i + 1} \to B \to B_{i + 1}$ is zero, then there is a contradictory non-zero map $B_{i + 1} \to B_{i}$. Hence the composition $B_{i + 1} \to B \to B_{i + 1}$ is non-zero and cannot be an isomorphism since $B$ is indecomposable. This contradicts the fact that $B_{i + 1}$ is a brick. The existence of a non-zero map $B \to B_{i}$ is likewise contradictory.
\end{proof}

The following lemma is the analogue of Lemma~\ref{lem:sq_crit} for bricks: it tells us when we can exchange two consecutive bricks in order to deform across a square.

\begin{lemma}\label{lem:sq_crit_brick}
Let $\mathcal{G}$ be a maximal green sequence of $\Lambda$ given by a maximal backwards $\Hom$-orthogonal sequence of bricks \[B_1, B_2, \dots, B_i, B_{i + 1}, \dots, B_r.\] Then \[B_1, B_2, \dots, B_{i - 1}, B_{i + 1}, B_i, B_{i + 2}, \dots, B_r\] is a maximal backwards $\Hom$-orthogonal sequence of bricks if and only if \[\Hom_{\Lambda}(B_i, B_{i + 1}) = \Ext_{\Lambda}^{1}(B_i, B_{i + 1}) = 0.\]
\end{lemma}
\begin{proof}
We first show that the conditions are necessary. It is immediate that we must have $\Hom_{\Lambda}(B_i, B_{i + 1}) = 0$ if the sequence is to remain backwards $\Hom$-orthogonal. By Lemma~\ref{lem:brick_no_ext}, we then also have that $\Ext_{\Lambda}^{1}(B_i, B_{i + 1}) = 0$.

We now show that the conditions are sufficient. We have that the new sequence is backwards $\Hom$-orthogonal as $\Hom_{\Lambda}(B_i, B_{i + 1}) = 0$. Suppose now that the new sequence is not maximal. Since the original sequence is maximal, the only place where a new brick $B$ could be added to the new sequence is between $B_{i + 1}$ and $B_{i}$. We have then that $B \in \Filt(B_{i}, B_{i + 1})$ by Theorem~\ref{thm:filt_intervals}. Since, by $\Hom$-orthogonality, we have that $\Hom_{\Lambda}(B_i, B) = \Hom_{\Lambda}(B, B_{i + 1}) = 0$, $B$ cannot contain $B_i$ as a submodule or $B_{i + 1}$ as a factor module. We let \[B = L_{0} \supset L_{1} \supset \dots \supset L_{s} = 0\] be the composition series of $B$ in terms of $B_{i}$ and $B_{i + 1}$. We let $F_{i} = L_{i - 1}/L_{i}$. Then we must have that $F_{1} \cong B_{i}$ and $F_{s} \cong B_{i + 1}$ by backwards $\Hom$-orthogonality. Hence, there must be some $j$ such that $F_{j} \cong B_{i}$ and $F_{j + 1} \cong B_{i + 1}$. We can swap the order of these two factors in the composition series if the subquotient $L_{j - 1}/L_{j + 1} \cong B_i \oplus B_{i+1}$.  If we continue making such swaps where possible, we must eventually reach a case where $L_{j - 1}/L_{j + 1}$ is not isomorphic to such a direct sum, since we cannot have a composition series with $F_{1} \cong B_{i + 1}$ or $F_{s} \cong B_{i}$.

Thus, there exists a module $B' = L_{j - 1}/L_{j + 1}$ with a non-split short exact sequence \[0 \to B_{i + 1} \to B' \to B_{i} \to 0.\] This gives that $\Ext_{\Lambda}^{1}(B_i, B_{i + 1}) \neq 0$, which is a contradiction.
\end{proof}

The following observation will be used in several proofs.

\begin{lemma}\label{lem:adj_bricks}
Suppose that $\mathcal{G}$ and $\mathcal{G}'$ are distinct maximal green sequences such that $\bricks{\mathcal{G}} \supseteq \bricks{\mathcal{G}'}$. Then there exists a pair of bricks $B, B'$ which are adjacent with $B < B'$ in $\mathcal{G}'$, but which have $B' < B$ in $\mathcal{G}$.
\end{lemma}
\begin{proof}
There must be a pair of bricks of $\mathcal{G}'$ which are ordered differently under $\mathcal{G}$, since we are assuming that $\mathcal{G}$ and $\mathcal{G}'$ are distinct. Thus, choose a pair of bricks $B$ and $B'$ of $\mathcal{G}'$ which are ordered differently in $\mathcal{G}$, such that $B$ and $B'$ are as close as possible in $\mathcal{G}'$. If there is a brick $B''$ of $\mathcal{G}'$ between $B$ and $B'$, then one of the pairs $(B, B'')$ and $(B'', B')$ must be ordered differently in $\mathcal{G}$. But this contradicts the choice of $B$ and $B'$ as the closest bricks which are ordered differently between the two sequences.
\end{proof}

\subsection{Characterising equivalent maximal green sequences}

Our preliminary work now puts us in a position to give different characterisations of equivalent maximal green sequences. The following theorem gives us six criteria for when a pair of maximal green sequences are equivalent.

\begin{theorem}\label{thm:mg_equiv}
Let $\Lambda$ be a finite-dimensional algebra over a field $K$. Let $\mathcal{G}_{1}$ and $\mathcal{G}_{2}$ be two maximal green sequences of $\Lambda$. Then the following are equivalent.
\begin{enumerate}[label=\textup{(}\arabic*\textup{)}]
\item $\mathcal{G}_{1}$ and $\mathcal{G}_{2}$ are related by a finite sequence of deformations across squares.\label{op:equiv:sq}
\item $\exch{\mathcal{G}_{1}} = \exch{\mathcal{G}_{2}}$.\label{op:equiv:exch}
\item $\summ{\mathcal{G}_{1}} = \summ{\mathcal{G}_{2}}$.\label{op:equiv:summ}
\item $\sumt{\mathcal{G}_{1}} = \sumt{\mathcal{G}_{2}}$.\label{op:equiv:sumt}
\item For any $\Lambda$-module $M$, $\hnf{\mathcal{G}_{1}}{M} = \hnf{\mathcal{G}_{2}}{M}$.\label{op:equiv:hn}
\item For any $\Lambda$-module $M$, $\bhnf{\mathcal{G}_{1}}{M} = \bhnf{\mathcal{G}_{2}}{M}$.\label{op:equiv:hnb}
\end{enumerate}
\end{theorem}

The fact that these six notions of equivalence of maximal green sequences are the same indicates that this is the ``correct'' equivalence relation to impose upon maximal green sequences. See also Remarks~\ref{rem:cluster} and \ref{rem:DT} for further explanation of the intuition behind this relation. 

\begin{definition} \label{def:equivalence}
Given a finite-dimensional algebra $\Lambda$ over a field $K$ with $\mathcal{G}$ and $\mathcal{G}'$ two maximal green sequences of $\Lambda$, we say that $\mathcal{G}$ and $\mathcal{G}'$ are \emph{equivalent} if any one of the six interchangeable conditions from Theorem~\ref{thm:mg_equiv} holds. If $\mathcal{G}$ and $\mathcal{G}'$ are equivalent, then we write $\mathcal{G} \sim \mathcal{G}'$.
\end{definition}

We first show that, under one of the conditions in the statement of Theorem~\ref{thm:mg_equiv}, maximal green sequences have the same bricks.

\begin{lemma}\label{lem:same_bricks}
If $\mathcal{G}$ and $\mathcal{G}'$ are maximal green sequences such that for any $\Lambda$-module $M$ we have $\bhnf{\mathcal{G}}{M} = \bhnf{\mathcal{G}'}{M}$, then we have that $\bricks{\mathcal{G}} = \bricks{\mathcal{G}'}$.
\end{lemma}
\begin{proof}
Let $B \in \bricks{\mathcal{G}}$. Then $\{B\} = \bhnf{\mathcal{G}}{B} = \bhnf{\mathcal{G}'}{B}$, by assumption. We must then have $B \in \bricks{\mathcal{G}'}$.
\end{proof}

We now prove our first main theorem.

\begin{proof}[Proof of Theorem~\ref{thm:mg_equiv}]
We will prove the implications 
\[\mbox{\ref{op:equiv:sq}} \Rightarrow \mbox{\ref{op:equiv:exch}} \Rightarrow \mbox{\ref{op:equiv:summ}}  \Rightarrow \mbox{\ref{op:equiv:sq}},\]
\[\mbox{\ref{op:equiv:sq}} \Rightarrow \mbox{\ref{op:equiv:hn}} \Rightarrow  \mbox{\ref{op:equiv:hnb}} \Rightarrow \mbox{\ref{op:equiv:sq}},\]
and the equivalence 
\[\mbox{\ref{op:equiv:summ}} \Leftrightarrow \mbox{\ref{op:equiv:sumt}}.\]
It is immediate from \cite[Theorem~3.2]{air} that \ref{op:equiv:summ} and \ref{op:equiv:sumt} are equivalent.
That \ref{op:equiv:exch} implies \ref{op:equiv:summ} is evident. It is also immediate that \ref{op:equiv:sq} implies \ref{op:equiv:exch} since deforming across a square does not change the set of exchange pairs by Lemma~\ref{lem:sq_exch_pairs}. Finally, it is clear that \ref{op:equiv:hn} implies \ref{op:equiv:hnb}, since the set of semistable factors determines the set of stable factors.

We now show that \ref{op:equiv:summ} implies \ref{op:equiv:sq}. Suppose that we have maximal green sequences $\mathcal{G}_{1}, \mathcal{G}_{2}$ such that $\summ{\mathcal{G}_{1}} = \summ{\mathcal{G}_{2}}$. We will prove that $\mathcal{G}_1$ and $\mathcal{G}_2$ can be deformed into each other across squares by induction on the distance between the first point where they diverge and the end of the maximal green sequences at $\Lambda[1]$. The base case is when this distance is zero, so that $\mathcal{G}_1 = \mathcal{G}_2$, in which case it is trivial that the two maximal green sequences are related by deformations across squares.

Now suppose that $\mathcal{G}_1 \neq \mathcal{G}_2$. Consider the first exchange pairs where $\mathcal{G}_{1}$ and $\mathcal{G}_{2}$ diverge. Let this exchange pair be $(X_{1}, Y_{1})$ for $\mathcal{G}_{1}$ and $(X_{2}, Y_{2})$ for $\mathcal{G}_{2}$ and let $T$ be the last two-term silting complex they share before they diverge. By Lemma~\ref{lem:exch_crit}, we must actually also have $(X_{1}, Y_{1}) \in \exch{\mathcal{G}_{2}}$ and $(X_{2}, Y_{2}) \in \exch{\mathcal{G}_{1}}$. Indeed, we know from the fact that $(X_{1}, Y_{1})$ is an exchange pair for $\mathcal{G}_{1}$ that $X_{1}$ is the unique indecomposable object (up to isomorphism) in $\summ{\mathcal{G}_1} \cap \eor{T} = \summ{\mathcal{G}_2} \cap \eor{T}$ such that $\Hom_{\kbproj}(Y_{1}, X_{1}[1]) \neq 0$, using Lemma~\ref{lem:exch_crit}\ref{op:exch_crit:use_exch}. Then, using Lemma~\ref{lem:exch_crit}\ref{op:exch_crit:get_exch} on $\mathcal{G}_{2}$, we obtain that $(X_{1}, Y_{1}) \in \exch{\mathcal{G}_{2}}$. The mirror-image of this argument shows that $(X_{2}, Y_{2}) \in \exch{\mathcal{G}_{1}}$. 

We claim that, by deforming across squares, we can make $(X_{2},Y_{2})$ the exchange pair before $(X_{1}, Y_{1})$ in $\mathcal{G}_{1}$. Suppose that we cannot deform $(X_{2}, Y_{2})$ back past some exchange pair $(A, B)$ in $\mathcal{G}_{1}$ which occurs after $T$. Then, by Lemma~\ref{lem:sq_crit},  we have that $\Hom_{\kbproj}(Y_{2}, A[1]) \neq 0$. We then know that we cannot have $A$ as a summand of $T$, since this would mean we cannot exchange $X_{2}$ for $Y_{2}$ in $T$ as part of $\mathcal{G}_2$, by Lemma~\ref{lem:mut_dir}\ref{op:silt3}.
We further know that $A \in \summ{\mathcal{G}_{1}} \cap \eor{T}$ by Corollary~\ref{cor:mut_dir}, since $A$ occurs in $\mathcal{G}_{1}$ after~$T$. By Lemma~\ref{lem:exch_crit}, we know that $X_{2}$ is the unique indecomposable in $\summ{\mathcal{G}_{2}} \cap \eor{T} = \summ{\mathcal{G}_{1}} \cap \eor{T}$ such that $\Hom_{\kbproj}(Y_{2}, X_{2}[1]) \neq 0$. But then $A \cong X_{2}$, which contradicts Lemma~\ref{lem:at_most_once}.

Hence we may deform $\mathcal{G}_{1}$ across squares to obtain a maximal green sequence $\mathcal{G}'_{1}$ where $(X_{1}, Y_{1})$ is the exchange pair immediately following $(X_{2}, Y_{2})$. But then, $\mathcal{G}'_1$ and $\mathcal{G}_2$ agree on a longer initial segment than $\mathcal{G}_1$ and $\mathcal{G}_2$, so, by the induction hypothesis, we can deform $\mathcal{G}'_1$ into $\mathcal{G}_2$ across squares. This then gives that we can deform $\mathcal{G}_1$ into $\mathcal{G}_2$ across squares, which establishes the claim.

We now show that \ref{op:equiv:sq} implies \ref{op:equiv:hn} by showing that the factors of Harder--Narasimhan filtrations are preserved by deformations across squares. We start with a maximal backwards $\Hom$-orthogonal sequence of bricks $\mathcal{G}_{1}$ \[B_1, B_2, \dots, B_i, B_{i + 1}, \dots, B_r\] and deform across a square to obtain a sequence $\mathcal{G}_{2}$ given by \[B_1, B_2, \dots, B_{i - 1}, B_{i + 1}, B_i, B_{i + 2}, \dots, B_r.\] By Lemma~\ref{lem:sq_crit_brick}, we have that 
\begin{align*}
    \Hom_{\Lambda}(B_i, B_{i + 1}) &= \Hom_{\Lambda}(B_{i + 1}, B_{i}) \\
    &= \Ext_{\Lambda}^{1}(B_i, B_{i + 1}) = \Ext_{\Lambda}^{1}(B_{i + 1}, B_{i}) = 0.
\end{align*}

We let $M$ be a $\Lambda$-module and show that $\hnf{\mathcal{G}_{1}}{M} = \hnf{\mathcal{G}_{2}}{M}$. If $\Filt(B_{i}) \cap \hnf{\mathcal{G}_{1}}{M} = \emptyset$ or $\Filt(B_{i + 1}) \cap \hnf{\mathcal{G}_{1}}{M} = \emptyset$, then the same filtration is also the Harder--Narasimhan filtration of $M$ by $\mathcal{G}_{2}$. Suppose then that the Harder--Narasimhan filtration of $M$ by $\mathcal{G}_{1}$ is \[M = M_0 \supset M_1 \supset \dots \supset M_{l - 1} \supset M_l = 0\] where $F_{k} = M_{k - 1 }/M_{k} \in \Filt(B_{j_{k}})$ for some $j_{k}$ for all $1 \leqslant k \leqslant l$, with \[j_{1} < j_{2} < \dots < j_{l}.\] Suppose that $j_{k} = i$ and $j_{k + 1} = i + 1$. Since $\Ext_{\Lambda}^{1}(B_{i}, B_{i + 1}) = 0$, the subquotient $M_{k - 1}/M_{k + 1}$ is isomorphic to $F_{k} \oplus F_{k + 1}$. Hence, we may replace $M_{k}$ by $M'_{k}$ such that $M_{k - 1}/M'_{k} = F_{k + 1}$ and $M'_{k}/M_{k + 1} = F_{k}$. The filtration obtained is then the Harder--Narasimhan filtration of $M$ by $\mathcal{G}_{2}$, since we have $F_{k + 1} \in \Filt(B_{i + 1})$, $F_{k} \in \Filt(B_{i})$ with $F_{k + 1}$ occurring before $F_{k}$ in the filtration. This has the same factors as the Harder--Narasimhan filtration of $M$ by $\mathcal{G}_{1}$, so that $\hnf{\mathcal{G}_1}{M} = \hnf{\mathcal{G}_2}{M}$, as desired.

Finally, we show that \ref{op:equiv:hnb} implies \ref{op:equiv:sq}. Note first that this implies that $\bricks{\mathcal{G}_{1}} = \bricks{\mathcal{G}_{2}}$ by Lemma~\ref{lem:same_bricks}. We now show that one can deform $\mathcal{G}_{1}$ across squares to obtain $\mathcal{G}_{2}$ by swapping adjacent bricks. By Lemma~\ref{lem:adj_bricks}, we have adjacent bricks $B < B'$ in $\mathcal{G}_{1}$ which are ordered $B' < B$ in $\mathcal{G}_{2}$. Since both sequences are backwards $\Hom$-orthogonal, we have \[\Hom_{\Lambda}(B, B') = \Hom_{\Lambda}(B', B) = 0.\] If we have $\Ext^{1}_{\Lambda}(B, B') \neq 0$, then, by Lemma~\ref{lem:extension_of_bricks}, we have that there is a brick $M$ which is given by a non-split extension $0 \to B' \to M \to B \to 0$. But then, by uniqueness of Harder--Narasimhan filtrations, this short exact sequence must give the Harder--Narasimhan filtration of $M$ with respect to~$\mathcal{G}_{1}$. However, the Harder--Narasimhan filtration of $M$ with respect to $\mathcal{G}_{2}$ cannot have the same bricks, otherwise the filtration gives a non-split extension $0 \to B \to M \to B' \to 0$. This gives an endomorphism $M \to B' \to M$ of $M$ which is neither zero nor an isomorphism, contradicting the fact that $M$ is a brick. This contradicts the assumption that the Harder--Narasimhan filtrations with respect to $\mathcal{G}_{1}$ and $\mathcal{G}_{2}$ must have the same multisets of bricks.

Hence, we have that $\Ext_\Lambda^1(B, B') = 0$, so by Lemma~\ref{lem:sq_crit_brick}, we can deform $\mathcal{G}_1$ across a square by swapping $B$ and $B'$. Since the number of pairs of bricks ordered differently in $\mathcal{G}_1$ and $\mathcal{G}_2$ was finite and decreases under each such swap, by repeating this process, we deform $\mathcal{G}_1$ into $\mathcal{G}_2$ across a sequence of squares in finitely many steps.  
\end{proof}

\begin{remark} \label{rem:cluster}
Note that, if  $\Lambda$ is a Jacobian algebra of a quiver with potential,
there is a cluster algebra associated to $\Lambda$. Its cluster variables will be in bijection with the reachable indecomposable presilting complexes in $\twoterm$ \cite{air,by,ck06}. Condition~\ref{op:equiv:summ} in Theorem~\ref{thm:mg_equiv} means that maximal green sequences are equivalent in the sense of Definition~\ref{def:equivalence} if and only if the corresponding sets of cluster variables appearing in the clusters in the mutation sequences coincide.
\end{remark}

\begin{remark} \label{rem:DT}
It is shown in \cite{reineke_poisson} that, given a path algebra of a simply-laced Dynkin diagram, maximal green sequences correspond to products of quantum dilogarithms and that the value of this product is in fact independent of the maximal green sequence chosen. These products of quantum dilogarithms give so-called ``refined Donaldson--Thomas invariants''. See also \cite{kel-green, dem-kel} for more general statements and further references. Maximal green sequences which are related by deformation across a square correspond to products of quantum dilogarithms which differ by swapping two adjacent commuting terms in the product. Hence, maximal green sequences which are equivalent are related by finitely many such swaps. It is natural to consider such products as the same. This is analogous to considering reduced words of longest elements of Weyl groups up to commutation, as we will explore in a sequel paper \cite{mgsii}.
\end{remark}

Thanks to Theorem~\ref{thm:mg_equiv}, Lemma~\ref{lem:same_bricks} implies the following.

\begin{corollary}
If $\mathcal{G} \sim \mathcal{G}'$ are equivalent maximal green sequences, then we have that $\bricks{\mathcal{G}} = \bricks{\mathcal{G}'}$.
\end{corollary}

 The converse is false, as is shown in the following example. Maximal green sequences which have the same set of bricks are not necessarily equivalent. Hence, our notion of equivalence does not coincide with the notion of ``weak equivalence'' from \cite[p.257]{Qiu15}.

\begin{example}\label{ex:brick_counter}
Consider the path algebra of the quiver \[1 \leftarrow 2 \rightarrow 3,\] where we use the convention of composing arrows as if they were functions, so that $\xrightarrow{\alpha}\xrightarrow{\beta} = \beta\alpha$.
The Auslander--Reiten quiver of this algebra is 
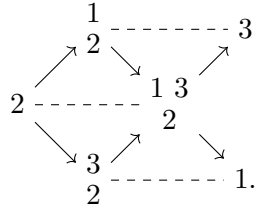
\[
\begin{tikzpicture}


\node(2) at (0,1) {$\tcs{2}$};
\node(32) at (1,0) {$\tcs{3\\2}$};
\node(12) at (1,2) {$\tcs{1\\2}$};
\node(123) at (2,1)	{$\tcs{1~3\\2}$};
\node(1) at (3,0) {$\tcs{1}$.};
\node(3) at (3,2) {$\tcs{3}$};


\draw[->] (2) -- (12);
\draw[->] (2) -- (32);
\draw[->] (32) -- (123);
\draw[->] (12) -- (123);
\draw[->] (123) -- (1);
\draw[->] (123) -- (3);


\draw[dashed] (32) -- (1);
\draw[dashed] (2) -- (123);
\draw[dashed] (12) -- (3);

\end{tikzpicture}
\]
The lattice of its two-term silting complexes is shown in Figure~\ref{fig:silting_counter} and the lattice of its torsion classes is shown in Figure~\ref{fig:brick_label_counter}. Consider the two maximal green sequences given by the maximal backward $\Hom$-orthogonal sequences of bricks \[\tcs{2},\, \tcs{1\\2},\, \tcs{1},\, \tcs{3\\2},\, \tcs{3}\] and \[\tcs{2},\, \tcs{3\\2},\, \tcs{3},\, \tcs{1\\2},\, \tcs{1}.\] These two maximal green sequences have the same set of bricks, but inspection shows that one cannot transform one into the other by deforming across squares. Indeed, these two maximal green sequences have different sets of summands, namely \[\left\{\tcs{1\\2},\, \tcs{2},\, \tcs{3\\2},\, \tcs{1~3\\2},\, \tcs{3},\, \tcs{1\\2}[1],\, \tcs{2}[1],\, \tcs{3\\2}[1] \right\}\] and \[\left\{\tcs{1\\2},\, \tcs{2},\, \tcs{3\\2},\, \tcs{1~3\\2},\, \tcs{1},\, \tcs{1\\2}[1],\, \tcs{2}[1],\, \tcs{3\\2}[1] \right\}.\]
\end{example}

\begin{figure}
    \caption{The lattice of two-term silting complexes from Example~\ref{ex:brick_counter}}
    \label{fig:silting_counter}
    \[
\begin{tikzpicture}[xscale=1.5,yscale=2]


\node(12p2p32) at (0,0) {$\tcs{1\\2} \oplus \tcs{2} \oplus \tcs{3\\2}$};

\node(121p2p32) at (-3,-1) {$\tcs{1\\2}[1] \oplus \tcs{2} \oplus \tcs{3\\2}$};
\node(12p123p32) at (0,-1) {$\tcs{1\\2} \oplus \tcs{1~3\\2} \oplus \tcs{3\\2}$};
\node(12p2p321) at (3,-1) {$\tcs{1\\2} \oplus \tcs{2} \oplus \tcs{3\\2}[1]$};

\node(3p123p32) at (-1.5,-2) {$\tcs{3} \oplus \tcs{1~3\\2} \oplus \tcs{3\\2}$};
\node(12p123p1) at (1.5,-2) {$\tcs{1\\2} \oplus \tcs{1~3\\2} \oplus \tcs{1}$};

\node(121p3p32) at (-3,-3) {$\tcs{1\\2}[1] \oplus \tcs{3} \oplus \tcs{3\\2}$};
\node(3p123p1) at (0,-3) {$\tcs{3} \oplus \tcs{1~3\\2} \oplus \tcs{1}$};
\node(321p12p1) at (3,-3) {$\tcs{3\\2}[1] \oplus \tcs{1\\2} \oplus \tcs{1}$};

\node(3p21p1) at (0,-4) {$\tcs{3} \oplus \tcs{2}[1] \oplus \tcs{1}$};

\node(121p3p21) at (-3,-5) {$\tcs{1\\2}[1] \oplus \tcs{3} \oplus \tcs{2}[1]$};
\node(121p2p321) at (0,-5) {$\tcs{1\\2}[1] \oplus \tcs{2} \oplus \tcs{3\\2}[1]$};
\node(321p21p1) at (3,-5) {$\tcs{3\\2}[1] \oplus \tcs{2}[1] \oplus \tcs{1}$};

\node(121p21p321) at (0,-6) {$\tcs{1\\2}[1] \oplus \tcs{2}[1] \oplus \tcs{3\\2}[1]$.};


\draw[->] (12p2p32) -- (121p2p32);
\draw[->] (12p2p32) -- (12p123p32);
\draw[->] (12p2p32) -- (12p2p321);

\draw[->] (12p123p32) -- (3p123p32);
\draw[->] (12p123p32) -- (12p123p1);

\draw[->] (121p2p32) -- (121p3p32);
\draw[->] (12p2p321) -- (321p12p1);
\draw[->] (121p2p32) -- (121p2p321);
\draw[->] (12p2p321) -- (121p2p321);

\draw[->] (3p123p32) -- (3p123p1);
\draw[->] (3p123p32) -- (121p3p32);
\draw[->] (12p123p1) -- (3p123p1);
\draw[->] (12p123p1) -- (321p12p1);

\draw[->] (3p123p1) -- (3p21p1);

\draw[->] (121p3p32) -- (121p3p21);
\draw[->] (321p12p1) -- (321p21p1);

\draw[->] (3p21p1) -- (121p3p21);
\draw[->] (3p21p1) -- (321p21p1);

\draw[->] (121p3p21) -- (121p21p321);
\draw[->] (121p2p321) -- (121p21p321);
\draw[->] (321p21p1) -- (121p21p321);

\end{tikzpicture}
\]
\end{figure}
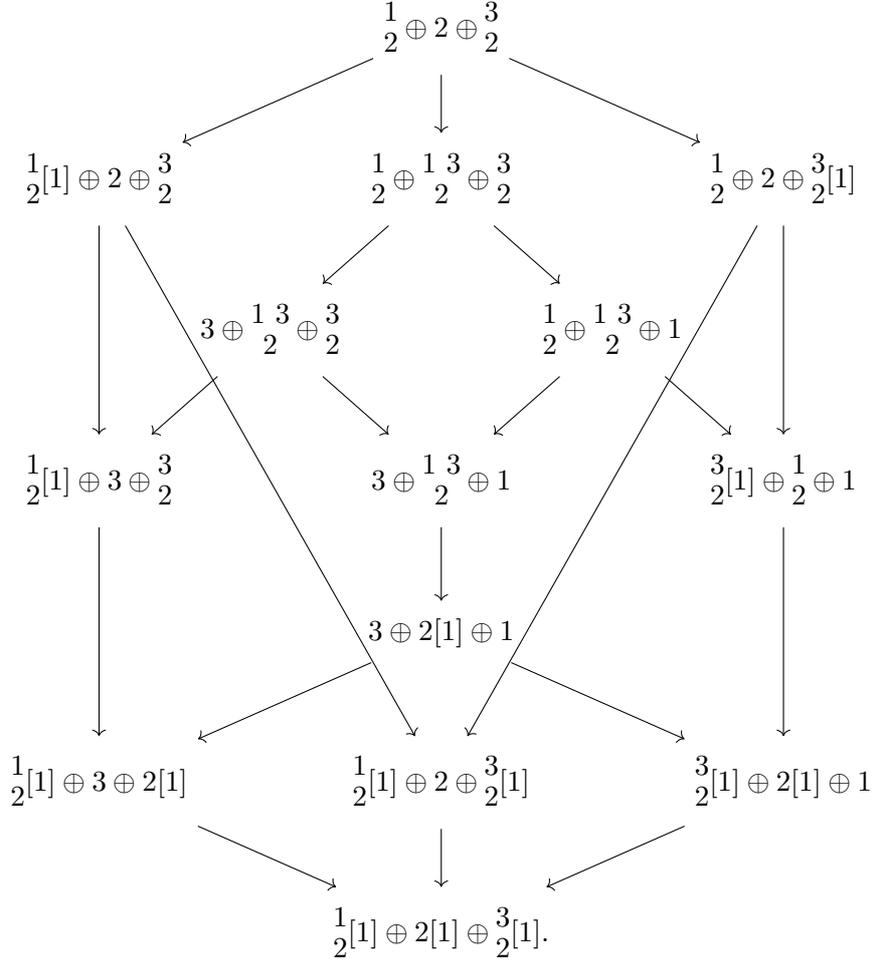

\begin{figure}
    \caption{The lattice of torsion classes from Example~\ref{ex:brick_counter} with brick labelling}
    \label{fig:brick_label_counter}
    \[
\begin{tikzpicture}[xscale=1.5,yscale=3.1]


\node(12p2p32) at (0,0) {$\additive \left\{\tcs{1\\2},\, \tcs{1~3\\2},\, \tcs{3\\2},\, \tcs{1},\, \tcs{2},\, \tcs{3}\right\}$};

\node(121p2p32) at (-3,-1) {$\additive \left\{\tcs{2},\, \tcs{3\\2},\, \tcs{3}\right\}$};
\node(12p123p32) at (0,-1) {$\additive \left\{\tcs{1\\2},\, \tcs{1~3\\2},\, \tcs{3\\2},\, \tcs{1},\, \tcs{3}\right\}$};
\node(12p2p321) at (3,-1) {$\additive \left\{\tcs{2},\, \tcs{1\\2},\, \tcs{1}\right\}$};

\node(3p123p32) at (-1.5,-2) {$\additive \left\{\tcs{1~3\\2},\, \tcs{3\\2},\, \tcs{1},\, \tcs{3}\right\}$};
\node(12p123p1) at (1.5,-2) {$\additive \left\{\tcs{1\\2},\, \tcs{1~3\\2},\, \tcs{1},\, \tcs{3}\right\}$};

\node(121p3p32) at (-3,-3) {$\additive \left\{\tcs{3\\2},\, \tcs{3}\right\}$};
\node(3p123p1) at (0,-3) {$\additive \left\{\tcs{1~3\\2},\, \tcs{1},\, \tcs{3}\right\}$};
\node(321p12p1) at (3,-3) {$\additive \left\{\tcs{1\\2},\, \tcs{1}\right\}$};

\node(3p21p1) at (0,-4) {$\additive \left\{\tcs{1},\, \tcs{3}\right\}$};

\node(121p3p21) at (-3,-5) {$\additive \left\{\tcs{3}\right\}$};
\node(121p2p321) at (0,-5) {$\additive \left\{\tcs{2}\right\}$};
\node(321p21p1) at (3,-5) {$\additive \left\{\tcs{1}\right\}$};

\node(121p21p321) at (0,-6) {$\{0\}$.};


\draw[->] (12p2p32) -- (121p2p32) node [midway,fill=white] {$\tcs{1}$};
\draw[->] (12p2p32) -- (12p123p32) node [midway,fill=white] {$\tcs{2}$};
\draw[->] (12p2p32) -- (12p2p321) node [midway,fill=white] {$\tcs{3}$};

\draw[->] (12p123p32) -- (3p123p32) node [midway,fill=white] {$\tcs{1\\2}$};
\draw[->] (12p123p32) -- (12p123p1) node [midway,fill=white] {$\tcs{3\\2}$};

\draw[->] (121p2p32) -- (121p3p32) node [midway,fill=white] {$\tcs{2}$};
\draw[->] (12p2p321) -- (321p12p1) node [midway,fill=white] {$\tcs{2}$};
\draw[->] (121p2p32) -- (121p2p321) node [midway,fill=white] {$\tcs{3}$};
\draw[->] (12p2p321) -- (121p2p321) node [midway,fill=white] {$\tcs{1}$};

\draw[->] (3p123p32) -- (3p123p1) node [midway,fill=white] {$\tcs{3\\2}$};
\draw[->] (3p123p32) -- (121p3p32) node [midway,fill=white] {$\tcs{1}$};
\draw[->] (12p123p1) -- (3p123p1) node [midway,fill=white] {$\tcs{1\\2}$};
\draw[->] (12p123p1) -- (321p12p1) node [midway,fill=white] {$\tcs{3}$};

\draw[->] (3p123p1) -- (3p21p1) node [midway,fill=white] {$\tcs{1~3\\2}$};

\draw[->] (121p3p32) -- (121p3p21) node [midway,fill=white] {$\tcs{3\\2}$};
\draw[->] (321p12p1) -- (321p21p1) node [midway,fill=white] {$\tcs{1\\2}$};

\draw[->] (3p21p1) -- (121p3p21) node [midway,fill=white] {$\tcs{1}$};
\draw[->] (3p21p1) -- (321p21p1) node [midway,fill=white] {$\tcs{3}$};

\draw[->] (121p3p21) -- (121p21p321) node [midway,fill=white] {$\tcs{3}$};
\draw[->] (121p2p321) -- (121p21p321) node [midway,fill=white] {$\tcs{2}$};
\draw[->] (321p21p1) -- (121p21p321) node [midway,fill=white] {$\tcs{1}$};

\end{tikzpicture}
\]
\end{figure}

Because of Lemma~\ref{lem:same_bricks}, one can regard having the same stable factors as an augmentation of the condition of having the same bricks. This augmentation is required to determine the equivalence class of the maximal green sequence.

In some ways it is surprising that maximal green sequences may have the same set of bricks whilst having different sets of $\tau$-rigids. Theorem~\ref{thm:air_relproj} shows that the $\tau$-rigids form the relative projectives of the torsion classes, whilst Theorem~\ref{thm:enomoto_relsimp} shows that the bricks form the relative simples of the torsion classes. Hence there is no duality between simples and projectives on the level of maximal green sequences. It is known from \cite[Example~6.17]{enomoto_jhp} that, in general, there is no duality between simples and projectives within torsion classes --- see also \cite[Theorem~5.10]{enomoto_jhp}. But it is still not obvious why the bricks of a maximal green sequence should give less information than the $\tau$-rigids.

As shown in \cite[Corollary~5.15]{enomoto_jhp}, the existence of the duality between simples and projectives --- at least in the sense of having the same number of simples as indecomposable projectives --- in a functorially finite torsion class $\mathcal{T}$ in $\mod \Lambda$, for an Artin algebra $\Lambda$, is equivalent to the Jordan--H\"older property (JHP) for this torsion class considered as an exact category. Here the exact structure is the one induced by the embedding $\mathcal{T} \hookrightarrow \mod \Lambda$. Enomoto shows that functorially finite torsion classes of Nakayama algebras possess the JHP \cite[Corollary~5.19]{enomoto_jhp}; we show in Section~\ref{sect:nak} that for these algebras the set of bricks does determine the equivalence class of the maximal green sequence.

One may wonder whether the set of bricks determines the equivalence class of the maximal green sequence if and only if every functorially finite torsion class in $\mod \Lambda$ satisfies the JHP. \cite[Example~2.9, Example~4.17]{enomoto_bruhat} shows that the ``only if'' part cannot be true. Namely, the example in \cite{enomoto_bruhat} gives a torsion class in the preprojective algebra of type $D_4$ which does not satisfy the JHP. On the other hand, in a sequel to this paper \cite{mgsii}, we will show that the converse of Lemma~\ref{lem:same_bricks} holds for all preprojective algebras of Dynkin type. Hence, there are examples where the converse of Lemma~\ref{lem:same_bricks} holds whilst there exist functorially finite torsion classes for which the JHP fails. We conjecture that the ``if'' part is true, motivated by the case of Nakayama algebras.

\begin{conjecture}
Let $\Lambda$ be an Artin algebra such that every functorially finite torsion class in $\mod \Lambda$ satisfies the JHP. Then two maximal green sequences $\mathcal{G}$ and $\mathcal{G}'$ have the same set of bricks if and only they are equivalent.
\end{conjecture}

Note that the JHP does not hold for some functorially finite torsion classes of the non-linearly oriented $A_3$ algebra considered in Example~\ref{ex:brick_counter}. Indeed, in the torsion class \[\additive \left\{\tcs{1\\2},\, \tcs{1\phantom{2}3\\\phantom{1}2\phantom{3}},\, \tcs{3\\2},\, \tcs{1},\, \tcs{3}\right\}\] the brick $\tcss{1\phantom{2}3\\\phantom{1}2\phantom{3}}$ has two different composition series: one with factors \[\tcs{3} \text{ and } \tcs{1\\2},\] and the other with factors \[\tcs{1} \text{ and } \tcs{3\\2}.\] Consequently, there is no bijection between relative projectives and relative simples. There are three relative projectives \[\tcs{1\\2},\, \tcs{1\phantom{2}3\\\phantom{1}2\phantom{3}}, \text{ and } \tcs{3\\2},\] while there are four relative simples \[\tcs{1\\2},\, \tcs{3\\2},\, \tcs{1}, \text{ and } \tcs{3}.\]

\begin{remark}
Another plausible definition of equivalence of maximal green sequences which fails to coincide with those of Theorem~\ref{thm:mg_equiv} is as follows. Recall from Section~\ref{sect:back:silt_notion} that a maximal green sequence can be specified as the sequence of vertices mutated at. One might conjecture that two maximal green sequences are equivalent if and only if the underlying multisets of these sequences coincide. It is clear that these multisets are preserved by deformation across squares. However, these multisets may coincide for inequivalent maximal green sequences. For instance, the maximal green sequence down the left-hand side of Figure~\ref{fig:silting_counter} has sequence \[1, 2, 3, 2,\] while the sequence down the right-hand side has sequence \[3, 2, 1, 2.\] The underlying multisets are the same here, but the maximal green sequences are not equivalent.
\end{remark}

\section{Partial orders on equivalence classes}\label{sect:partial_order}

The equivalence relation on maximal green sequences defined in the previous section reveals more structure on the set of maximal sequences. Indeed, there are natural partial orders on the equivalence classes of maximal green sequences, analogous to the partial orders on silting complexes. There are three such partial orders, which we study in this section. The first uses deformations across oriented polygons; the second uses reverse-inclusion of summands; the third uses refinement of Harder--Narasimhan filtrations. We show that the first order implies the second and the third, and we conjecture that the three actually coincide, analogously to how the two partial orders on silting complexes have the same Hasse diagram.

\subsection{Deformations across oriented polygons}

The first of these partial orders has covering relations given by deformations across oriented polygons.

\begin{definition}\label{def:iepd}
Let $\mathcal{G}$ and $\mathcal{G}'$ be maximal green sequences of $\Lambda$. If $\mathcal{G}$ and $\mathcal{G}'$ only differ in that $\mathcal{G}$ contains the path of length greater than two around an oriented polygon, whilst $\mathcal{G}'$ contains the length-two path, then we say that $\mathcal{G}'$ is an \emph{increasing elementary polygonal deformation} of $\mathcal{G}$. By extension, we also say that $[\mathcal{G}']$ is an increasing elementary polygonal deformation of $[\mathcal{G}]$. Similarly, we say that $\mathcal{G}$ is a decreasing elementary polygonal deformation of $\mathcal{G}'$ and that $[\mathcal{G}]$ is a decreasing elementary polygonal deformation of $[\mathcal{G}']$.
\end{definition}

Note that an increasing elementary polygonal deformation \textit{decreases} the length of the maximal green sequence. One can think of it instead as increasing the speed of the maximal green sequence.

\begin{figure}
    \caption{An increasing elementary polygonal deformation, where $X$ and $A$ are indecomposable}\label{fig:deformation}
    \[
    \begin{tikzpicture}[scale=1.5,xscale=1.3]

\node (start) at (-1,0) {$\dots$};
\node (exx') at (0.25,0) {$E \oplus X \oplus A$};
\node (eyx') at (2,1) {$E \oplus Z \oplus A$};
\node (eyy') at (3.75,0) {$E \oplus Z \oplus C$};
\node (exz') at (0.75,-1) {$E \oplus X \oplus C'$};
\node (ezy') at (3.25,-1) {$E \oplus X' \oplus C$};
\node (dots) at (2,-1) {$\dots$};
\node (end) at (5,0) {$\dots$};

\draw[->] (start) -- (exx');
\draw[->] (exx') -- (eyx');
\draw[->] (eyx') -- (eyy');
\draw[->] (exx') -- (exz');
\draw[->] (exz') -- (dots);
\draw[->] (dots) -- (ezy');
\draw[->] (ezy') -- (eyy');
\draw[->] (eyy') -- (end);

\end{tikzpicture}
    \]
\end{figure}

\subsubsection{Deformations across oriented polygons in terms of bricks}

Just as we interpreted deformations across squares in terms of maximal backwards $\Hom$-orthogonal sequences of bricks, we also wish to do the same for increasing elementary polygonal deformations.

\begin{lemma}\label{lem:polygon_bricks}
Let $\Lambda$ be a finite-dimensional algebra over a field $K$. Suppose that \[
\begin{tikzcd}
&& \mathcal{V} \ar[drr,"M"] && \\
\mathcal{T} \ar[urr,"L"] \ar[dr,"B_1"] &&&& \mathcal{U} \\
& \mathcal{V}'_{1} \ar[r,"B_2"] & \dots \ar[r,"B_{r-1}"] & \mathcal{V}'_{r} \ar[ur,"B_r"] &
\end{tikzcd}
\] is an oriented polygon in $\fftors \Lambda$ with its brick labels. Then $B_1 \cong M$ and $B_r \cong L$.
\end{lemma}
\begin{proof}
We have that $[\mathcal{T}, \mathcal{U}] = \Filt(L, M) = \Filt(B_1, B_2, \dots, B_r)$, recalling the notation $[\mathcal{T}, \mathcal{U}]$ from Section~\ref{sect:back:brick_label}. We then have that $L$ and $M$ must be precisely the relatively simple objects of $[\mathcal{T}, \mathcal{U}]$, since neither brick can admit a filtration with the other as factors without violating the brick condition or the backwards $\Hom$-orthogonality condition.

It then follows from Proposition~\ref{prop:simples_in_intervals} that $L$ and $M$ must occur as elements of the set $\{B_1, B_2, \dots, B_r\}$. It is clear that $L$ cannot occur before $M$ amongst these bricks, otherwise $L, M$ cannot be a maximal backwards $\Hom$-orthogonal sequence of bricks in $[\mathcal{T},\mathcal{U}]$. If we do not have $B_1 \cong M$ and $B_r \cong L$, then we also get a contradiction to backwards $\Hom$-orthogonality, since all of $B_1, B_2, \dots, B_r$ have filtrations with $L$ and $M$ as factors.
\end{proof}

One can view an increasing elementary polygonal deformation as swapping maximal green sequences in an abelian category with two simple objects. Such a category has at most two maximal green sequences up to equivalence, and in our case it has precisely two, where one has length two and the other is longer.

\begin{lemma}\label{lem:poly_ab_cat}
In the situation of the polygon from Lemma~\ref{lem:polygon_bricks}, we have the following.
\begin{enumerate}[label=\textup{(}\arabic*\textup{)}]
    \item $\Filt(L, M)$ is an abelian category.\label{op:poly_ab_cat:ab_cat}
    \item The two paths around the polygon give two maximal green sequences of $\Filt(L, M)$.\label{op:poly_ab_cat:mgs}
    \item The two paths around the polygon give two different sets of Harder--Narasimhan filtrations of $\Filt(L, M)$.\label{op:poly_ab_cat:hn}
\end{enumerate}
\end{lemma}
\begin{proof}
We have $\Hom_{\Lambda}(L, M) = \Hom_{\Lambda}(M, L) = 0$ by backwards $\Hom$-orthogon\-al\-ity of maximal green sequences going through the sides of the polygon. Then \ref{op:poly_ab_cat:ab_cat} follows from \cite[1.2]{ringel_rksb}, which states that $\Filt(-)$ of a set of $\Hom$-orthogonal bricks is an abelian category. For \ref{op:poly_ab_cat:mgs}, it follows from the definition of brick labels that the sequences of bricks given by the two paths around the polygon must be backwards $\Hom$-orthogonal in $\Filt(L, M)$. For \ref{op:poly_ab_cat:hn}, the fact that the two paths around the polygon give Harder--Narasimhan filtrations on $\Filt(L, M)$ then follows from \ref{op:poly_ab_cat:mgs} by \cite{igusa_lsc,treff_hn}. The fact that the two sets of Harder--Narasimhan filtrations must be different then follows from the fact that in the longer path around the polygon, $B_2$ is the only factor in its filtration, whilst this cannot be the case for the shorter path around the polygon.
\end{proof}

The interpretation of increasing elementary polygonal deformations is then as follows.

\begin{lemma}\label{lem:iepd_via_bricks}
A maximal green sequence $\mathcal{G}'$ is an increasing elementary polygonal deformation of a maximal green sequence $\mathcal{G}$ if and only if, as maximal backwards $\Hom$-orthogonal sequences of bricks, we have that $\mathcal{G}'$ is \[B_1, B_2, \dots, B_r\] whilst $\mathcal{G}$ is \[B_1, \dots, B_{i - 1}, B_{i + 1}, B'_1, \dots, B'_s, B_i, B_{i + 2} \dots, B_r\] for $s \geqslant 1$.
\end{lemma}
\begin{proof}
The forwards direction is Lemma~\ref{lem:polygon_bricks}. For the backwards direction, let $T$ be the basic two-term silting complex corresponding to the torsion class $\Tors{B_i, B_{i + 1}, \dots, B_r}$ and $T'$ be the basic two-term silting complex corresponding to the torsion class $\Tors{B_{i + 2}, \dots, B_r}$. Then $T'$ is obtained from $T$ by two green mutations since the corresponding torsion classes differ by two minimal inclusions. Hence $T \cong E \oplus X \oplus X'$ and $T' \cong E \oplus Y \oplus Y'$ where $X$, $X'$, $Y$, and $Y'$ are all indecomposable. Then, by Lemma~\ref{lem:iepd_unique}, the only other path from $T$ to $T'$ in $\twosilt \Lambda$ is the other path around the polygon determined by $E$. This must be the path taken in $\mathcal{G}.$ Since $s \geqslant 1$, the polygon determined by $E$ is oriented. Hence $\mathcal{G}$ and $\mathcal{G}'$ only differ in that $\mathcal{G}'$ contains the length two path around the oriented polygon whilst $\mathcal{G}$ contains the longer path, so $\mathcal{G}'$ is an increasing elementary polygonal deformation of $\mathcal{G}$. 
\end{proof}

\subsection{Partial orders}

We can now define the three partial orders on equivalence classes of maximal green sequences.

\begin{definition}\label{def:partial_orders}
\begin{enumerate}
    \item The partial order $\dleq$ on equivalence classes of maximal green sequences is defined via its covering relations, which are that $[\mathcal{G}] \dlessdot [\mathcal{G}']$ if and only if $[\mathcal{G}']$ is an increasing elementary polygonal deformation of $[\mathcal{G}]$. We refer to this as the \emph{deformation partial order}.
    \item The partial order $\sleq$ is defined via $[\mathcal{G}] \sleq [\mathcal{G}']$ if and only if $\summ{\mathcal{G}} \supseteq \summ{\mathcal{G}'}$. This is evidently also equivalent to having $\sumt{\mathcal{G}} \supseteq \sumt{\mathcal{G}'}$. We refer to this as the \emph{summand partial order}.
    \item The partial order $\hleq$ is defined by $[\mathcal{G}] \hleq [\mathcal{G}']$ if and only if, for any module $M$, we have \[\bhnf{\mathcal{G}'}{M} = \bigsqcup_{B \in \bhnf{\mathcal{G}}{M}}\bhnf{\mathcal{G}'}{B}.\] If $[\mathcal{G}] \hleq [\mathcal{G}']$, then we say that the $\mathcal{G}'$-HN filtrations \emph{refine} the $\mathcal{G}$-HN filtrations. Informally, $[\mathcal{G}] \hleq [\mathcal{G}']$ if the $\mathcal{G}'$-stable factors of any $\Lambda$-module $M$ can be obtained by breaking up the $\mathcal{G}$-stable factors of $M$ into \textit{their} $\mathcal{G}'$-stable factors. We refer to this as the \emph{Harder--Narasimhan} or \emph{HN partial order}.\label{op:partial_orders:hn}
\end{enumerate}
\end{definition}

\begin{remark}\label{rmk:silt_orders}
As shown in \cite[Theorem~4.3.1, Theorem~4.4.4]{njw-phd} \cite[Theorem~3.4, Theorem~5.6]{njw-hst}, the deformation order here should be seen as a higher-dimensional incarnation of the order on silting complexes given by green mutation, whilst the summand partial order should be seen as a higher-dimensional incarnation of the order on silting complexes given by inclusion of aisles \cite[Definition~2.10, Theorem~2.11]{ai}. These orders are known to have the same Hasse diagram \cite[Theorem~2.35]{ai}. Analogous orders exist on tilting modules \cite[2.2]{rs-simp} and support $\tau$-tilting pairs \cite[Section~2.4]{air}, which are likewise known to have the same Hasse diagram \cite[Theorem~2.1]{hu-potm} \cite[Theorem~2.33]{air}. The Harder--Narasimhan order is new and does not have an analogue on silting complexes.

The other new feature here is of course that one must introduce the equivalence relation on maximal green sequences in order to see the partial orders. Note that a partial order on equivalence classes is exactly the same thing as a preorder, so one could instead consider preorders on maximal green sequences. We prefer to keep the equivalence relation and the partial orders conceptually separate.
\end{remark}

\begin{remark}
The reason why we must use the stable factors rather than the semistable factors to define the Harder--Nara\-sim\-han order is as follows. The plausible alternative definition using the semistable factors would be that $[\mathcal{G}] \hleq [\mathcal{G}']$ if and only if for all $\Lambda$-modules $M$, we have that \[\hnf{\mathcal{G}'}{M} = \bigsqcup_{F \in \hnf{\mathcal{G}}{M}} \hnf{\mathcal{G}'}{F}.\] Consider the path algebra of the $A_2$ quiver $1 \leftarrow 2$. This has two maximal green sequences, namely $\mathcal{G}$ given by the sequence of bricks \[\tcs{2}, \, \tcs{1\\2}, \, \tcs{1},\] and $\mathcal{G}'$ given by the sequence of bricks \[\tcs{1}, \, \tcs{2}.\] Then we have that $[\mathcal{G}] \hleq [\mathcal{G}']$, but we have
\begin{align*}
    \hnfbig{\mathcal{G}'}{2 \oplus \tcss{1\\2}} &= \left\{1, 2 \oplus 2\right\} \\
    &\neq \left\{1, 2, 2\right\} \\
    &= \hnfbig{\mathcal{G}'}{\tcss{1\\2}} \sqcup \hnf{\mathcal{G}'}{2} \\
    &= \bigsqcup_{F \in \hnfbig{\mathcal{G}}{2 \oplus \tcss{1\\2}}} \hnf{\mathcal{G}'}{F}.
\end{align*}
In order for this to work correctly, we need to break up $2 \oplus 2$ into $2, 2$ by considering stable factors rather than semistable factors. 
\end{remark}

\begin{remark}\label{rmk:hn_filt_cant_stick}
Given two maximal green sequences $\mathcal{G}$ and $\mathcal{G}'$ such that $[\mathcal{G}] \hleq [\mathcal{G}']$, one might wonder whether, up to equivalence, the $\mathcal{G}'$-HN filtration of a $\Lambda$-module $M$ can be obtained from the $\mathcal{G}$-HN filtration of $M$ by breaking up the $\mathcal{G}$-semistable factors according to their $\mathcal{G}'$-HN filtrations, without doing any rearranging of the orders of the factors. To be more precise, suppose that, up to equivalence, the $\mathcal{G}$-HN filtration of $M$ is \[M = M_0 \supset M_1 \supset \dots \supset M_{l - 1} \supset M_l = 0\] with $F_{i} := M_{i - 1}/M_{i}$, and suppose that the $\mathcal{G}'$-HN filtration of each $F_i$ is \[F_i = L_{i0} \supset L_{i1} \supset \dots \supset L_{il_{i}} = 0\] with $H_{ij} := L_{i(j - 1)}/L_{ij}$. One might hope that, again up to equivalence, the $\mathcal{G}'$-HN filtration of $M$ is 
\begin{align*}
M = M_{10} &\supset M_{11} \supset \dots M_{1(l_1 - 1)} \supset M_{20} \supset \dots \\ &\supset M_{i0} \supset \dots \supset M_{i(l_i - 1)} \supset \dots \supset M_{l(l_{l} - 1)} = 0    
\end{align*}
where $M_{i(j - 1)}/M_{ij} = H_{ij}$ for $1 \leqslant i \leqslant l$ and $0 \leqslant j \leqslant l_i - 2$ and $M_{i(l_i - 1)}/M_{(i + 1)0} = H_{il_i}$ for $1 \leqslant i \leqslant l - 1$.

Unfortunately, this is not generally true. In general, one has to reorganise the semistable factors $H_{ij}$ to obtain the $\mathcal{G}'$-HN filtration, even if one replaces $\mathcal{G}$ and $\mathcal{G}'$ by equivalent maximal green sequences. This is shown in the following example. Consider the path algebra of the following algebra of type $\widetilde{A}_{4}$. 
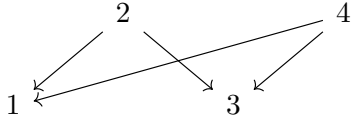
\[
\begin{tikzcd}
& 2 \ar[dr] \ar[dl] && 4 \ar[dl] \ar[dlll] \\
1 && 3 &
\end{tikzcd}
\] This algebra has a maximal green sequence $\mathcal{G}$ given by the sequence of bricks \[1,\, \tcs{2\\1},\, \tcs{4\\1},\, \tcs{2\phantom{1}4\\\phantom{2}1\phantom{4}},\, \tcs{3},\, \tcs{2\\3},\, \tcs{2},\, \tcs{4\\3},\, \tcs{4}\] and a maximal green sequence $\mathcal{G}'$ given by the sequence of bricks \[4,\, 2,\, 3,\, 1.\] Since the bricks of $\mathcal{G}'$ are precisely the simple modules, it is clear that we have $[\mathcal{G}] \hleq [\mathcal{G}']$. Now consider the module \[M = \tcs{\phantom{1}2\phantom{3}4\\1\phantom{2}3\phantom{4}}.\] The $\mathcal{G}$-HN filtration of $M$ is \[\tcs{\phantom{1}2\phantom{3}4\\1\phantom{2}3\phantom{4}} \supset \tcs{\phantom{3}4\\3\phantom{4}} \supset 0,\] with $\hnf{\mathcal{G}}{M} = \left\{\tcss{\phantom{1}2\\1\phantom{2}},\,\tcss{\phantom{3}4\\3\phantom{4}}\right\}$, whilst the $\mathcal{G}'$-HN filtration of $M$ is \[\tcs{\phantom{1}2\phantom{3}4\\1\phantom{2}3\phantom{4}} \supset \tcs{\phantom{1}2\phantom{3}\\1\phantom{2}3} \supset 1 \oplus 3 \supset 1 \supset 0,\] with $\hnf{\mathcal{G}'}{M} = \{1, 2, 3, 4\}$.

Now, if we were to try to construct the $\mathcal{G}'$-HN filtration of $M$ as above, by sticking together the $\mathcal{G}'$-HN filtration of \[\tcs{2\\1}\] with the $\mathcal{G}'$-HN filtration of \[\tcs{4\\3}\,,\] then the $\mathcal{G}'$-semistable factors would appear in the order $2, 1, 4, 3$, whereas in actuality they appear in the order $4, 2, 3, 1$. Moreover, no amount of deformation across squares for either $\mathcal{G}$ or $\mathcal{G}'$ can change this. Indeed, the $\mathcal{G}$-HN filtration of $M$ is unique in the equivalence class, whilst deformation across squares cannot change the order in which $1$ and $4$ occur in $\mathcal{G}'$, since $\Ext^{1}_{\Lambda}(4, 1) \neq 0$.
\end{remark}

The HN order on maximal green sequences implies inclusion of bricks.

\begin{lemma}\label{lem:hn->bricks}
If $[\mathcal{G}] \hleq [\mathcal{G}']$, then $\bricks{\mathcal{G}} \supseteq \bricks{\mathcal{G}'}$. Furthermore, if $[\mathcal{G}] \hl [\mathcal{G}']$, then $\bricks{\mathcal{G}} \supset\bricks{\mathcal{G}'}$.
\end{lemma}
\begin{proof}
Let $B' \in \bricks{\mathcal{G}'}$, so that $\bhnf{\mathcal{G}'}{B'} = \{B'\}$. Since $[\mathcal{G}] \hleq [\mathcal{G}']$, we have
\begin{align*}
    \{B'\} &= \bhnf{\mathcal{G}'}{B'} \\
    &= \bigsqcup_{B \in \bhnf{\mathcal{G}}{B'}}\bhnf{\mathcal{G}'}{B}.
\end{align*}
This means that we must have $\#\bhnf{\mathcal{G}}{B'} = 1$, so $\bhnf{\mathcal{G}}{B'} = \{B'\}$. 
This gives that $B' \in \bricks{\mathcal{G}}$, as desired.

For the second statement, suppose that we have $[\mathcal{G}] \hl [\mathcal{G}']$, so that, in particular, there exists a module $M$ such that $\bhnf{\mathcal{G}}{M} \neq \bhnf{\mathcal{G}'}{M}$. However, we still have refinement, so that \[\bhnf{\mathcal{G'}}{M} = \bigsqcup_{B \in \bhnf{\mathcal{G}}{M}} \bhnf{\mathcal{G'}}{B}.\] Therefore, we must have some $B \in \bhnf{\mathcal{G}}{M}$ such that $\bhnf{\mathcal{G}'}{B} \neq \{B\}$, otherwise the refinement property would give us that $\bhnf{\mathcal{G}'}{M} = \bhnf{\mathcal{G}}{M}$. But this implies that we cannot have $B \in \bricks{\mathcal{G}'}$. We conclude that $\bricks{\mathcal{G}} \supset \bricks{\mathcal{G}'}$.
\end{proof}

The converse to the statement of Lemma~\ref{lem:hn->bricks} is not true in general: we might have $[\mathcal{G}]$ and $[\mathcal{G}']$ incomparable in the Harder--Narasimhan partial order, but such that $\bricks{\mathcal{G}} \supset \bricks{\mathcal{G}'}$.
See Remark~\ref{rem:incomplete_strategy} for an example in type~$\widetilde{A}_4$.

\begin{proposition}\label{prop:poset_well_def}
\begin{enumerate}[label=\textup{(}\arabic*\textup{)}]
    \item The relation $\dleq$ is a well-defined partial order.\label{op:poset_well_def:dleq}
    \item The relation $\sleq$ is a well-defined partial order.\label{op:poset_well_def:sleq}
    \item The relation $\hleq$ is a well-defined partial order.\label{op:poset_well_def:hleq}
\end{enumerate}
\end{proposition}
\begin{proof}
In each case, what needs to be shown is that the relation respects equivalence classes, and that it is reflexive, anti-symmetric, and transitive.

\ref{op:poset_well_def:dleq} By construction, $\dleq$ respects equivalence classes and is reflexive and transitive. To show that $\dleq$ is anti-symmetric, we must show that there can be no sequence of covering relations which gives a cycle. But this is clear, since an increasing elementary polygonal deformation reduces the length of the maximal green sequence by at least one, and length is invariant under equivalence of maximal green sequences.

\ref{op:poset_well_def:sleq} The relation $\sleq$ is clearly reflexive, anti-symmetric, and transitive. To see that $\sleq$ respects equivalence classes, it suffices to note that the set of summands of a maximal green sequence is invariant under equivalence by Theorem~\ref{thm:mg_equiv}.

\ref{op:poset_well_def:hleq} It follows from Theorem~\ref{thm:mg_equiv} that $\hleq$ respects equivalence classes, since $\bhnf{\mathcal{G}}{M}$ is invariant on the equivalence class of $\mathcal{G}$. 
Reflexivity then follows from the fact that \[\bhnf{\mathcal{G}}{M} = \bigsqcup_{B \in \bhnf{\mathcal{G}}{M}} \bhnf{\mathcal{G}}{B},\] since $\bhnf{\mathcal{G}}{B} = \{B\}$ for $B \in \bricks{\mathcal{G}}$.

We now show that $\hleq$ is transitive. Suppose that there are equivalence classes of maximal green sequences $[\mathcal{G}]$, $[\mathcal{G}']$, and $[\mathcal{G}'']$ such that $[\mathcal{G}] \hleq [\mathcal{G}']$ and $[\mathcal{G}'] \hleq [\mathcal{G}'']$, and again let $M$ be a $\Lambda$-module. Then
\begin{align*}
    \bhnf{\mathcal{G}''}{M} &= \bigsqcup_{B' \in \bhnf{\mathcal{G}'}{M}}\bhnf{\mathcal{G}''}{B'} &&\because [\mathcal{G}'] \hleq [\mathcal{G}''] \\
    &= \bigsqcup_{B' \in \bigsqcup_{B \in \bhnf{\mathcal{G}}{M}}\bhnf{\mathcal{G}'}{B}}\bhnf{\mathcal{G}''}{B'} &&\because [\mathcal{G}] \hleq [\mathcal{G}'] \\
    &= \bigsqcup_{B \in \bhnf{\mathcal{G}}{M}}\bigsqcup_{B' \in \bhnf{\mathcal{G}'}{B}}\bhnf{\mathcal{G}''}{B'} \\
    &= \bigsqcup_{B \in \bhnf{\mathcal{G}}{M}} \bhnf{\mathcal{G}''}{B}. &&\because [\mathcal{G}'] \hleq [\mathcal{G}''] \text{ for } B
\end{align*}
Hence $[\mathcal{G}] \hleq [\mathcal{G}'']$, so $\hleq$ is transitive.

We now show that $\hleq$ is anti-symmetric. Suppose that there are equivalence classes of maximal green sequences $[\mathcal{G}]$ and $[\mathcal{G}']$ such that $[\mathcal{G}] \hleq [\mathcal{G}']$ and $[\mathcal{G}'] \hleq [\mathcal{G}]$, and let $M$ be a $\Lambda$-module. By Lemma~\ref{lem:hn->bricks}, we have $\bricks{\mathcal{G}} \supseteq \bricks{\mathcal{G}'} \supseteq \bricks{\mathcal{G}}$, so $\bricks{\mathcal{G}}  = \bricks{\mathcal{G}'}$. 
Then
\begin{align*}
    \bhnf{\mathcal{G}'}{M} &= \bigsqcup_{B \in \bhnf{\mathcal{G}}{M}}\bhnf{\mathcal{G}'}{B} && \because [\mathcal{G}] \hleq [\mathcal{G}'] \\
    &= \bigsqcup_{B \in \bhnf{\mathcal{G}}{M}}\{B\} && \because \bricks{\mathcal{G}} = \bricks{\mathcal{G}'} \\
    &= \bhnf{\mathcal{G}}{M}.
\end{align*}
Thus $[\mathcal{G}] = [\mathcal{G}']$, by Theorem~\ref{thm:mg_equiv}, so $\hleq$ is anti-symmetric.
\end{proof}

\subsubsection{Comparing the orders}

The three partial orders each have their own advantages. The HN partial order has the most interesting defining property, showing how all of the Harder--Narasimhan filtrations of one maximal green sequence are refined by those of another. However, it is in principle difficult to compute because it requires checking many filtrations (in general, infinitely many). It is likewise difficult to verify whether there is a sequence of increasing elementary polygonal deformations relating two maximal green sequences. In contrast, the partial order $\sleq$ is the easiest to compute, since one only has to compare two finite sets. The advantage of the order $\dleq$ is that its local structure is clear, since we know its covering relations. This also makes it the easiest order to prove things about. Knowing that these orders were the same would give a single partial order on equivalence classes of maximal green sequences with all of these virtues.

\begin{conjecture}\label{conj}
For a finite-dimensional $K$-algebra $\Lambda$ and two maximal green sequences $\mathcal{G}$ and $\mathcal{G}'$ of $\Lambda$, the following are equivalent.
\begin{enumerate}[label=\textup{(}\arabic*\textup{)}]
    \item $[\mathcal{G}] \dleq [\mathcal{G}']$.
    \item $[\mathcal{G}] \sleq [\mathcal{G}']$.
    \item $[\mathcal{G}] \hleq [\mathcal{G}']$.
\end{enumerate}
\end{conjecture}

Note that the orders on silting complexes analogous to these orders on equivalence classes of maximal green sequences discussed in Remark~\ref{rmk:silt_orders} in general only have the same Hasse diagram, rather than being equal. This is because the order on silting complexes given by green mutations only applies when the sequence of mutations is finite, whereas there may be inclusion of aisles between two silting complexes not related by such a finite sequence of mutations.
In the present case, we believe that the orders should actually be equal to each other, since two maximal green sequences related by the orders should only be related by a finite sequence of covering relations.

\begin{remark}\label{rmk:no_gap}
Conjecture~\ref{conj} can also be seen as a refined version of the No-Gap Conjecture \cite[Conjecture~1.22]{bdp}, cases of which were proven in \cite{g-mc,hi-no-gap}. Said conjecture says that the set of lengths of maximal green sequences of a quiver should have no gaps. In these cases, the only oriented polygons are pentagons, and so increasing elementary polygonal deformations only change the length of the maximal green sequence by one. For an acyclic quiver $Q$, there exists a maximal green sequence $\mathcal{G}'$ of $KQ$ given by only mutating at sources.
Furthermore, $\mathcal{G}'$ is of minimal length, and for any other maximal green sequence $\mathcal{G}$, we have that $\summ{\mathcal{G}} \supseteq \summ{\mathcal{G}'}$, since $\summ{\mathcal{G}'}$ only consists of projectives and shifted projectives. Conjecture~\ref{conj} would then give that $[\mathcal{G}] \dleq [\mathcal{G}']$, which gives a set of maximal green sequences of every length between that of $\mathcal{G}$ and $\mathcal{G}'$, thereby establishing the No-Gap Conjecture for acyclic quivers. However, Conjecture~\ref{conj} applies to a much larger class of algebras, including those of non-simply-laced type and those which are not Jacobian algebras of quivers with potential.
\end{remark}

\begin{example}\label{ex:brick_poset}
We consider the algebra from Example~\ref{ex:brick_counter}. The poset of its equivalence classes of maximal green sequences under reverse-inclusion of summands is shown in Figure~\ref{fig:summ_poset}. It can be verified that equivalence classes related by a covering relation in this order are also related by an increasing elementary polygonal deformation, so that $\sleq$ coincides with $\dleq$ in this instance. By computing the Harder--Narasimhan filtrations for the indecomposables, one can also check that $\hleq$ coincides with the other two orders here.

We give some particular examples for the deformation order and Harder--Narsimhan order. The maximal green sequences $\mathcal{G}$ and $\mathcal{G}'$ with \[\summ{\mathcal{G}} = \left\{\tcs{1},\, \tcs{1\\2},\, \tcs{2},\, \tcs{3\\2},\, \tcs{1\\2}[1],\, \tcs{2}[1],\, \tcs{3\\2}[1] \right\}\] and \[\summ{\mathcal{G}'} = \left\{ \tcs{1\\2},\, \tcs{2},\, \tcs{3\\2},\, \tcs{1\\2}[1],\, \tcs{2}[1],\, \tcs{3\\2}[1] \right\}\] are such that $[\mathcal{G}] \dlessdot [\mathcal{G}']$ due to the increasing elementary polygonal deformation across the oriented polygon 
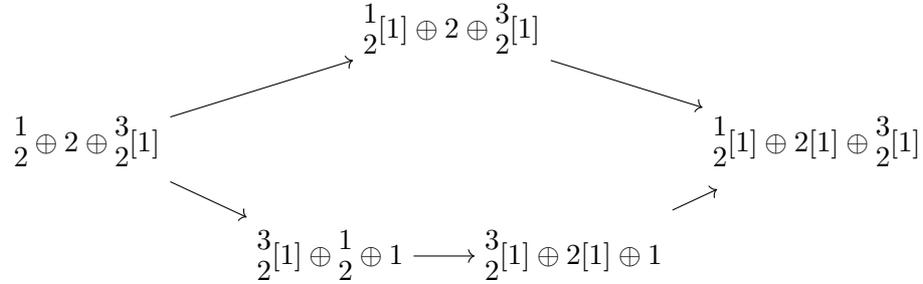
\[
\begin{tikzpicture}[xscale=3.2,yscale=1.5]

     \node (s) at (0,0) {$\tcs{1\\2} \oplus \tcs{2} \oplus \tcs{3\\2}[1]$};
     \node (t) at (1.5,1) {$\tcs{1\\2}[1] \oplus \tcs{2} \oplus \tcs{3\\2}[1]$};
     \node (e) at (3,0) {$\tcs{1\\2}[1] \oplus \tcs{2}[1] \oplus \tcs{3\\2}[1]$};
     \node (b1) at (1,-1) {$\tcs{3\\2}[1] \oplus \tcs{1\\2} \oplus \tcs{1}$};
     \node (b2) at (2,-1) {$\tcs{3\\2}[1] \oplus \tcs{2}[1] \oplus \tcs{1}$};

     \draw[->] (s) -- (t);
     \draw[->] (t) -- (e);
     \draw[->] (s) -- (b1);
     \draw[->] (b1) -- (b2);
     \draw[->] (b2) -- (e);
     
\end{tikzpicture}
\] as can be seen from Figure~\ref{fig:silting_counter}.

As sequences of bricks, we have that $\mathcal{G}$ is \[\tcs{3}, \, \tcs{2}, \, \tcs{1\\2}, \, \tcs{1}\] and $\mathcal{G}'$ is \[\tcs{3}, \, \tcs{1}, \, \tcs{2}.\] The fact that $[\mathcal{G}] \hl [\mathcal{G}']$ can be seen from the fact that \[\bhnfbig{\mathcal{G}}{\tcs{1\\2}} = \Bigl\{ \tcs{1\\2} \Bigr\}\] and \[\bhnfbig{\mathcal{G}'}{\tcs{1\\2}} = \{\tcs{1}, \, \tcs{2}\}.\] Hence, we have
\begin{align*}
    \bhnfbig{\mathcal{G}'}{\tcs{1\\2}} &= \{\tcs{1}, \, \tcs{2}\} \\
    &= \bigsqcup_{B \in \Bigl\{\tcss{1\\2}\Bigr\}} \bhnf{\mathcal{G}'}{B} \\
    &= \bigsqcup_{B \in \bhnfbig{\mathcal{G}}{\tcss{1\\2}}} \bhnf{\mathcal{G}'}{B}.
\end{align*}

An important non-example for the Harder--Narasimhan order is given by the maximal green sequences $\mathcal{G}_1$ \[\tcs{2}, \, \tcs{1\\2}, \, \tcs{1}, \, \tcs{3\\2}, \, \tcs{3}\] and $\mathcal{G}_2$ \[\tcs{2}, \, \tcs{3\\2}, \, \tcs{3}, \, \tcs{1\\2}, \, \tcs{1}\] from Example~\ref{ex:brick_counter}. In this case, we have that \[\bhnfbig{\mathcal{G}_1}{\tcs{1\phantom{2}3\\\phantom{1}2\phantom{3}}} = \Bigl\{\tcs{1}, \, \tcs{3\\2}\Bigr\}\] and \[\bhnfbig{\mathcal{G}_2}{\tcs{1\phantom{2}3\\\phantom{1}2\phantom{3}}} = \Bigl\{\tcs{3}, \, \tcs{1\\2}\Bigr\}.\] Hence, one can show that $[\mathcal{G}_1] \not\hleq [\mathcal{G}_2]$, since
\begin{align*}
    \bhnfbig{\mathcal{G}_2}{\tcs{1\phantom{2}3\\\phantom{1}2\phantom{3}}} &= \Bigl\{\tcs{3}, \, \tcs{1\\2}\Bigr\} \\
    &\neq \Bigl\{\tcs{1}, \, \tcs{3\\2}\Bigr\} \\
    &= \bigsqcup_{B \in \bhnfbig{\mathcal{G}_1}{\tcss{1\phantom{2}3\\\phantom{1}2\phantom{3}}}} \bhnf{\mathcal{G}_2}{B}.
\end{align*}The argument that $[\mathcal{G}_2] \not\hleq [\mathcal{G}_1]$ is similar.
\end{example}

\begin{figure}
    \caption{The poset of reverse-inclusion of indecomposable direct summands of two-term silting complexes from Example~\ref{ex:brick_poset}}
    \label{fig:summ_poset}
\[
\begin{tikzpicture}[scale=2,xscale=1.7]

\node(max) at (0,3) {$\left\{ \tcs{1\\2},\, \tcs{2},\, \tcs{3\\2},\, \tcs{1\\2}[1],\, \tcs{2}[1],\, \tcs{3\\2}[1] \right\}$};
\node(r2) at (1,2) {$\left\{\tcs{1},\, \tcs{1\\2},\, \tcs{2},\, \tcs{3\\2},\, \tcs{1\\2}[1],\, \tcs{2}[1],\, \tcs{3\\2}[1] \right\}$};
\node(r1) at (1,1) {$\left\{\tcs{1},\, \tcs{1\\2},\, \tcs{1\phantom{2}3\\2},\, \tcs{2},\, \tcs{3\\2},\, \tcs{1\\2}[1],\, \tcs{2}[1],\, \tcs{3\\2}[1] \right\}$};
\node(l2) at (-1,2) {$\left\{ \tcs{1\\2},\, \tcs{2},\, \tcs{3\\2},\, \tcs{3},\, \tcs{1\\2}[1],\, \tcs{2}[1],\, \tcs{3\\2}[1] \right\}$};
\node(l1) at (-1,1) {$\left\{ \tcs{1\\2},\, \tcs{1\phantom{2}3\\2},\, \tcs{2},\, \tcs{3\\2},\, \tcs{3},\, \tcs{1\\2}[1],\, \tcs{2}[1],\, \tcs{3\\2}[1] \right\}$};
\node(min) at (0,0) {$\left\{\tcs{1},\, \tcs{1\\2},\, \tcs{1\phantom{2}3\\2},\, \tcs{2},\, \tcs{3\\2},\, \tcs{3},\, \tcs{1\\2}[1],\, \tcs{2}[1],\, \tcs{3\\2}[1] \right\}$};

\draw[->] (max) -- (r2);
\draw[->] (max) -- (l2);
\draw[->] (r2) -- (r1);
\draw[->] (l2) -- (l1);
\draw[->] (l1) -- (min);
\draw[->] (r1) -- (min);

\end{tikzpicture}
\]
\end{figure}
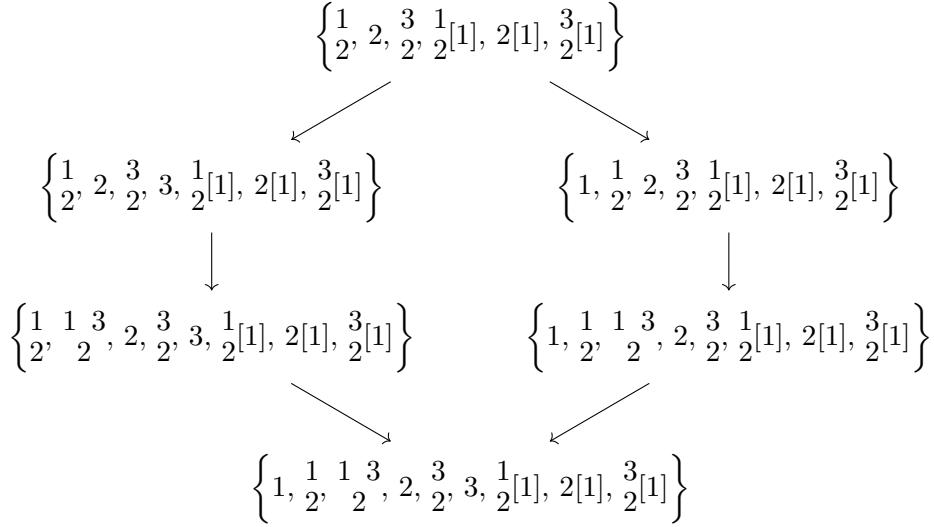

We show that the summand order holds whenever the deformation order holds.

\begin{theorem}\label{thm:first->second}
Let $\Lambda$ be a finite-dimensional algebra over a field $K$ with $\mathcal{G}$ and $\mathcal{G}'$ two maximal green sequences of $\Lambda$. Then $[\mathcal{G}] \dleq [\mathcal{G}']$ implies that $[\mathcal{G}] \sleq [\mathcal{G}']$.
\end{theorem}
\begin{proof}
It suffices to show that $[\mathcal{G}] \dlessdot [\mathcal{G}']$ implies $[\mathcal{G}] \sleq [\mathcal{G}']$. Hence, we shall show that if $\mathcal{G}'$ is an increasing elementary polygonal deformation of $\mathcal{G}$, then $\summ{\mathcal{G}} \supseteq \summ{\mathcal{G}'}$. Locally, an increasing elementary polygonal deformation looks as in Figure~\ref{fig:deformation}. The maximal green sequence passing along the top of this polygon is $\mathcal{G}$ and the maximal green sequence passing along the bottom is $\mathcal{G}'$. By inspection, every summand from the bottom path also occurs in the top path. Hence, $\summ{\mathcal{G}} \supseteq \summ{\mathcal{G}'}$, as desired.
\end{proof}

We now show that the deformation order also always implies the Harder--Narasimhan order.

\begin{theorem}\label{thm:def->hn}
Let $\Lambda$ be a finite-dimensional algebra over a field $K$ with $\mathcal{G}$ and $\mathcal{G}'$ two maximal green sequences of $\Lambda$. Then $[\mathcal{G}] \dleq [\mathcal{G}']$ implies that $[\mathcal{G}] \hleq [\mathcal{G}']$.
\end{theorem}
\begin{proof}
It suffices to show that $[\mathcal{G}] \dlessdot [\mathcal{G}']$ implies $[\mathcal{G}] \hleq [\mathcal{G}']$, so we suppose that $\mathcal{G}'$ is an increasing elementary polygonal deformation of $\mathcal{G}$. Let $B$ and $B'$ be the brick labels for the short path around the polygon, so that $B < B'$ in $\mathcal{G}'$, but $B' < B$ in $\mathcal{G}$. Further, let $M$ be a $\Lambda$-module. If $B \notin \bhnf{\mathcal{G}'}{M}$ or $B' \notin \bhnf{\mathcal{G}'}{M}$, then the $\mathcal{G}'$-HN filtration of $M$ is also a valid $\mathcal{G}$-HN filtration, and so $\bhnf{\mathcal{G}}{M} = \bhnf{\mathcal{G}'}{M}$. 

We thus assume that $B, B' \in \bhnf{\mathcal{G}'}{M}$. Let $N$ be the subquotient of $M$ spanned by the two $\mathcal{G}'$-semistable factors in $\Filt(B)$ and $\Filt(B')$. Then $N \in \Filt(B, B')$. By Lemma~\ref{lem:poly_ab_cat}, the long path around the polygon therefore induces an HN filtration of $N$ in the abelian category $\Filt(B, B')$, which gives us the $\mathcal{G}$-HN filtration of $N$. The $\mathcal{G}$-HN filtration of $M$ is then obtained from the $\mathcal{G}'$-HN filtration of $M$ by replacing the filtration of $N$ with this filtration. We then have that \[\bhnf{\mathcal{G}'}{N} = \bigsqcup_{B'' \in \bhnf{\mathcal{G}}{N}}\bhnf{\mathcal{G}'}{B''},\] since the elements of $\bhnf{\mathcal{G}}{N}$ lie in $\Filt(B, B')$ and $B$ and $B'$ are the relatively simple objects in $\Filt(B, B')$. Indeed, $\Filt(B, B')$ is an abelian category by Lemma~\ref{lem:poly_ab_cat}, so applying $\bhnf{\mathcal{G}'}{-}$ just gives the multiset of composition factors in this category. Using the fact that $N$ is a subquotient of both the $\mathcal{G}$-HN filtration and the $\mathcal{G}'$-HN filtration, we deduce that
\begin{align*}
    \bhnf{\mathcal{G}'}{M} &= (\bhnf{\mathcal{G}'}{M} \setminus \bhnf{\mathcal{G}'}{N}) \sqcup \bhnf{\mathcal{G}'}{N} \\
    &= (\bhnf{\mathcal{G}}{M} \setminus \bhnf{\mathcal{G}}{N}) \sqcup \bhnf{\mathcal{G}'}{N} \\
    &= (\bhnf{\mathcal{G}}{M} \setminus \bhnf{\mathcal{G}}{N}) \sqcup \bigsqcup_{B'' \in \bhnf{\mathcal{G}}{N}}\bhnf{\mathcal{G}'}{B''} \\
    &= \bigsqcup_{B'' \in \bhnf{\mathcal{G}}{M}}\bhnf{\mathcal{G}'}{B''},
\end{align*}
since for $B'' \in \bhnf{\mathcal{G}}{M} \setminus \bhnf{\mathcal{G}}{N}$, we have $\bhnf{\mathcal{G}'}{B''} = \{B''\}$. This gives us that $[\mathcal{G}] \hleq [\mathcal{G}']$, as desired.
\end{proof}

\begin{corollary}\label{cor:def->bricks}
Let $\Lambda$ be a finite-dimensional algebra over a field $K$ with $\mathcal{G}$ and $\mathcal{G}'$ two maximal green sequences of $\Lambda$. Then $[\mathcal{G}] \dleq [\mathcal{G}']$ implies that $\bricks{\mathcal{G}} \supseteq \bricks{\mathcal{G}'}$.
\end{corollary}
\begin{proof}
This follows immediately either from Lemma~\ref{lem:iepd_via_bricks}, or from combining Theorem~\ref{thm:def->hn} and Proposition~\ref{lem:hn->bricks}.
\end{proof}

\begin{remark} \label{rem:incomplete_strategy}
Given maximal green sequences $\mathcal{G}$ and $\mathcal{G}'$ such that $[\mathcal{G}] \hleq [\mathcal{G}']$, it is natural to attempt to prove the converse of Theorem~\ref{thm:def->hn} by using Lemma~\ref{lem:adj_bricks} to find a decreasing elementary polygonal deformation $\mathcal{G}''$ of $\mathcal{G}'$ such that $[\mathcal{G}] \hleq [\mathcal{G}''] \dlessdot [\mathcal{G}']$. However, simply applying Lemma~\ref{lem:adj_bricks} does not always give a deformation which works. Indeed, suppose that we applied this lemma and found $[\mathcal{G}'']$ which differs from $[\mathcal{G}']$ by a sequence of deformations across squares and a single decreasing deformation across a polygon with the short side $B < B'$ inside $\mathcal{G}'$, such that $B > B'$ in $\mathcal{G}$ and in $\mathcal{G}''$.
The issue is that we do not automatically have $[\mathcal{G}] \hleq [\mathcal{G}'']$, as can be shown by applying the example from Remark~\ref{rmk:hn_filt_cant_stick}. 
Recall here that we have the path algebra $\Lambda$ of 
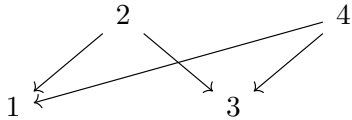
\[
\begin{tikzcd}
& 2 \ar[dr] \ar[dl] && 4 \ar[dl] \ar[dlll] \\
1 && 3 &
\end{tikzcd}
\] of type $\widetilde{A}_{4}$ with maximal green sequences $\mathcal{G}$ \[1,\, \tcs{2\\1},\, \tcs{4\\1},\, \tcs{2\phantom{1}4\\\phantom{2}1\phantom{4}},\, \tcs{3},\, \tcs{2\\3},\, \tcs{2},\, \tcs{4\\3},\, \tcs{4}\] and $\mathcal{G}'$ \[4,\, 2,\, 3,\, 1.\] All of the adjacent pairs of bricks in $\mathcal{G}'$ are ordered differently in $\mathcal{G}$, so we could obtain any of them from Lemma~\ref{lem:adj_bricks}. But, in particular, we could obtain the pair $2$ and $3$. Deforming these across a pentagon yields the maximal green sequence~$\mathcal{G}''$ \[4,\, 3,\, \tcs{2\\3},\, 2,\, 1.\] But now we do not have $[\mathcal{G}] \hleq [\mathcal{G}'']$, since \[\bhnfbig{\mathcal{G}}{\tcs{2\phantom{3}4\\\phantom{2}3\phantom{4}}} = \left\{4,\, \tcs{2\\3}\right\},\] whereas \[\bhnfbig{\mathcal{G}''}{\tcs{2\phantom{3}4\\\phantom{2}3\phantom{4}}} = \left\{2,\, \tcs{4\\3}\right\}.\] It is an exercise for the reader to verify that in this case we still have $[\mathcal{G}] \dleq [\mathcal{G}']$, so that this does not provide a counter-example to Conjecture~\ref{conj}. This all contrasts with the later situation in Theorem~\ref{thm:nak_brick_order} for the case of Nakayama algebras, where any deformation obtained from Lemma~\ref{lem:adj_bricks} will work. The difference is in the converse of Lemma~\ref{lem:hn->bricks}: it holds for Nakayama algebras by Corollary~\ref{cor:nak_orders}, but does not hold in general. The failure is illustrated by the present example: we have $\bricks{\mathcal{G}''} \subseteq \bricks{\mathcal{G}}$, but do not have $[\mathcal{G}] \hleq [\mathcal{G}'']$. One can also see that we do not have $[\mathcal{G}] \dleq [\mathcal{G}'']$. In order to deform $[\mathcal{G}'']$ into $[\mathcal{G}]$, one would need to move $1$ past $\tcss{2\\3}$, but when one does this one is forced to insert the brick \[\tcs{\phantom{1}2\phantom{3}\\1\phantom{2}3},\] which is not a brick of $[\mathcal{G}]$.
\end{remark}

We prove Conjecture~\ref{conj} in the simple case where the algebra $\Lambda$ only has two simple modules up to isomorphism, for which we need the following lemma. This was previously obtained independently in \cite[Proposition~4.2]{hofmann_thesis}.

\begin{lemma}\label{lem:simple_bricks}
If $\mathcal{G}$ is a maximal green sequence given as a maximal backwards $\Hom$-orthogonal sequence of bricks $B_1, B_2, \dots, B_r$, then both $B_1$ and $B_r$ are simple $\Lambda$-modules.
\end{lemma}
\begin{proof}
That $B_1$ must be a simple $\Lambda$-module follows from the fact that it must be a relatively simple object in the first torsion class of $\mathcal{G}$, which is $\modules \Lambda$. That $B_r$ must also be a simple $\Lambda$-module follows from the duality between torsion and torsion-free, but can also be seen by the following direct argument. As explained in Section~\ref{sect:back:brick_label}, we must have that the final non-zero torsion class of $\mathcal{G}$ is $\Filt(B)$. But for this to be a torsion class, we must have that $B$ has no proper factor modules, which implies that $B$ is a simple $\Lambda$-module.
\end{proof}

\begin{theorem}\label{thm:two_simples}
Let $\Lambda$ be a finite-dimensional algebra over a field $K$ with two isomorphism classes of simple modules. Let $\mathcal{G}$ and $\mathcal{G}'$ be maximal green sequences of $\Lambda$. Then the following are equivalent.
\begin{enumerate}[label=\textup{(}\arabic*\textup{)}]
\item $[\mathcal{G}] \dleq [\mathcal{G}']$.
\item $[\mathcal{G}] \sleq [\mathcal{G}']$.
\item $[\mathcal{G}] \hleq [\mathcal{G}']$.
\item $\bricks{\mathcal{G}} \supseteq \bricks{\mathcal{G}'}$.
\end{enumerate}
\end{theorem}
\begin{proof}
We already know that $[\mathcal{G}] \dleq [\mathcal{G}']$ implies $[\mathcal{G}] \sleq [\mathcal{G}']$, $[\mathcal{G}] \hleq [\mathcal{G}']$, and $\bricks{\mathcal{G}} \supseteq \bricks{\mathcal{G}'}$ by 
Theorem~\ref{thm:first->second},
Theorem~\ref{thm:def->hn}, and Corollary~\ref{cor:def->bricks}. We wish to show the three converse implications. Note first that the Hasse diagram of the poset of two-term silting complexes of $\Lambda$ is 2-regular by \cite[Corollary~3.8(a)]{air}. Hence $\Lambda$ has at most two maximal green sequences. If $\Lambda$ has fewer than two maximal green sequences, then the result is trivial, so suppose that $\Lambda$ has two maximal green sequences $\mathcal{G}$ and $\mathcal{G}'$.

Suppose that $[\mathcal{G}] \not\dleq [\mathcal{G}']$. Since we are assuming that $\mathcal{G}$ and $\mathcal{G}'$ are the two distinct maximal green sequences of $\Lambda$, the only option is that $\mathcal{G}'$ is a maximal green sequence of length greater than two, by Definition~\ref{def:iepd}.

To show that $[\mathcal{G}] \not\sleq [\mathcal{G}']$, we note that there exists $X \in \summ{\mathcal{G}'}$ which is neither a projective nor a shifted projective, since $\mathcal{G}'$ has length greater than two. We have that $X$ completes to precisely two different basic two-term silting complexes $X \oplus Y$ and $X \oplus Y'$ \cite[Corollary~3.8(a)]{air}. Both $X \oplus Y$ and $X \oplus Y'$ therefore occur in $\mathcal{G}'$ and neither is $\Lambda$ or $\Lambda[1]$. Hence $X \notin \summ{\mathcal{G}}$, since neither of the two-term silting complexes it is an indecomposable summand of occurs in $\mathcal{G}$. 
We conclude that $[\mathcal{G}] \not\sleq [\mathcal{G}']$, as desired.

To show that $[\mathcal{G}] \not\hleq [\mathcal{G}']$, we note that there exists $B \in \bricks{\mathcal{G}'}$ which is not a simple. Let $S_1$ and $S_2$ be the simple $\Lambda$-modules. By Lemma~\ref{lem:simple_bricks}, these must be the first and last bricks in $\mathcal{G}$ and $\mathcal{G}'$. Suppose that these are ordered $S_1 < B < S_2$ by $\mathcal{G}'$, in which case the simples must be ordered $S_2 < S_1$ in $\mathcal{G}$. By backwards $\Hom$-orthogonality, the top of $B$ cannot contain $S_1$, so it
must contain $S_2$. Similarly, the socle of $B$ cannot contain $S_2$, so it must contain $S_1$. This implies that $B \notin \bricks{\mathcal{G}}$, since $B$ would have to occur after $S_2$ and before $S_1$ by Lemma~\ref{lem:simple_bricks}, which would contradict backwards $\Hom$-orthogonality. 
Hence, we have $\bricks{\mathcal{G}} \not\supseteq \bricks{\mathcal{G}'}$. By Lemma~\ref{lem:hn->bricks}, we then have $[\mathcal{G}] \not\hleq [\mathcal{G}']$. Note also that bricks therefore determine the equivalence class for algebras with two simples.
\end{proof}

One implication of Theorem~\ref{thm:two_simples} is that unoriented polygons are ``unoriented'' in all of the partial orders. That is, maximal green sequences which differ by a deformation across an unoriented polygon are not related in any of the orders.

\subsubsection{Maxima and minima}

We show that in certain cases the partial orders have unique maxima or unique minima. We first consider cases where the partial orders have unique maxima.

\begin{proposition}\label{prop:unique_max}
Let $\Lambda = KQ/I$ be a path algebra with relations, where $Q$ is acyclic. Then there is a maximal green sequence $\mathcal{G}_{\max}$ whose equivalence class is the unique maximum for $\sleq$ and $\hleq$, and which is maximal in $\dleq$.
\end{proposition}
\begin{proof}
If $Q$ is an acyclic quiver, then it is clear that the vertices of $Q$ may be labelled $1, 2, \dots, n$ such that $\Hom_{\Lambda}(P_i, P_j) = 0$ for $i < j$, where $P_i$ is the indecomposable projective at vertex $i$. Then there is a maximal green sequence $\mathcal{G}_{\max}$ of $\Lambda$ given by the sequence of exchange pairs $(P_1, P_1[1]), (P_2, P_2[1]), \dots, (P_n, P_n[1])$. Then $\summ{\mathcal{G}_{\max}}$ consists of the indecomposable projectives and indecomposable shifted projectives. Hence, for any other maximal green sequence $\mathcal{G}'$ of $\Lambda$, we have that $[\mathcal{G}'] \sleq [\mathcal{G}_{\max}]$, which proves that $[\mathcal{G}_{\max}]$ is the unique maximum for~$\sleq$.

To show that $[\mathcal{G}_{\max}]$ is the unique maximum for $\hleq$, we first note that $\bricks{\mathcal{G}_{\max}}$ consists of only the simple $\Lambda$-modules $\{S_1, \dots, S_n\}$, since $\bricks{\mathcal{G}_{\max}}$ must contain all simple modules by Theorem~\ref{thm:enomoto_relsimp} and $\# \bricks{\mathcal{G}_{\max}} = \# \summ{\mathcal{G}_{\max}} = n$. By the ordering of $Q_0$ we chose earlier, we in fact have that $S_1, S_2, \dots, S_n$ is a backwards maximal $\Hom$-orthogonal sequence of bricks corresponding to $\mathcal{G}_{\max}$. Now let $\mathcal{G}$ be a maximal green sequence of $\Lambda$ and $M$ be a $\Lambda$-module. Then $\bhnf{\mathcal{G}_{\max}}{M}$ just consists of the composition factors of $M$, so it is clear that \[\bhnf{\mathcal{G}_{\max}}{M} = \bigsqcup_{B \in \bhnf{\mathcal{G}}{M}} \bhnf{\mathcal{G}_{\max}}{B}.\]

Finally, it is clear that $\mathcal{G}_{\max}$ is maximal in $\dleq$, because it does not admit any increasing elementary polygonal deformations.
\end{proof}

It is not obvious if $\mathcal{G}_{\max}$ is the unique maximal green sequence which does not admit any increasing elementary polygonal deformations. \textit{A priori}, there may be other maximal green sequences which do not admit any increasing elementary polygonal deformations --- maximal green sequences where one has got stuck, so to speak. However, due to Conjecture~\ref{conj}, we expect this not to happen.

\bigskip

For the cases where the partial orders have unique minima, we need the following concepts, which can be found in \cite{ringel_ta_iqf}. A \emph{path} in $\modules \Lambda$ is a tuple $(M_0, \dots, M_s)$ of $\Lambda$-modules with $s \geqslant 1$ such that for each $1 \leqslant i \leqslant s$, there exists a map $M_{i - 1} \to M_{i}$ which is neither zero nor an isomorphism. An algebra $\Lambda$ is \emph{representation-directed} if and only if there exists no path $(M_0, \dots, M_s)$ with $M_0 \cong M_s$. It is known that every representation-directed algebra is representation-finite \cite{ringel_ta_iqf}.

\begin{proposition}\label{prop:unique_min}
If $\Lambda$ is representation-directed, then there is a maximal green sequence $\mathcal{G}_{\min}$ whose equivalence class is the unique minimum for all orders.
\end{proposition}
\begin{proof}
Since $\Lambda$ is representation-directed, it is representation-finite, and so has finitely many bricks. Moreover, the fact that $\Lambda$ is representation-directed means that these bricks can be ordered $B_1, B_2, \dots, B_r$ in such a way that $\Hom_\Lambda(B_j, B_i) = 0$ for $i < j$. We therefore have a maximal backwards $\Hom$-orthogonal sequence of bricks, which gives us a maximal green sequence $\mathcal{G}_{\min}$. Since $\# \bricks{\mathcal{G}_{\min}} = \# \sumt{\mathcal{G}_{\min}}$, as both are equal to the length of $\mathcal{G}_{\min}$, we have that $\sumt{\mathcal{G}_{\min}}$ must consist of all indecomposable $\tau$-rigid modules of $\Lambda$, because these are in bijection with the bricks of $\Lambda$ by \cite[Theorem~4.1]{dij}. Hence, by Theorem~\ref{thm:mg_equiv}, having $\bricks{\mathcal{G}_{\min}}$ consist of all of the bricks of $\Lambda$ determines $\mathcal{G}_{\min}$ up to equivalence.

Since $\sumt{\mathcal{G}_{\min}}$ consists of all indecomposable $\tau$-rigid modules, $[\mathcal{G}_{\min}]$ is clearly minimal in $\sleq$. To show that $[\mathcal{G}_{\min}]$ is the unique minimum in $\dleq$, let $\mathcal{G}$ be a maximal green sequence of $\Lambda$. If $\bricks{\mathcal{G}_{\min}} = \bricks{\mathcal{G}}$, the above argument shows that $\mathcal{G} \sim \mathcal{G}_{\min}$, so it is sufficient to check that $[\mathcal{G}_{\min}] \dleq [\mathcal{G}]$ for $\mathcal{G}$ such that
$\bricks{\mathcal{G}_{\min}} \supset \bricks{\mathcal{G}}$. For such $\mathcal{G}$, we can use Lemma~\ref{lem:adj_bricks} to find bricks $B$ and $B'$ which are adjacent with $B < B'$ in $\mathcal{G}$ but $B' > B$ in $\mathcal{G}_{\min}$. If swapping $B$ and $B'$ in $\mathcal{G}$ results in a deformation across a square, then we can repeat this process, so we can assume that swapping $B$ and $B'$ results in an increasing elementary polygonal deformation $\mathcal{G}'$ of~$\mathcal{G}$. Since $\bricks{\mathcal{G}_{\min}}$ consists of all bricks, we have that $\bricks{\mathcal{G}_{\min}} \supseteq \bricks{\mathcal{G}'} \supset \bricks{\mathcal{G}}$.  By induction, we conclude that $[\mathcal{G}_{\min}] \dleq [\mathcal{G}]$. (The induction process stops since $\Lambda$ is representation-finite, and so $\#(\bricks{\mathcal{G}_{\min}} \backslash \bricks{\mathcal{G}}) < \infty$.) By Theorem~\ref{thm:def->hn}, we have that $[\mathcal{G}_{\min}] \hleq [\mathcal{G}]$, so that $[\mathcal{G}_{\min}]$ is also the unique minimum in~$\hleq$.
\end{proof}

\begin{remark}
Note that for algebras which have infinite global dimension, the posets of equivalence classes of maximal green sequences do not always have unique maxima and minima. For example, one can compute the posets for the path algebra of the three-cycle with relations given by paths of length two.

Furthermore, for hereditary algebras which are representation-infinite, the posets may not have unique minima. For instance, \cite[Theorem~M3]{ai2020} shows that for path algebras of type $\widetilde{A}_4$, there are maximal green sequences which are of maximal length but have different sets of bricks. More generally, maximal green sequences of maximal length for tame hereditary algebras are studied in \cite{ai2020} and \cite{kn2021}, whilst minimal length maximal green sequences are studied in \cite{gms}.
\end{remark}

\subsection{Exchange pairs}\label{sect:partial_order_exch_pair}

In this subsection, we consider how the partial orders on equivalence classes of maximal green sequences interact with exchange pairs. Note first that one cannot define a partial order using inclusion of exchange pairs.

\begin{proposition}
If maximal green sequences $\mathcal{G}$ and $\mathcal{G}'$ of $\Lambda$ are such that $\exch{\mathcal{G}} \supseteq \exch{\mathcal{G'}}$, then we have $\mathcal{G} \sim \mathcal{G}'$.
\end{proposition}
\begin{proof}
Suppose that we have $\exch{\mathcal{G}} \supseteq \exch{\mathcal{G}'}$. Suppose for contradiction that there is an exchange pair $(X, Y)$ of $\mathcal{G}$ which is not an exchange pair of $\mathcal{G}'$. We may choose $(X, Y)$ to be the first exchange pair of $\mathcal{G}$ for which this is the case.

We claim that $X \in \summ{\mathcal{G}'}$. This is certainly true if $X$ is projective. If $X$ is not projective, then there exists an exchange pair $(W, X) \in \exch{\mathcal{G}}$. This must occur before $(X, Y)$ in $\mathcal{G}$ by Lemma~\ref{lem:at_most_once}. Hence, by the choice of $(X, Y)$, we have that $(W, X) \in \exch{\mathcal{G}'}$, and so $X \in \summ{\mathcal{G'}}$ in this case too.

However, if $X \in \summ{\mathcal{G}'}$, we must have, for some $Y' \not\cong Y$, that $(X, Y')$ is an exchange pair of $\mathcal{G}'$ and therefore of $\mathcal{G}$, since $X$ cannot be a shifted projective due to the exchange pair $(X, Y)$. But this contradicts Lemma~\ref{lem:at_most_once}. Hence $\exch{\mathcal{G}} = \exch{\mathcal{G}'}$, and so $\mathcal{G} \sim \mathcal{G}'$ by Theorem~\ref{thm:mg_equiv}.
\end{proof}

Moreover, the data of the exchange pairs gives the data of the indecomposable presilting summands, so it would seem impossible to define a partial order using exchange pairs by other means, without its collapsing into the summand partial order. However, the exchange pairs of a maximal green sequence do exhibit interesting behaviour with respect to the deformation order.

\begin{proposition}\label{prop:exch_red}
Let $\Lambda$ be a finite-dimensional algebra over a field~$K$. Let
\[
\begin{tikzpicture}[scale=1.5,xscale=1.7]
    
\node (exa) at (0,0) {$E \oplus X \oplus A$};
\node (eza) at (2,-1) {$E \oplus Z \oplus A$};
\node (ezc) at (4,0) {$E \oplus Z \oplus C$};
\node (exc') at (0.4,1) {$E \oplus X \oplus C'$};
\node (ez'c') at (1.7,1) {$E \oplus Z' \oplus C'$};
\node (dots) at (2.65,1) {$\dots$};
\node (ex'c) at (3.6,1) {$E \oplus X' \oplus C$};

\draw[->] (exa) -- (eza);
\draw[->] (eza) -- (ezc);
\draw[->] (exa) -- (exc');
\draw[->] (exc') -- (ez'c');
\draw[->] (ez'c') -- (dots);
\draw[->] (dots) -- (ex'c);
\draw[->] (ex'c) -- (ezc);

\end{tikzpicture}
\] be an oriented polygon in $\twosilt \Lambda$. There are then commutative diagrams of exchange triangles \[
\begin{tikzcd}
X \ar[r] \ar[d,equal] & Y' \ar[r] \ar[d] & Z' \ar[r] \ar[d] & X[1] \ar[d,equal] \\
X \ar[r] & Y \ar[r] & Z \ar[r] & X[1]
\end{tikzcd}
\] and \[
\begin{tikzcd}
A \ar[r] \ar[d,equal] & B' \ar[r] \ar[d] & C' \ar[r] \ar[d] & A[1] \ar[d,equal] \\
A \ar[r] & B \ar[r] & C \ar[r] & A[1],
\end{tikzcd}
\] where $Y$, $Y'$, $B$, and $B'$ are the two-term complexes which appear in the middle of the relevant exchange triangles.
\end{proposition}
\begin{proof}
By definition of silting mutation, we have that $X \to Y'$ is a minimal left $\additive (E \oplus C')$-approximation and that $f \colon X \to Y$ is a minimal left $\additive (E \oplus A)$-approximation. We claim that $f$ is in fact a minimal left $\additive E$-approximation. We prove this claim using silting reduction \cite{iyama_yang}. Let $\mathcal{Z} = (\hol{E[>0]}) \cap (\hor{E[<0]})$. We have that $\mathcal{Z}/[E]$ is a triangulated category by \cite[Theorem~4.2]{iy-red}, with shift functor denoted by $\iysh$. Let $\widetilde{(-)}\colon \mathcal{Z} \to \mathcal{Z}/[E]$ be the quotient map and $\Gamma = \End_{\mathcal{Z}/[E]}(\widetilde{X} \oplus \widetilde{A})$. Then, as in Remark~\ref{rmk:reduction}, two-term silting complexes in $\twosilt \Gamma$ correspond to the different completions of $E$ to a two-term silting complex.

Then, since $\widetilde{E} \oplus \widetilde{Z} \oplus \widetilde{C}$ is the minimum of $\twosilt \Gamma$, we must have that the images of the exchange pairs on the short path are $(\widetilde{X}, \widetilde{X}\iysh)$  and $(\widetilde{A}, \widetilde{A}\iysh)$. Thus, we have $\widetilde{Z} \cong \widetilde{X}\iysh$ and $\widetilde{C} \cong \widetilde{A}\iysh$. Since $\widetilde{A} \oplus \widetilde{Z} \cong \widetilde{A} \oplus \widetilde{X}\iysh$ is silting in $\mathcal{Z}/[E]$, we must have that $0 = \Hom_{\mathcal{Z}/[E]}(\widetilde{X}\iysh, \widetilde{A}\iysh) \cong \Hom_{\mathcal{Z}/[E]}(\widetilde{X}, \widetilde{A})$. This means that every morphism $X \to A$ must factor through $E$. Letting $Y \cong Y_E \oplus A^{\oplus m}$ where $Y_E \in \additive E$, $m \geq 0$, we then have that there is a factorisation \[
\begin{tikzcd}
    X \ar[r,"g"] \ar[dr,"f",swap] & Y_E \oplus E' \ar[d,"h"] \\
    & Y_E \oplus A^{\oplus m},
\end{tikzcd}
\] where $E' \in \additive E$. Then, since $f$ is a left $\additive (E \oplus A)$-approximation and $Y_E \oplus E' \in \additive E \subseteq \additive (E \oplus A)$, we have that there is a map $l \colon Y_E \oplus A^{\oplus m} \to Y_E \oplus E'$ such that $g = lf$. But then we have that $f = hg = hlf$. Since, $f$ is left minimal, we must have that $hl$ is an isomorphism. This implies that $Y_E \oplus A^{\oplus m}$ is a direct summand of $Y_E \oplus E'$, which in turn implies that $A^{\oplus m} = 0$. Thus, $f \colon X \to Y$ is in fact a minimal left $\additive E$-approximation as claimed.

Therefore, since $X \to Y'$ is a left $\additive (E \oplus C)$-approximation, we must get a factorisation \[
\begin{tikzcd}
X \ar[r] \ar[dr,"f",swap] & Y' \ar[d] \\
& Y.
\end{tikzcd}
\] Using the axioms of triangulated categories, we can extend this to the desired commutative diagram between exchange triangles given in the statement of the proposition.

The proof of the second claim is similar. We have that $\widetilde{X} \oplus \widetilde{A}$ is silting in $\mathcal{Z}/[E]$, so we must have $0 = \Hom_{\mathcal{Z}/[E]}(\widetilde{A}, \widetilde{X}\iysh) = \Hom_{\mathcal{Z}/[E]}(\widetilde{A}, \widetilde{Z})$. Hence, any homomorphism from $A$ to $Z$ must factor through $E$. This means that the minimal left $\additive (E \oplus Z)$-approximation $A \to B'$ of $A$ must in fact be a minimal left $\additive E$-approximation by the same argument as above. Since $A \to B$ is a minimal left $\additive (E \oplus X)$-approximation of $A$, we have a factorisation \[
\begin{tikzcd}
A \ar[r] \ar[dr] & B \ar[d] \\
& B'.
\end{tikzcd}
\] We again use the axioms of triangulated categories to extend this to the commutative diagram shown in the statement of the proposition.
\end{proof}

\begin{corollary}\label{cor:exch_fac}
Let $\Lambda$ be a finite-dimensional algebra over a field $K$ with $\mathcal{G}'$ and $\mathcal{G}$ two maximal green sequences of $\Lambda$. Suppose that $[\mathcal{G}'] \dleq [\mathcal{G}]$ and let $(X, Z)$ be an exchange pair of $\mathcal{G}$ with exchange triangle \[X \to Y \to Z \to X[1].\] Then there is an exchange pair $(X, Z')$ of $\mathcal{G}'$ with exchange triangle \[X \to Y' \to Z' \to X[1]\] and a commutative diagram \[
\begin{tikzcd}
X \ar[r] \ar[d,equal] & Y' \ar[r] \ar[d] & Z' \ar[r] \ar[d] & X[1] \ar[d,equal] \\
X \ar[r] & Y \ar[r] & Z \ar[r] & X[1]
\end{tikzcd}.
\]
\end{corollary}
\begin{proof}
We show the result by induction on the number of deformations from $[\mathcal{G}']$ to $[\mathcal{G}]$. The base case, where $[\mathcal{G}'] = [\mathcal{G}]$ is trivial, since we have $\exch{\mathcal{G}'} = \exch{\mathcal{G}}$ by Theorem~\ref{thm:mg_equiv}.

Hence, we suppose that $[\mathcal{G}]$ is the result of at least one increasing elementary polygonal deformation of $[\mathcal{G}']$. Hence, there is a maximal green sequence $\mathcal{G}''$ such that $[\mathcal{G}'] \dlessdot [\mathcal{G}''] \dleq [\mathcal{G}]$. Let $(X, Z)$ be an exchange pair of $\mathcal{G}$. Then, by Theorem~\ref{thm:first->second}, we must have exchange pairs $(X, Z'')$ of $\mathcal{G}''$ and $(X, Z')$ of $\mathcal{G}'$. By applying Proposition~\ref{prop:exch_red} to $[\mathcal{G}'] \dlessdot [\mathcal{G}'']$ and the induction hypothesis to $[\mathcal{G}''] \dleq [\mathcal{G}]$, we obtain a commutative diagram \[
\begin{tikzcd}
X \ar[r] \ar[d,equal] & Y' \ar[r] \ar[d] & Z' \ar[r] \ar[d] & X[1] \ar[d,equal] \\
X \ar[r] \ar[d,equal] & Y'' \ar[r] \ar[d] & Z'' \ar[r] \ar[d] & X[1] \ar[d,equal] \\
X \ar[r] & Y \ar[r] & Z \ar[r] & X[1]
\end{tikzcd}.
\] This then gives us the desired diagram  \[
\begin{tikzcd}
X \ar[r] \ar[d,equal] & Y' \ar[r] \ar[d] & Z' \ar[r] \ar[d] & X[1] \ar[d,equal] \\
X \ar[r] & Y \ar[r] & Z \ar[r] & X[1]
\end{tikzcd}
\] completing the proof.
\end{proof}

Note that Proposition~\ref{prop:exch_red} and Corollary~\ref{cor:exch_fac} are both dualisable via fixing the second entry of the exchange pairs, rather than the first. The intuition for Corollary~\ref{cor:exch_fac} is that as we go higher up in the deformation order, the exchange pairs become closer to $(X, X[1])$.

\subsection{An example from the twice-punctured torus}\label{sect:oct}

In this section, we consider the maximal green sequences of the Jacobian algebra $\Lambda$ associated to the triangulation of the twice-punctured torus shown in Figure~\ref{fig:torus_tri}. In this figure, A and B are the two punctures, and the triangulation has six different arcs, labelled 1 to 6, dividing the torus into four triangles. The quiver of this triangulation is given by a clockwise oriented three-cycle within each triangle and is shown in Figure~\ref{fig:oct}. The associated potential is \[W = \lambda\delta\kappa + \iota\alpha\delta + \zeta\eta\gamma + \mu\epsilon\beta + \kappa\mu\zeta\theta + \delta\iota\alpha\eta\gamma\epsilon\beta\lambda.\] We then have that $\Lambda = KQ/\langle\partial W\rangle$, that is, $\Lambda$ is the path algebra of $Q$ modulo the cyclic derivatives of the potential, in the usual way from \cite{dwz}. We do not give background on cluster algebras from triangulated surfaces and the associated quivers with potential and cluster categories. Relevant background can be found in \cite{fst,dwz,lf1,lf2,amiot-cluster,dominguez}.

This quiver $Q$ was used in \cite[Example~1]{ky_clus} to give an example of a quiver whose exchange graph had a fundamental group not generated by squares and pentagons, following earlier work in \cite[Remark~9.19]{fst}. We discuss the implications for posets of equivalence classes of maximal green sequences, since this fact about the exchange graph makes the algebra $\Lambda$ a natural place to look for a counter-example to Conjecture~\ref{conj}. However, as we shall explain, \cite[Example~1]{ky_clus} does not give a counter-example to the conjecture that the summand order is equal to the deformation order.

\begin{figure}
    \caption{A triangulation of the twice-punctured torus, with its associated quiver}\label{fig:torus_tri}
    \[
    \begin{tikzpicture}[scale=2]
        
        \draw (0,0) -- (2,2);
        \draw (0,0) -- (2,-2);
        \draw (0,0) -- (-2,2);
        \draw (0,0) -- (-2,-2);
        
        \draw (2,2) -- (-2,2) -- (-2,-2) -- (2,-2) -- (2,2);
        
        \node (1l) at (-2,0) {};
        \node (1r) at (2,0) {};
        \node (2b) at (0,-2) {};
        \node (2t) at (0,2) {};
        \node (3) at (-1,-1) {};
        \node (4) at (1,-1) {};
        \node (5) at (1,1) {};
        \node (6) at (-1,1) {};
        
        \node(alpha) at ($(1r)!0.5!(4)$) {$\alpha$};
        \node(beta) at ($(1l)!0.5!(6)$) {$\beta$};
        \node(gamma) at ($(2b)!0.5!(3)$) {$\gamma$};
        \node(delta) at ($(2t)!0.5!(5)$) {$\delta$};
        \node(eps) at ($(3)!0.5!(1l)$) {$\epsilon$};
        \node(zeta) at ($(3)!0.5!(4)$) {$\zeta$};
        \node(eta) at ($(4)!0.5!(2b)$) {$\eta$};
        \node(theta) at ($(4)!0.5!(5)$) {$\theta$};
        \node(iota) at ($(5)!0.5!(1r)$) {$\iota$};
        \node(kappa) at ($(5)!0.5!(6)$) {$\kappa$};
        \node(lambda) at ($(6)!0.5!(2t)$) {$\lambda$};
        \node(mu) at ($(6)!0.5!(3)$) {$\mu$};
        
        \draw[->,dotted] (6) -- (lambda) -- (2t);
        \draw[->,dotted] (2t) -- (delta) -- (5);
        \draw[->,dotted] (5) -- (kappa) -- (6);
        \draw[->,dotted] (5) -- (iota) -- (1r);
        \draw[->,dotted] (1r) -- (alpha) -- (4);
        \draw[->,dotted] (4) -- (theta) -- (5);
        \draw[->,dotted] (4) -- (eta) -- (2b);
        \draw[->,dotted] (2b) -- (gamma) -- (3);
        \draw[->,dotted] (3) -- (zeta) -- (4);
        \draw[->,dotted] (3) -- (eps) -- (1l);
        \draw[->,dotted] (1l) -- (beta) -- (6);
        \draw[->,dotted] (6) -- (mu) -- (3);
        
        \node at (1l) [left = 1mm of 1l] {1};
        \node at (1r) [right = 1mm of 1r] {1};
        \node at (2t) [above = 1mm of 2t] {2};
        \node at (2b) [below = 1mm of 2b] {2};
        \node at (3) [left = 1mm of 3] {3};
        \node at (4) [below = 1mm of 4] {4};
        \node at (5) [right = 1mm of 5] {5};
        \node at (6) [above = 1mm of 6] {6};
        
        \node at (2.2,2.2) {$A$};
        \node at (-2.2,2.2) {$A$};
        \node at (2.2,-2.2) {$A$};
        \node at (-2.2,-2.2) {$A$};
        \node at (0,0.4) {$B$};
        
        \node at (0,0) {$\bullet$};
        \node at (2,2) {$\bullet$};
        \node at (-2,2) {$\bullet$};
        \node at (2,-2) {$\bullet$};
        \node at (-2,-2) {$\bullet$};
        
    \end{tikzpicture}
    \]
\end{figure}
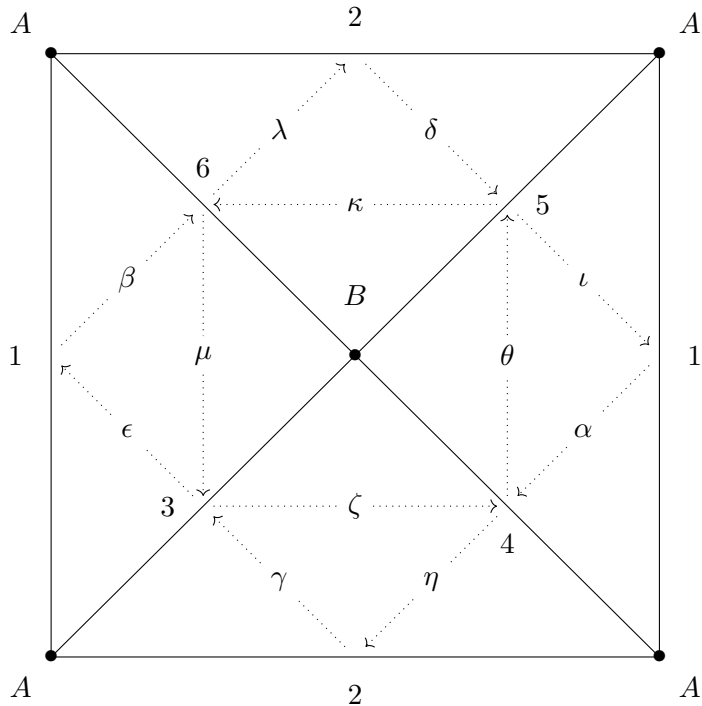

\begin{figure}
    \caption{The quiver $Q$ of the algebra $\Lambda$ considered in Section~\ref{sect:oct}}
    \label{fig:oct}
\[
\begin{tikzpicture}[scale=1.5]
    
    \node(1) at (0,3) {1};
    \node(2) at (0,-3) {2};
    \node(3) at (-3,0) {3};
    \node(5) at (3,0) {5};
    \node(4) at (-1,-1) {4};
    \node(6) at (1,1) {6};
    
    \node(alpha) at ($(1)!0.5!(4)$) {$\alpha$};
    \node(beta) at ($(1)!0.5!(6)$) {$\beta$};
    \node(gamma) at ($(2)!0.5!(3)$) {$\gamma$};
    \node(delta) at ($(2)!0.5!(5)$) {$\delta$};
    \node(eps) at ($(3)!0.5!(1)$) {$\epsilon$};
    \node(zeta) at ($(3)!0.5!(4)$) {$\zeta$};
    \node(eta) at ($(4)!0.5!(2)$) {$\eta$};
    \node(theta) at ($(4)!0.5!(5)$) {$\theta$};
    \node(iota) at ($(5)!0.5!(1)$) {$\iota$};
    \node(kappa) at ($(5)!0.5!(6)$) {$\kappa$};
    \node(lambda) at ($(6)!0.5!(2)$) {$\lambda$};
    \node(mu) at ($(6)!0.5!(3)$) {$\mu$};
    
    \draw[->] (1) -- (beta) -- (6);
    \draw[->] (1) -- (alpha) -- (4);
    \draw[->] (3) -- (eps) -- (1);
    \draw[->] (5) -- (iota) -- (1);
    \draw[->] (2) -- (gamma) -- (3);
    \draw[->] (2) -- (delta) -- (5);
    \draw[->] (4) -- (eta) -- (2);
    \draw[->] (6) -- (lambda) -- (2);
    \draw[->] (3) -- (zeta) -- (4);
    \draw[->] (4) -- (theta) -- (5);
    \draw[->] (5) -- (kappa) -- (6);
    \draw[->] (6) -- (mu) -- (3);
    
\end{tikzpicture}
\]
\end{figure}
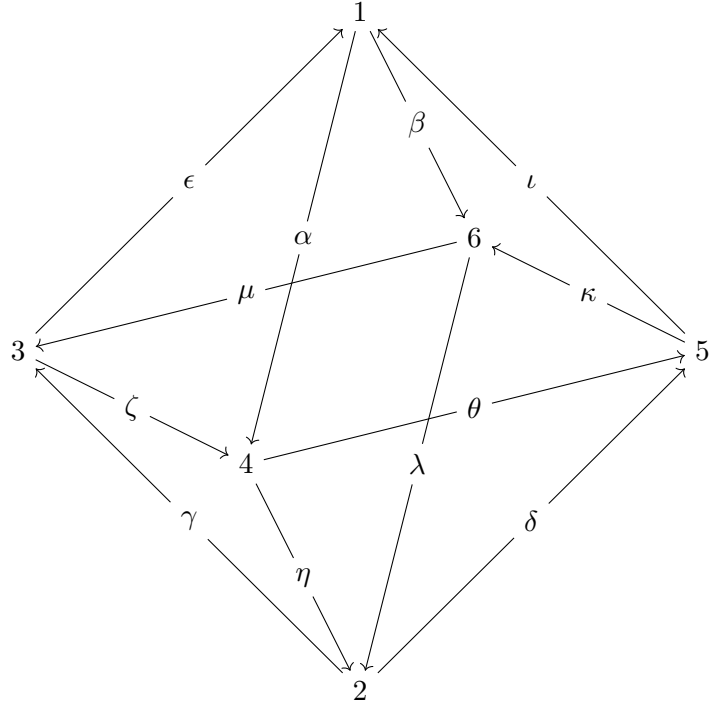

Using the work of \cite[Example~1]{ky_clus}, two maximal green sequences of this algebra are \begin{equation}\label{eq:mgs1}
5, 6, 4, 3, 6, 5, 1, 2, 4, 3, 6, 5, 3, 4, 2, 1    
\end{equation}
and
\begin{equation}\label{eq:mgs2}
1, 2, 3, 4, 6, 5, 4, 3, 2, 1, 6, 5, 4, 3, 5, 6,  
\end{equation}
given as sequences of vertices for mutation, as explained in Section~\ref{sect:back:silt_notion}. It is shown in \cite{ky_clus}, using arguments from \cite{fst}, that these two maximal green sequences cannot be deformed into each other across squares and pentagons. In fact, these are the only types of polygon in this case. Hence the poset $\dleq$ of maximal green sequences for $\Lambda$ has at least two connected components.

The intuitive reason why these two maximal green sequences cannot be deformed into each other across squares and pentagons is as follows. The quiver $Q$ is of infinite cluster type, as can be seen from the fact that it contains subquivers of affine type $\widetilde{D}_{4}$. Moreover, the poset of two-term silting complexes of this algebra is therefore also infinite. One cannot deform one maximal green sequence into the other, since doing so would require traversing regions of the exchange graph which are infinite, and, naturally, this cannot be done.

The proof of that these two maximal green sequences cannot be deformed into each other across squares and larger polygons uses theory from tagged triangulations. Generalising \cite{lf2}, by \cite{dominguez}, tagged triangulations of the twice-punctured surface are in bijection with cluster-tilting objects in the associated generalised cluster category of $\Lambda$ which are connected to the initial cluster-tilting object corresponding to $\Lambda$ via mutation. See \cite[Section~3.4]{amiot-survey} for a summary. It then follows from \cite[Theorem~4.7]{air} that cluster-tilting objects in the generalised cluster category of $\Lambda$ are in bijection with two-term silting complexes over~$\Lambda$. These bijections are moreover induced by a bijection between tagged arcs, indecomposable rigid objects in the generalised cluster category, and indecomposable two-term presilting complexes. In the last two cases, the indecomposable objects and complexes must respectively be summands of cluster-tilting objects and two-term silting complexes connected to $\Lambda$, respectively to the corresponding cluster-tilting object, by mutation.

We briefly outline some of the theory of tagged triangulations. Arcs in a tagged triangulation may be ``notched'' at either end of the arc, or notched at both ends, or simply plain. For a tagget triangulation $\mathcal{T}$ of the twice-punctured torus, there is an associated \emph{signature} $\delta_{\mathcal{T}}\colon \{A, B\} \to \{-, 0, +\}$ on the set of punctures, where \[
\delta_{\mathcal{T}}(X) = 
\left\{
	\begin{array}{ll}
		+, & \quad \text{if all arc-ends incident to $X$ are plain} \\
		-, & \quad \text{if all arc-ends incident to $X$ are notched} \\
		0. & \quad \text{if there are both plain and notched arc-ends}
	\end{array}
\right.
\] Tagged triangulations of the twice-punctured torus fall into eight disjoint sets, known as `strata', according to the signatures at the two punctures. These eight strata are denoted \[\Omega_{++},\, \Omega_{+0},\, \Omega_{0+},\, \Omega_{+-},\, \Omega_{-+},\, \Omega_{0-},\, \Omega_{-0},\, \Omega_{--},\] where triangulations $\mathcal{T} \in \Omega_{xy}$ have $\delta_{\mathcal{T}}(A) = x$ and $\delta_{\mathcal{T}}(B) = y$. Dividing up the associated two-term silting complexes into these strata yields the depiction of the poset $\twosilt \Lambda$ shown in Figure~\ref{fig:strata} --- see \cite[Remark~9.19]{fst}. Here $\Omega_{xy}$ represents the the subposet corresponding to the relevant stratum. 

\begin{figure}
    \caption{The poset of two-term silting objects of $\Lambda$ divided into strata}
    \label{fig:strata}
    \[
    \begin{tikzpicture}[yscale=1.2]
        \node(++) at (0,0) {$\Omega_{++}$};
        \node(+0) at (-2,-1) {$\Omega_{+0}$};
        \node(0+) at (2,-1) {$\Omega_{0+}$};
        \node(+-) at (-2,-2) {$\Omega_{+-}$};
        \node(-+) at (2,-2) {$\Omega_{-+}$};
        \node(0-) at (-2,-3) {$\Omega_{0-}$};
        \node(-0) at (2,-3) {$\Omega_{-0}$};
        \node(--) at (0,-4) {$\Omega_{--}$};
        
        \draw[->] (++) -- (+0);
        \draw[->] (++) -- (0+);
        \draw[->] (+0) -- (+-);
        \draw[->] (0+) -- (-+);
        \draw[->] (+-) -- (0-);
        \draw[->] (-+) -- (-0);
        \draw[->] (0-) -- (--);
        \draw[->] (-0) -- (--);
    \end{tikzpicture}
    \]
\end{figure}
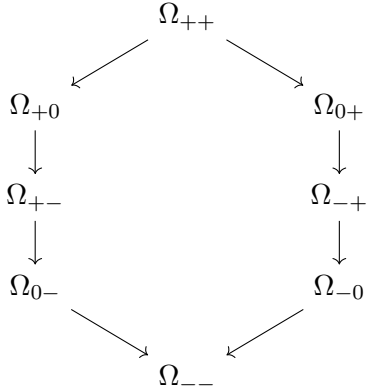

Hence, the two-term silting complex of projectives lies in the stratum $\Omega_{++}$, whilst the two-term silting complex of shifted projectives lies in the stratum $\Omega_{--}$. It is clear from Figure~\ref{fig:strata} that there are only two possible routes for a maximal green sequence of $\Lambda$ through the strata. The maximal green sequence \eqref{eq:mgs1} takes the left-hand route, whilst the maximal green sequence \eqref{eq:mgs2} takes the right-hand route \cite[Example~1]{ky_clus}. Bear in mind that the strata $\Omega_{xy}$ have cardinality larger than $1$: each of these maximal green sequences has length $16$, while each route changes strata $4$ times. In particular, one should not think of Figure~\ref{fig:strata} as depicting an unoriented octagon from Subsection~\ref{ssec:polygons}.

If there were a maximal green sequence taking one route through the strata which contained all the summands of a maximal green sequence taking the other route through the strata, then we would have an immediate counter-example to Conjecture~\ref{conj}, since these maximal green sequences could not be connected by deformations across squares and oriented polygons, as we have explained. However, we now show that two such maximal green sequences cannot exist, meaning that there is no such apparent counter-example.

\begin{proposition}
Let $\mathcal{G}$ be a maximal green sequence of $\Lambda$ passing through the stratum $\Omega_{+-}$ and $\mathcal{G}'$ a maximal green sequence of $\Lambda$ passing through the stratum $\Omega_{-+}$. Then we have $[\mathcal{G}] \not\sleq [\mathcal{G}']$ and $[\mathcal{G}'] \not\sleq [\mathcal{G}]$.
\end{proposition}
\begin{proof}
In every triangulation of the twice-punctured torus, there must be a tagged arc connecting $A$ and $B$. In the stratum $\Omega_{+-}$, this arc must be plain at $A$ and notched at $B$. Such an arc, however, cannot exist in any of the strata in the right-hand path through the maximal green sequences in Figure~\ref{fig:strata}, since its presence in the triangulation implies that the signature at $A$ is $+$ or $0$ and the signature at $B$ is $0$ or $-$. Therefore, there exists a tagged arc $\gamma$ which is in the sequence of triangulations corresponding to $\mathcal{G}$ but not in the sequence of triangulations corresponding to $\mathcal{G}'$. This then corresponds to an indecomposable two-term presilting complex $X$ such that $X \in \summ{\mathcal{G}}$ but $X \notin \summ{\mathcal{G}'}$. Hence, we obtain that $[\mathcal{G}'] \not\sleq [\mathcal{G}]$, and the converse can be shown symmetrically.
\end{proof}

Hence, the arcs which are tagged at one puncture but not another are special in some sense, because their presence forces a maximal green sequence to be of one type rather than another. It would be interesting to know whether the corresponding indecomposable complexes have particular homological properties giving algebraic reasons for the two different types of maximal green sequences of $\Lambda$.

\begin{corollary}
The poset of equivalence classes of maximal green sequences of $\Lambda$ with respect to $\sleq$ also has at least two connected components.
\end{corollary}

\begin{remark}
Note that it is itself unremarkable for these posets to have several connected components, since unoriented polygons may prevent connectedness. For instance, one can compute that the algebra from \cite[Example~3.30]{bst} has three connected components in its poset of equivalence classes of maximal green sequences. What is interesting about the example considered here is firstly that the disconnectedness does not result from unoriented polygons, secondly that the fundamental group of the poset $\twosilt \Lambda$ is not generated by polygons (in the sense of Definition~\ref{def:polygons}), and thirdly how the disconnectedness of the posets of equivalence classes of maximal green sequences can be understood in terms of the punctured surface.
\end{remark}

\section{Nakayama algebras}\label{sect:nak}

Recall that a finite-dimensional $K$-algebra $\Lambda$ is a \emph{Nakayama algebra} if every finite-dimensional indecomposable $\Lambda$-module is \emph{uniserial}, meaning that it has a unique composition series. In this section, we show that equivalence of maximal green sequences for Nakayama algebras is given by having the same sets of bricks. This makes it possible to define a partial order on equivalence classes of maximal green sequences in terms of reverse-inclusion of bricks. This in turn allows us to prove Conjecture~\ref{conj} in this special case.

\subsection{Equivalence using bricks}

For Nakayama algebras, using the following lemma, we can show that the bricks of a maximal green sequence do determine its equivalence class, for which we need the following series of lemmas. We first show that filtrations of indecomposables by any bricks are unique, not just filtrations by simples.

\begin{lemma}\label{lem:nak_unique_filt}
Given an indecomposable module $M$ over a Nakayama algebra $\Lambda$ such that $M \in \Filt(B_1, B_2, \dots, B_s)$ where $B_1, B_2, \dots, B_s$ are pairwise $\Hom$-orthogonal bricks, we have that the filtration of $M$ in terms of $B_1, B_2, \dots, B_s$ is unique.
\end{lemma}
\begin{proof}
We show that there is a unique $B_i$ such that $B_i$ is a submodule of $M$, whence the claim follows by induction on the length of the filtration. Note that $M/B_i$ must also be indecomposable, otherwise $M$ cannot be uniserial. Suppose that $B_i$ and $B_j$ are both submodules of $M$. Then, since $M$ is uniserial, we either have $B_i \hookrightarrow B_j$ or $B_j \hookrightarrow B_i$. But this contradicts the pairwise $\Hom$-orthogonality unless $i = j$, which is what we wanted to show.
\end{proof}

In the proof of Theorem~\ref{thm:nak_brick_order} we will need to consider not only bricks, but indecomposables which are self-extensions of a single brick. In the following few lemmas, we show how such indecomposables behave somewhat like bricks.

\begin{lemma}\label{lem:nak_hom_orth_self_ext}
Suppose that $L \in \Filt(B)$ and $N \in \Filt(B')$ are indecomposable modules, where $B$ and $B'$ are bricks over a Nakayama algebra $\Lambda$ such that $\Hom_{\Lambda}(B', B) = 0$. Then $\Hom_{\Lambda}(N, L) = 0$.
\end{lemma}
\begin{proof}
Assume that there exists a non-zero map $f\colon N \to L$. 
Choose a maximal submodule $X$ of $L$ such that $X \subset \im f$ and $X \in \Filt(B)$. If we consider the quotient $p \colon L \to L/X$, then $L/X \in \Filt(B)$. Moreover, $L/X$ must be indecomposable and so uniserial, otherwise $L$ cannot be uniserial. Since $\im (pf)$ and $B$ are both submodules of $L/X$, we have $\im (pf) \subseteq B$. Indeed, since $L/X$ is uniserial, the only alternative would be that $B \subset \im (pf)$, but in this case $X$ would not be maximal with respect to its defining properties. Since $\im f \supset X$, $\im (pf) \cong \im f / X$ is non-zero.

Now choose a minimal submodule $i \colon Y \hookrightarrow N$ such that $Y \in \Filt(B')$ but $\ker (pf) \subset Y$. Then $\coim (pfi) \cong Y/\ker (pfi) \not \cong 0$, since $\ker (pf) \subset Y$ as submodules of $N$. Moreover, $\coim(pfi)$ and $B'$ are both factor modules of $Y$, and since $Y$ must be indecomposable and hence uniserial, we must either have that $\coim (pfi) \twoheadrightarrow B'$ as a proper factor module or $B' \twoheadrightarrow \coim (pfi)$. However, the former possibility cannot hold, since then we could replace $Y$ by its submodule $Y'$ such that $Y/Y' \cong B'$, and we would still have that $\ker (pf) \subset Y'$. Thus, we have that $B' \twoheadrightarrow \coim (pfi)$.

The composition $B' \twoheadrightarrow \coim (pfi) \cong \im (pfi) \hookrightarrow \im (pf) \hookrightarrow B$ is then non-zero. This contradicts the assumption that $\Hom_{\Lambda} (B', B) = 0$. Thus, such non-zero $f$ cannot exist and $\Hom_{\Lambda}(N, L) = 0$.
\end{proof}

\begin{corollary}\label{cor:nak_hom_orth_self_ext}
If $B$ and $B'$ are bricks over a Nakayama algebra $\Lambda$ with $\Hom_{\Lambda}(B', B) = 0$, then $\Filt(B) \cap \Filt(B') = \{0\}$.
\end{corollary}
\begin{proof}
This follows from letting $L \in \Filt(B) \cap \Filt(B')$ and applying Lemma~\ref{lem:nak_hom_orth_self_ext} with $N = L$.
\end{proof}

A similar argument to Lemma~\ref{lem:nak_hom_orth_self_ext} also shows the following lemma.

\begin{lemma}\label{lem:nak_self_ext}
Suppose that $M \in \Filt(B)$ is an indecomposable module over a Nakayama algebra $\Lambda$, where $B$ is a brick. Then every endomorphism $f \colon M \to M$ is either
\begin{enumerate}[label=\textup{(}\arabic*\textup{)}]
    \item zero,
    \item an isomorphism, or
    \item has $\im f \in \Filt(B)$ being a proper submodule of $M$, equivalently, has $\ker f \in \Filt(B)$ being a proper submodule of $M$. 
\end{enumerate}
\end{lemma}

The following lemma can be seen as a stronger version of Lemma~\ref{lem:extension_of_bricks} which holds for Nakayama algebras.

\begin{lemma}\label{lem:nak_extension_of_bricks_new}
Suppose that $L \in \Filt(B)$ and $N \in \Filt(B')$ are indecomposable modules, where $B$ and $B'$ are bricks over a Nakayama algebra $\Lambda$ such that $\Hom_{\Lambda}(B', B) = 0$. Then every indecomposable $M$ which is a non-split extension \[0 \to L \xrightarrow{f} M \xrightarrow{g} N \to 0\] of $L$ and $N$ is a brick.
\end{lemma}
\begin{proof}
We suppose that $h \colon M \to M$ is an endomorphism of $M$ and seek to show that this must either be zero or an isomorphism. Since $\Lambda$ is a Nakayama algebra, we are in one of the following two cases:
\begin{enumerate}
    \item $L \subseteq \ker h$;\label{case:im_l_ker}
    \item $\ker h \subset L$.\label{case:ker_l}
\end{enumerate}
Case~\eqref{case:im_l_ker} is similar in flavour to Lemma~\ref{lem:extension_of_bricks}, but we will refrain from copying out the commutative diagrams once more. We have that $hf = 0$, so that there is an induced map $a \colon N \to M$ such that $h = ag$.
Since $N \in \Filt(B')$, by Lemma~\ref{lem:nak_self_ext}, its endomorphim $ga$ is either an isomorphism, zero, or there is a proper submodule $X = \im (ga) \in \Filt(B')$ of $N$ which is also a proper factor module of $N$. 
If $ga$ is an isomorphism, then $a(ga)^{-1}$ gives a splitting of $g$, which contradicts the assumption that the sequence is non-split. If $ga = 0$, then there is an induced map $b \colon N \to L$ such that $fb = a$. But, since $\Hom_{\Lambda}(N, L) = 0$ by Lemma~\ref{lem:nak_hom_orth_self_ext}, we have that $b = 0$, and so $h = ag = fbg = 0$.

We are left with the last option, where there is a proper submodule $X \in \Filt(B')$ which is also a proper factor module. First, note that $\im a = \im h$, since $h = ag$ and $g$ is epic. Then, we have that $X = \im (ga) = g(\im a) = g(\im h) \cong \im h/L$. Secondly, note that there exists a submodule $m \colon M' \hookrightarrow M$ such that $M/M' \cong X$, as $X$ is a factor module of $N$ and so also a factor module of $M$. Then $N' := M'/L \in \Filt(B')$ is a proper  non-zero submodule $l \colon N' \hookrightarrow N$ and, moreover, $N' \cong \ker (ga)$ as $X = \im (ga) \cong \coim (ga)$. Since $N$ is uniserial, we actually have that $N'$ is indecomposable and $N' = \ker (ga)$. We denote the map $M' \twoheadrightarrow N'$ by $g'$. We have a map $a' \colon N' \to M$ defined by $a' = al$. Now we have that $ga' = gal = 0$,
which gives an induced map $b' \colon N' \to L$ such that $fb' = a'$.
\[
\begin{tikzcd}
    & M' \ar[d,hookrightarrow,"m"] \ar[r,"g'"] & N' \ar[ddll,dotted,"b'",swap] \ar[ddl,dotted,"a'",swap] \ar[d,hookrightarrow,"l"] \\
    & M \ar[d,"h"] \ar[r,"g",swap] & N \ar[dl,"a"] \\
    L \ar[r,"f",hookrightarrow] & M\ar[r,"g",twoheadrightarrow] & N
\end{tikzcd}
\]
Then $b' = 0$, by Lemma~\ref{lem:nak_hom_orth_self_ext}, which gives that $hm = agm = alg' = a'g' = fb'g' = 0$. This then implies that $M' \subseteq \ker h$, and so that $\im h \cong M/\ker h \cong (M/M')/(\ker h/M') \cong X/(\ker h/M')$ by the third isomorphism theorem. But $\im h$ cannot be a quotient of $X$, since $X \cong \im h/L$ and $L \neq 0$. This is a contradiction, and so the last option cannot hold.

In case~\eqref{case:ker_l}, we have that $L/\ker h$ is a non-zero factor module of $L$. If $\ker h = 0$, then $h$ is an isomorphism, so we can ignore this case and assume that $L/\ker h$ is a proper factor module of $L$. We have an induced monomorphism $L/\ker h \hookrightarrow M/\ker h \cong \im h \subset M$. Thus, $L/\ker h$ is a submodule of~$M$. Since $M$ is uniserial, we have $L/\ker h \subset L$ or $L \subseteq L/\ker h$. By length considerations, only the former is possible. Thus, we have an endomorphism $L \twoheadrightarrow L/\ker h \hookrightarrow L$ which is certainly not an isomorphism. Since we assume that $\ker h \neq 0$, we must have that $\ker h \in \Filt(B)$ by Lemma~\ref{lem:nak_self_ext}.

We show that having $\ker h \in \Filt(B)$ leads to a contradiction. By the third isomorphism theorem, we have \[N \cong \frac{M}{L} \cong \frac{M/\ker h}{L/\ker h} \cong \frac{\im h}{L/\ker h} \hookrightarrow \frac{M}{L/\ker h},\] so that $N$ is a submodule of $M/(L/\ker h)$. We also have that $L/(L/\ker h)$ is a submodule of $M/(L/\ker h)$. Since $M/(L/\ker h)$ is uniserial as a factor module of a uniserial module $M$, we must either have that $N \subseteq L/(L/\ker h)$ or $L/(L/\ker h) \subset N$ as submodules of this module. Having $N \subseteq L/(L/\ker h)$ contradicts Lemma~\ref{lem:nak_hom_orth_self_ext}, since $L/(L/\ker h) \in \Filt(B)$. Hence, we must have $L/(L/\ker h) \subset N$. Having $L/(L/\ker h) \in \Filt(B')$ contradicts Corollary~\ref{cor:nak_hom_orth_self_ext}, so $L/(L/\ker h) \notin \Filt(B')$. But then we apply the third isomorphism theorem again to obtain \[\frac{N}{L/(L/ \ker h)} \hookrightarrow \frac{M/(L/\ker h)}{L/(L/\ker h)} \cong \frac{M}{L} \cong N.\] We therefore have a composition of maps \[N \twoheadrightarrow \frac{N}{L/(L/\ker h)} \hookrightarrow N\] which is neither zero nor an isomorphism, with $N/(L/(L/\ker h)) \notin \Filt(B')$, which contradicts Lemma~\ref{lem:nak_self_ext}.

We conclude that $M$ is a brick.
\end{proof}

We can now show the key lemma which establishes that bricks determine the equivalence class of a maximal green sequence over a Nakayama algebra.

\begin{lemma}\label{lem:nak_ses}
Let $\mathcal{G}$ be a maximal green sequence of a Nakayama algebra~$\Lambda$. Suppose that $B_{i}$ and $B_{k}$ are bricks of $\mathcal{G}$ such that $i < k$ in the sequence of bricks and that there is a non-split short exact sequence \[0 \to B_{i} \xrightarrow{f} B \xrightarrow{g} B_{k} \to 0\] 
such that $B$ is a brick. Then $B \cong B_{j}$ is a brick of $\mathcal{G}$ with $i < j < k$.
\end{lemma}
\begin{proof}
We use the fact that the brick sequence of $\mathcal{G}$ must be maximal with respect to the property of being backwards $\Hom$-orthogonal. Hence, if we can show that there is a space between $B_i$ and $B_k$ where one can insert $B$ without disrupting the backwards $\Hom$-orthogonality, then it follows that $B$ must actually occur in the brick sequence of $\mathcal{G}$ as $B \cong B_{j}$ for $i < j < k$.

Let us insert $B$ in the brick sequence of $\mathcal{G}$ at some point between $B_{i}$ and $B_{k}$. If $B$ cannot be inserted at such a point, then there must exist either a brick $L$ in $\mathcal{G}$ which occurs after $B$ with a non-zero homomorphism $l\colon L \to B$, or there exists a brick $N$ before $B$ with a non-zero homomorphism $n\colon B \to N$. In the first case, if $L$ occurs after~$B_k$, then the composition $L \xrightarrow{l} B \xrightarrow{g} B_k$ must be zero. This gives a non-zero map $L \to B_i$, which is a contradiction. One can argue in a similar way that $N$ cannot occur before $B_i$.

Hence, we have that all such bricks $N$ and $L$ occur between $B_i$ and~$B_k$. If all such bricks $L$ occur before all such bricks $N$, then we can place $B$ between these two sets of bricks, thereby preserving backwards $\Hom$-orthogonality. Thus, we may assume that $N$ occurs before $L$.

We now use the fact that $\Lambda$ is a Nakayama algebra and so that all the modules in question must be uniserial. We have that $\im l \supset B_i$, otherwise there is a non-zero map $L \twoheadrightarrow \im l \hookrightarrow B_i$. Likewise, $\ker n \subset \ker g = B_i$, otherwise there is a non-zero map $B_k \to N$. Thus, we have that $\im l \supset \ker n$.
This means that the composition $L \xrightarrow{l} B \xrightarrow{n} N$ is non-zero, which contradicts the backwards $\Hom$-orthogonality. Hence, all of the bricks $L$ occur before all the bricks $N$ between $B_i$ and $B_k$, and so there is a point between $B_i$ and $B_k$ where $B$ can be placed without violating backwards $\Hom$-orthogonality, as desired. 
\end{proof}

\begin{remark}
Note that Lemma~\ref{lem:nak_ses} does not hold for the non-linearly oriented type~$A$ algebra from Example~\ref{ex:brick_counter}, which is not a Nakayama algebra. Indeed, considering the maximal green sequence \[\tcs{2},\, \tcs{1\\2},\, \tcs{1},\, \tcs{3\\2},\, \tcs{3},\] we have a short exact sequence \[0 \to \tcs{1\\2} \to \tcs{1\phantom{2}3\\\phantom{1}2\phantom{3}} \to \tcs{3} \to 0,\] but the non-uniserial brick $\tcss{1\phantom{2}3\\\phantom{1}2\phantom{3}}$ does not occur in the maximal green sequence.
\end{remark}

We can then make the following argument showing that maximal green sequences of Nakayama algebras with the same bricks are connected by deformations across squares.

\begin{theorem}\label{thm:nak_brick_equiv}
Let $\mathcal{G}_1$ and $\mathcal{G}_2$ be maximal green sequences of a Nakayama algebra $\Lambda$ such that $\bricks{\mathcal{G}_1} = \bricks{\mathcal{G}_2}$. Then we have that $\mathcal{G}_1 \sim \mathcal{G}_2$.
\end{theorem}
\begin{proof}
Suppose that we have maximal green sequences $\mathcal{G}_1$ and $\mathcal{G}_2$ such that $\bricks{\mathcal{G}_1} = \bricks{\mathcal{G}_2}$. We will prove that $\mathcal{G}_1$ and $\mathcal{G}_2$ can be deformed into each other across squares by induction on the number of bricks after the first brick where they differ. In the base case, the maximal green sequences coincide, and so there is no point at which they diverge.

Now we suppose that $\mathcal{G}_1 \neq \mathcal{G}_2$. We consider the first bricks where $\mathcal{G}_1$ and $\mathcal{G}_2$ differ. Let this brick be $B_1$ for $\mathcal{G}_1$ and $B_2$ for $\mathcal{G}_2$. Since $\bricks{\mathcal{G}_1} = \bricks{\mathcal{G}_2}$, we must therefore have $B_1 \in \bricks{\mathcal{G}_2}$ and $B_2 \in \bricks{\mathcal{G}_1}$ as well.

We claim that, by deforming across squares, we can make $B_2$ the brick immediately before $B_1$ in $\mathcal{G}_1$. Suppose that, on the contrary, there is a brick $B \in \bricks{\mathcal{G}_1}$ such that $B_2$ cannot be moved back past $B$. By Lemma~\ref{lem:sq_crit_brick}, this must either be because $\Hom_{\Lambda}(B, B_2) \neq 0$ or because $\Ext_{\Lambda}^{1}(B, B_2) \neq 0$. We must have $B \in \bricks{\mathcal{G}_2}$. Furthermore, $B$ must occur in $\mathcal{G}_2$ after $B_2$, since it coincides with $B_1$ or occurs in $\mathcal{G}_1$ after $B_1$, and the respective segments of $\mathcal{G}_1$ and $\mathcal{G}_2$ before $B_1$ and $B_2$ coincide. We then cannot have $\Hom_{\Lambda}(B, B_2) \neq 0$, since $B$ occurs after $B_2$ in $\mathcal{G}_2$. Suppose, then, that $\Ext_{\Lambda}^{1}(B, B_2) \neq 0$. Since $B$ and $B_2$ appear in different orders in $\mathcal{G}_1$ and $\mathcal{G}_2$, we have $\Hom_{\Lambda}(B_2, B) = \Hom_{\Lambda}(B, B_2) = 0$. Then, by Lemma~\ref{lem:extension_of_bricks}, there is a brick $B'$ which is a non-split extension of $B$ and $B_2$, \[0 \to B_2 \to B' \to B \to 0.\] Because $B_2$ occurs before $B$ in $\mathcal{G}_2$, we must have that $B' \in \bricks{\mathcal{G}_2}$ by Lemma~\ref{lem:nak_ses}. Then we also have that $B' \in \bricks{\mathcal{G}_1}$. Due to backwards $\Hom$-orthogonality, we must have that $B'$ occurs after $B_2$ and before $B$ in $\mathcal{G}_1$. However, this is a contradiction, since $B_2$ occurs after $B$ in $\mathcal{G}_1$.

Thus, we can move $B_2$ back along $\mathcal{G}_1$ by deforming across squares to obtain a maximal green sequence $\mathcal{G}'_1$ where $B_2$ occurs before $B_1$. The maximal green sequences $\mathcal{G}'_1$ and $\mathcal{G}_2$ then have fewer bricks which occur after they differ, so by the induction hypothesis, we have that $\mathcal{G}'_1 \sim \mathcal{G}_2$. This then establishes that $\mathcal{G}_1 \sim \mathcal{G}_2$, as desired, since $\mathcal{G}_1 \sim \mathcal{G}'_1$. 
\end{proof}

\subsection{Partial order using bricks}

For Nakayama algebras we may also consider the partial order on maximal green sequences given by reverse inclusion of bricks, since maximal green sequences with the same sets of bricks are equivalent by Theorem~\ref{thm:nak_brick_equiv}. For general algebras this will not be a well-defined partial order; the relation will not in general be anti-symmetric, since non-equivalent maximal green sequences may have the same set of bricks.

\begin{definition}
Let $[\mathcal{G}]$ and $[\mathcal{G}']$ be equivalence classes of maximal green sequences over a Nakayama algebra $\Lambda$. We write $[\mathcal{G}] \bleq [\mathcal{G}']$ if $\bricks{\mathcal{G}} \supseteq \bricks{\mathcal{G}'}$. We refer to this as the \emph{brick order}.
\end{definition}

\begin{lemma}
The partial order $\bleq$ is well-defined for Nakayama algebras.
\end{lemma}
\begin{proof}
We first note $\bleq$ is well-defined on equivalence classes, since $\bricks{\mathcal{G}} = \bricks{\widehat{\mathcal{G}}}$ for all $\widehat{\mathcal{G}} \in [\mathcal{G}]$ by Lemma~\ref{lem:same_bricks}. The relation $\bleq$ is clearly transitive and reflexive. It is anti-symmetric since $\bricks{\mathcal{G}} = \bricks{\mathcal{G}'}$ implies $[\mathcal{G}] = [\mathcal{G}']$ by Theorem~\ref{thm:nak_brick_equiv}.
\end{proof}

We can show that the brick order coincides with the deformation order for Nakayama algebras, which is the key to showing that in fact all of the orders coincide in this case.

\begin{theorem}\label{thm:nak_brick_order}
Let $\Lambda$ be a finite-dimensional Nakayama algebra over a field $K$, with $\mathcal{G}$ and $\mathcal{G}'$ two maximal green sequences of $\Lambda$. Then $[\mathcal{G}] \bleq [\mathcal{G}']$ if and only if $[\mathcal{G}] \dleq [\mathcal{G}']$.
\end{theorem}
\begin{proof}
We know that $[\mathcal{G}] \dleq [\mathcal{G}']$ implies
$[\mathcal{G}] \bleq [\mathcal{G}']$ by Corollary~\ref{cor:def->bricks}. 
We now show that $[\mathcal{G}] \bleq [\mathcal{G}']$ implies $[\mathcal{G}] \dleq [\mathcal{G}']$. Let $\mathcal{G}$ and $\mathcal{G}'$ be maximal green sequences such that $[\mathcal{G}] \bl [\mathcal{G}']$. Hence, every brick of $\mathcal{G}'$ is also a brick of $\mathcal{G}$.

By Lemma~\ref{lem:adj_bricks}, there must be an adjacent pair of bricks $B$ and $B'$ of $\mathcal{G}'$ such that $B$ occurs before $B'$ in $\mathcal{G}'$, but $B$ occurs after $B'$ in $\mathcal{G}$. Let $\mathcal{S}$ be the sequence of bricks given by starting with $\mathcal{G}'$ and swapping $B$ and $B'$ so that $B'$ now occurs before $B$. This sequence $\mathcal{S}$ must in fact be backwards $\Hom$-orthogonal, since $B'$ occurs before $B$ in $\mathcal{G}$. If $\mathcal{S}$ is maximal backwards $\Hom$-orthogonal, then we have deformed across a square and we apply Lemma~\ref{lem:adj_bricks} again to find a new pair of bricks. Hence, we can assume that $\mathcal{S}$ is not maximal backwards $\Hom$-orthogonal.

Therefore, we can add bricks to $\mathcal{S}$ to obtain a maximal backwards $\Hom$-orthogonal sequence $\mathcal{G}''$. By Lemma~\ref{lem:iepd_unique}, $\mathcal{G}''$ is unique. Since $\mathcal{G}'$ was maximal backwards $\Hom$-orthogonal, the extra bricks in $\mathcal{G}''$ can only occur between $B$ and~$B'$. We have that all the extra bricks lie in $\Filt(B, B')$ by Theorem~\ref{thm:filt_intervals}.

We claim that, by applying Lemma~\ref{lem:nak_ses} iteratively, we obtain that all the extra bricks of $\mathcal{G}''$ must also be bricks of $\mathcal{G}$ lying in between $B'$ and $B$, and so there can only be finitely many of them. Let $B''$ be an extra brick of $\mathcal{G}''$, which thus has a unique filtration with factors $B$ and $B'$ by Lemma~\ref{lem:nak_unique_filt}. Note then that we must have $\Ext_{\Lambda}^{1}(B', B) = 0$ by Lemma~\ref{lem:brick_no_ext}, since $B$ occurs immediately before $B'$ in $\mathcal{G}'$ and $\Hom(B, B') = 0$ by the the backwards $\Hom$-orthogonality of $\mathcal{G}$. Hence any subquotient of $B''$ with factors $B$ and $B'$ where $B$ occurred below $B'$ would have to be isomorphic to $B \oplus B'$, with the result that the order of the factors could be switched so that $B'$ occurs below $B$. This would contradict uniqueness, so in fact we must have that in the filtration of $B''$ by $B$ and $B'$, all of the $B'$ factors occur below all of the $B$ factors.

We now prove the claim that $B'' \in \bricks{\mathcal{G}}$ and lies in between $B'$ and $B$ by induction on the length of the filtration of $B''$ by $B$ and $B'$. In the base case, where this filtration is of length 1, we either have $B'' \cong B$ or $B'' \cong B'$, and in either case the claim is immediate. We also use the case where $B''$ has only one $B$ factor and only one $B'$ factor as a base case. Since $B'$ occurs in $\mathcal{G}$ before $B$, we then have that $B''$ occurs in $\mathcal{G}$ between $B'$ and $B$ by Lemma~\ref{lem:nak_ses}. For the induction step there are then two cases to consider, namely, either
\begin{itemize}
    \item $B''$ has at least two $B$ factors, or
    \item $B''$ has at least two $B'$ factors.
\end{itemize}
The two cases are similar, so we only consider the first one. First note that $B''$ must also have $B'$ factors in its filtration, otherwise it is clearly not a brick. We claim that if $L$ is the unique submodule of $B''$ such that $B''/L \cong B$, then we have that $L$ is a brick. Indeed, we have that there is a non-split short exact sequence $0 \to X \to L \to Y \to 0$ such that $X \in \Filt(B')$ and $Y \in \Filt(B)$. Also, $L$ is indecomposable, since otherwise $B''$ would not be uniserial. Hence, $L$ is a brick by Lemma~\ref{lem:nak_extension_of_bricks_new}.

Hence, by the induction hypothesis, we have that $L$ is a brick of $\mathcal{G}$ lying between $B'$ and $B$. Then Lemma~\ref{lem:nak_ses} establishes that $B''$ is a brick of $\mathcal{G}$ lying between $L$ and $B$, and therefore also between $B'$ and $B$. Therefore, by induction, we have that all of the extra bricks of $\mathcal{G}''$ are bricks of $\mathcal{G}$ and lie between $B'$ and~$B$.

We thus obtain that $\mathcal{G}'$ is an increasing elementary polygonal deformation of~$\mathcal{G}''$, and that the bricks of $\mathcal{G}''$ are contained in the bricks of $\mathcal{G}$, and so $[\mathcal{G}] \bleq [\mathcal{G}''] \dlessdot [\mathcal{G}']$. Applying this argument inductively gives that $[\mathcal{G}] \dl [\mathcal{G}']$, as desired.
\end{proof}

\begin{corollary}\label{cor:nak_hn}
Let $\Lambda$ be a finite-dimensional Nakayama algebra over a field $K$, with $\mathcal{G}$ and $\mathcal{G}'$ two maximal green sequence of $\Lambda$. Then $[\mathcal{G}] \hleq [\mathcal{G}']$ if and only if $[\mathcal{G}] \dleq [\mathcal{G}']$.
\end{corollary}
\begin{proof}
We already know from Theorem~\ref{thm:def->hn} that if $[\mathcal{G}] \dleq [\mathcal{G}']$ then $[\mathcal{G}] \hleq [\mathcal{G}']$. For the other direction, we have from Lemma~\ref{lem:hn->bricks} that $[\mathcal{G}] \hleq [\mathcal{G}']$ implies $[\mathcal{G}] \bleq [\mathcal{G}']$. Then, by Theorem~\ref{thm:nak_brick_order}, we have that this implies $[\mathcal{G}] \dleq [\mathcal{G}']$.
\end{proof}

\subsection{Bricks versus summands}

We now show how one can compute the set of $\tau$-rigid summands of a maximal green sequence for a Nakayama algebra from the set of bricks, which will allow us to show that the brick order coincides with the summand order for Nakayama algebras.

We begin by letting $\Lambda$ be a Nakayama algebra, and introducing the set \[ \mathcal{B}(\Lambda) := \{B \in \brick \Lambda \st B \notin \simp \Lambda\}\] of bricks over $\Lambda$ which are not simple. We further introduce \[\mathcal{R}(\Lambda) := \{M \in \ind \Lambda \st M \text{ is $\tau$-rigid but } M \notin \proj \Lambda\}\] of indecomposable $\tau$-rigid modules which are not projective. Note that $\tau$-tilting theory for Nakayama algebras was studied in \cite{adachi_nak}.

\begin{lemma}\label{lem:nak_bs_bij}
If $\Lambda$ is a Nakayama algebra, then there is a bijection
\begin{align*}
    \phi \colon \mathcal{B}(\Lambda) &\to \mathcal{R}(\Lambda) \\
    B &\mapsto B/\soc B
\end{align*}
\end{lemma}
\begin{proof}
Note first that, since $B$ is not a simple, we have that $B/\soc B$ is non-zero. Moreover, $B/\soc B$ is not a projective due to the non-split short exact sequence \[0 \to \soc B \to B \to B/\soc B \to 0,\] whilst $B/\soc B$ must be indecomposable since $B$ is a uniserial module.

We now show that $B/\soc B$ is $\tau$-rigid. It follows from \cite[Theorem~VI.2.1]{ars_book} that $\tau(B/\soc B) = \rad B$ (see also \cite[Theorem~V.4.1]{ass}).
Thus $\Hom_{\Lambda}(B/\soc B, \tau(B/\soc B)) = \Hom_{\Lambda}(B/\soc B, \rad B) = 0$, since $B$ is a brick. Hence $B/\soc B$ is $\tau$-rigid, as desired.

We now show that $\phi$ is injective. Suppose that we have bricks $B, B' \in \mathcal{B}(\Lambda)$ such that $B/\soc B \cong B'/\soc B'$. We use the fact from \cite{nak_ii}, \cite[Theorem~VI.2.1]{ars_book}, \cite[Theorem~V.3.5]{ass} that there exist indecomposable projectives $P, P' \in \proj \Lambda$ such that $B \cong P/ \rad^t P$ and $B' \cong P' \rad^{t'} P'$. Then $B/\soc B \cong P/\rad^{t - 1}P$ and $B'/\soc B' \cong P'/\rad^{t' - 1}P'$. Hence $B/\soc B \cong B'/\soc B'$ implies that $P \cong P'$ since indecomposable projectives are determined by their top, and so we also have $t = t'$ and $B \cong B'$, as desired.

Surjectivity of the map $\phi$ then follows from the fact that $\# \mathcal{B}(\Lambda) = \# \mathcal{R}(\Lambda)$ because there is a bijection between indecomposable $\tau$-rigid modules and bricks generating functorially finite torsion classes from \cite[Theorem~4.1]{dij}, as well as a well-known bijection between simple modules and indecomposable projectives (sending a simple module to its projective cover). Note that, since Nakayama algebras are representation-finite \cite{nak_ii}, \cite[Theorem~VI.2.1]{ars_book}, \cite[Theorem~V.3.5]{ass}, we have that all bricks generate functorially finite torsion classes and that both the sets $\mathcal{B}(\Lambda)$ and $\mathcal{R}(\Lambda)$ are finite.
\end{proof}

Note that this bijection $\phi$ is specific to Nakayama algebras; it does not coincide with the bijection between indecomposable $\tau$-rigid modules and bricks generating functorially finite torsion classes for general algebras given in \cite[Theorem~4.1]{dij}. We can then use the map $\phi$ to relate the relative simples of a torsion class to the relative projectives of the torsion class.

\begin{lemma}\label{lem:nak_bs_tors}
If $B$ is a relatively simple object in a torsion class $\mathcal{T}$ for a Nakayama algebra $\Lambda$, then $\phi(B)$ is relatively projective in $\mathcal{T}$. 
\end{lemma}
\begin{proof}
Let $B$ be a relative simple in $\mathcal{T}$. To show that $\phi(B)$ is a relative projective in $\mathcal{T}$, it suffices to show that $\Hom_{\Lambda}(X, \tau \phi(B)) = 0$ for any $X \in \mathcal{T}$, by the Auslander--Reiten formula. As in Lemma~\ref{lem:nak_bs_bij}, we have that $\Hom_{\Lambda}(X, \tau \phi(B)) = \Hom_{\Lambda}(X, \rad B)$. If there were then a non-zero map $f \colon X \to \rad B$, then $\im f \in \mathcal{T}$ would be a proper submodule of $B$, which would contradict the fact that $B$ is relatively simple in $\mathcal{T}$. 
\end{proof}

Note that it is not true in general that for an arbitrary torsion class $\mathcal{T}$, $\phi$ restricts to a bijection between non-simple relatively simple objects and non-projective indecomposable relatively projective objects in $\mathcal{T}$. Nonetheless, Lemma~\ref{lem:nak_bs_tors} allows us to compute the $\tau$-rigid summands of a maximal green sequence of a Nakayama algebra from the bricks.

\begin{proposition}\label{prop:nak_bs_green}
Let $\mathcal{G}$ be a maximal green sequence of a Nakayama algebra~$\Lambda$. Then \[\phi(\bricks{\mathcal{G}}\setminus\simp\Lambda) = \sumt{\mathcal{G}}\setminus\proj\Lambda.\]
\end{proposition}
\begin{proof}
Let $B \in \bricks{\mathcal{G}}\setminus\simp\Lambda$. Then by Theorem~\ref{thm:enomoto_relsimp} $B$ is a relative simple in some torsion class $\mathcal{T}$ in $\mathcal{G}$ containing $B$. 
By Lemma~\ref{lem:nak_bs_tors}, we the obtain that $\phi(B)$ is a relative projective in $\mathcal{T}$. Recalling that $\phi(B) \notin \proj\Lambda$, we conclude that $\phi(B) \in \sumt{\mathcal{G}}\setminus\proj\Lambda$. Thus, $\phi(\bricks{\mathcal{G}}\setminus\simp\Lambda) \subseteq \sumt{\mathcal{G}}\setminus\proj\Lambda$. We then obtain the result from the fact that $\# \bricks{\mathcal{G}} = \# \sumt{\mathcal{G}}$, since both are equal to the length of the maximal green sequence.
\end{proof}

We conclude that the orders on maximal green sequences defined by bricks and summands coincide for Nakayama algebras.

\begin{theorem}\label{thm:nak_bs}
Given a Nakayama algebra $\Lambda$ and two maximal green sequences $\mathcal{G}$ and $\mathcal{G}'$ of $\Lambda$, we have that \[[\mathcal{G}] \bleq [\mathcal{G}'] \text{ if and only if } [\mathcal{G}] \sleq [\mathcal{G}'].\]
\end{theorem}
\begin{proof}
Since the simple modules and indecomposable projectives are always contained in $\bricks{\mathcal{G}}$ and $\sumt{\mathcal{G}}$ respectively, we can make the following chain of deductions.
\begin{align*}
    [\mathcal{G}] \bleq [\mathcal{G}'] &\iff \bricks{\mathcal{G}} \supseteq \bricks{\mathcal{G}'} \\
    &\iff \bricks{\mathcal{G}}\setminus\simp\Lambda \supseteq \bricks{\mathcal{G}'}\setminus\simp\Lambda \\
    &\iff \phi(\bricks{\mathcal{G}}\setminus\simp\Lambda) \supseteq \phi(\bricks{\mathcal{G}'}\setminus\simp\Lambda) \\
    &\iff \sumt{\mathcal{G}}\setminus\proj\Lambda \supseteq \sumt{\mathcal{G}'}\setminus\proj\Lambda \\
    &\iff \sumt{\mathcal{G}} \supseteq \sumt{\mathcal{G}'} \\
    &\iff [\mathcal{G}] \sleq [\mathcal{G}'].
\end{align*}
\end{proof}

\begin{corollary}\label{cor:nak_orders}
For a Nakayama algebra $\Lambda$ and two maximal green sequences $\mathcal{G}$ and $\mathcal{G}'$ of $\Lambda$, the following are equivalent.
\begin{enumerate}[label=\textup{(}\arabic*\textup{)}]
    \item $[\mathcal{G}] \dleq [\mathcal{G}']$.\label{op:nak_thm:def}
    \item $[\mathcal{G}] \sleq [\mathcal{G}']$.\label{op:nak_thm:sum}
    \item $[\mathcal{G}] \hleq [\mathcal{G}']$.\label{op:nak_thm:hn}
    \item $[\mathcal{G}] \bleq [\mathcal{G}']$.\label{op:nak_thm:brick}
\end{enumerate}
\end{corollary}
\begin{proof}
Equivalence of \ref{op:nak_thm:def} and \ref{op:nak_thm:brick} is Theorem~\ref{thm:nak_brick_order}. Equivalence of \ref{op:nak_thm:def} and \ref{op:nak_thm:hn} is Corollary~\ref{cor:nak_hn}. Finally, equivalence of \ref{op:nak_thm:sum} and \ref{op:nak_thm:brick} is Theorem~\ref{thm:nak_bs}. 
\end{proof}

\printbibliography

\end{document}